\documentclass[11pt, letterpaper,english]{amsart}
\usepackage[left=1in,right=1in,bottom=1in,top=1in]{geometry}
\usepackage[utf8]{inputenc}
\usepackage[numbers]{natbib}

\usepackage{amsmath}
\usepackage{amsthm}
\usepackage{amssymb}
\usepackage{stmaryrd}
\usepackage{colonequals}   %
\usepackage[shortlabels]{enumitem}
\usepackage{hyperref}
\usepackage{caption}

\usepackage{upgreek}

\usepackage[all,cmtip]{xy}
\usepackage{tikz} %
\usepackage{tikz-cd}

\theoremstyle{plain}
\newtheorem{theorem}{Theorem}[section]
\newtheorem{lemma}[theorem]{Lemma}
\newtheorem{proposition}[theorem]{Proposition}
\newtheorem{corollary}[theorem]{Corollary}

\theoremstyle{definition}
\newtheorem{definition}[theorem]{Definition}
\newtheorem{notation}[theorem]{Notation}
\newtheorem{example}[theorem]{Example}
\newtheorem{remark}[theorem]{Remark}
\newtheorem{convention}[theorem]{Convention}

\numberwithin{equation}{section}

\newcommand{\F}{\mathbf{F}}

\newcommand{\PP}{\mathbf{P}}
\newcommand{\Q}{\mathbf{Q}}
\newcommand{\R}{\mathbf{R}}

\newcommand{\Z}{\mathbf{Z}}

\newcommand{\frakm}{\mathfrak{m}}
\newcommand{\frakp}{\mathfrak{p}}

\newcommand{\calA}{\mathcal{A}}
\newcommand{\calB}{\mathcal{B}}

\newcommand{\calD}{\mathcal{D}}
\newcommand{\calE}{\mathcal{E}}

\newcommand{\calL}{\mathcal{L}}
\newcommand{\calM}{\mathcal{M}}
\newcommand{\calO}{\mathcal{O}}

\newcommand{\calV}{\mathcal{V}}
\newcommand{\calW}{\mathcal{W}}

\newcommand{\calZ}{\mathcal{Z}}

\newcommand{\et}{\text{\'et}}
\newcommand{\mult}{\text{m}}
\renewcommand{\ll}{\text{ll}}

\newcommand{\bfA}{A}
\newcommand{\bfB}{B}
\newcommand{\bfC}{C}
\newcommand{\bfD}{D}

\newcommand{\bfV}{V} %

\DeclareFontFamily{OT1}{rsfs}{}
\DeclareFontShape{OT1}{rsfs}{n}{it}{<-> rsfs10}{}
\DeclareMathAlphabet{\mathscr}{OT1}{rsfs}{n}{it}

\newcommand{\Gscr}{\mathscr{G}}
\newcommand{\E}{\mathscr{E}}

\DeclareMathOperator{\NP}{\mathrm{NP}}

\DeclareMathOperator{\HP}{\mathrm{HP}}

\DeclareMathOperator{\Cl}{Cl}
\DeclareMathOperator{\Char}{char}
\DeclareMathOperator{\sep}{sep}
\DeclareMathOperator{\Frac}{Frac}
\DeclareMathOperator{\ev}{ev}

\DeclareMathOperator{\Hom}{\mathrm{Hom}}

\DeclareMathOperator{\Gal}{\mathrm{Gal}}

\DeclareMathOperator{\rev}{rev}
\DeclareMathOperator{\Det}{det}

\newcommand{\Fil}[1]{\textup{Fil}^{ #1 }}
\newcommand{\oFil}[1]{\textup{Fil}_\perp^{ #1 }}
\newcommand{\filindex}{t}

\newcommand{\ord}{\operatorname{ord}}
\newcommand{\ordn}{\ord_n}
\newcommand{\Frob}{\text{Frob}}

\newcommand{\y}{\mathbf{y}}
\newcommand{\powerseries}[1]{[\![#1]\!]}
\newcommand{\Spec}{\operatorname{Spec}}

\newcommand{\fielddegree}{{\nu}}
\newcommand{\FF}{{\mathbb{F}_p}}

\newcommand{\prt}[1]{\operatorname{pr}_{#1}}

\newcommand{\klin}{\Psi}
\newcommand{\Flin}{\Theta}

\newcommand{\dnom}{\delta}

\title{Iwasawa theory of Frobenius-torsion class group schemes}

\author{Jeremy Booher, Bryden Cais, Joe Kramer-Miller, and James Upton}
\date{\today}

\begin{document}

\begin{abstract}
We establish a new Iwasawa theory 
for the kernel of Frobenius on Jacobians of curves
in geometric $\Z_p$-towers over
the projective line in characteristic $p$,
 thereby proving several of the main conjectures of
 \cite{boohercais21}.  
\end{abstract}

\maketitle

\section{Introduction}\label{sec:intro}

Let $K$ be a function field in one variable over a finite field $k$, and $p$ a prime number.
As with number fields, Iwasawa theory aims to understand the growth
of the $p$-primary part of class groups in $\Z_{p}$-extensions of $K$ unramified outside a finite set $S$ of places. %
For the constant $\Z_{p}$-extension of $K$, %
the regular growth of these $p$-Sylow subgroups
was observed by Iwasawa himself \cite{IwasawaAnalogies}, and served as a foundational motivation
for the theory he initiated \cite{Iwasawa}.  More natural from the viewpoint of arithmetic geometry 
are {\em geometric} $\Z_p$-extensions $L/K$, {\em i.e.~}those with
$k$ algebraically closed in $L$. 
If $p\neq \Char(k)$, it is a straightforward consequence of class field theory that $K$ admits
no geometric $\Z_{p}$-extensions {\em whatsoever}, and when $p=\Char(k)$---which we henceforth assume---there are no 
everywhere unramified %
geometric $\Z_p$-extensions of $K$.  
What remains
are the geometric $\Z_p$-extensions that ramify---necessarily infinitely and wildly---at one or more places,
and there are myriads of them. %
 We fix such a geometric $\Z_p$-extension, %
and for simplicity assume that every place in $S$ totally ramifies.  Write $K_n$ for the fixed field
of $p^n\Z_p \subseteq \Z_p$; %
as for number fields, the class group $\Cl_{K_n}$
of $K_n$ is a finite%
---and generally mysterious---abelian group
controlling the unramified abelian extensions of $K_n$.
In \cite{MWAnalogies}, Mazur and Wiles studied
the projective limit $M\colonequals \varprojlim \Hom(\Cl_{K_n},\Q_p/\Z_p)$ of Pontryagin duals, and proved that it is finitely generated
and torsion as a module for the Iwasawa algebra $\Lambda$.
The function field analogue of Iwasawa's celebrated formula---which gives the exact power of $p$ dividing $|\Cl_{K_n}|$---follows from this, as does the existence of a characteristic ideal for $M$ governing its structure as a $\Lambda$-module. 
Subsequently, Crew proved \cite{Crew} that this characteristic ideal is generated by a {\em Stickelberger element}, given by the special value of 
a $\Lambda$-adic $L$-function, thereby establishing a beautiful analogue of the Main Conjecture of Iwasawa theory for number fields %
and proving a conjecture of Katz \cite{KatzDwork}.
In the last twenty years, the work of Mazur–Wiles and Crew has been generalized and expanded extensively
(e.g.~\cite{BurnsZeta}, \cite{Burns}, \cite{BurnsKakde}, \cite{BurnsTrihan}, \cite{EmertonKisin}, \cite{GP1}, \cite{GP2}, \cite{GP3}, \cite{WanClass},  \cite{RLiuWan}).

Writing $X_n$ for the unique smooth, projective, and geometrically connected curve with function field $K_n$,
the %
Jacobian $J_{X_n}$ of $X_n$ provides a powerful, geometric avatar for the class group of each $K_n$
via the identification $\Cl_{K_n} = J_{X_n}(k)$. %
This %
remarkable fact---which has no analogue for number fields---yields an identification of 
$p$-groups $\Cl_{K_n}(p) = \Gscr_n(k)$, where $\Gscr_n \colonequals  J_{X_n}[p^{\infty}]$ is the $p$-divisible group of $J_{X_n}$; 
this is an inductive system of finite group {\em schemes}, and is an exceedingly rich algebro-geometric object. 
Unfortunately, the passage $\Gscr\rightsquigarrow \Gscr(k)$ from a $p$-divisible group $\Gscr$ to its group of $k$-points 
incurs a {\em profound} loss of information: for example, when $K=k(x)$ is the rational function field and $S=\{\infty\}$,
the corresponding $p$-divisible groups $\Gscr_n$ have 
$\Gscr_n(k) = 0$ for all $n$, yet the dimension of $\Gscr_n$ can grow {\em arbitrarily} fast! 
In general, a $p$-divisible group $\Gscr$ breaks up as the product $\Gscr = \Gscr^{\et} \times \Gscr^{\mult}\times \Gscr^{\ll}$ of its  \'etale, %
multiplicative, %
and local-local %
components, %
and one has 
$\Gscr(k)=\Gscr^{\et}(k)$.  In other words, the passage to $k$-points completely loses the connected component $\Gscr^0=\Gscr^{\mult}\times \Gscr^{\ll}$,
which is arguably the {\em most} interesting part of $\Gscr$.

The guiding philosophy of this paper is that Iwasawa theory for function fields should encompass the study of the
{\em entire} $p$-divisible group $\Gscr_n$
and not just its group of $k$-rational points. There is compelling recent evidence
for the merit of this inclusive vision:
in \cite{DWX}, Davis, Wan, and Xiao use techniques from Dwork theory to study the $p$-adic properties
of Hasse--Weil Zeta functions of the curves $X_n$ in certain $\Z_p$-extensions of $k(x)$
when $S=\{\infty\}$.
Their work has been extended 
by 
Xiang \cite{XiangSlopes} and Kosters--Zhu
\cite{KostersZhu}
to larger classes of $\Z_p$-extensions
of $k(x)$ with $S=\{\infty\}$, and more recently to $\Z_p$-extensions of arbitrary {\em ordinary}
function fields $K$ with $S$ any finite set under certain
ramification hypotheses\footnote{These hypotheses, called {\em eventually minimal break ratios}
in this paper (see Remark \ref{rem: embr}) are imposed by all cited prior work on this subject.
We warn the reader that \cite{KostersZhu}
relies on the erroneous \cite[Theorem 1 (3)]{KW}, %
a corrected version of which is provided by \cite{KWErrata}.  As such, 
the main results of \cite{KostersZhu} are only valid
for those $\Z_p$-extensions with eventually minimal break ratios, and not for the larger class of {\em genus stable}
extensions as claimed.
}
by the second two authors \cite{kramer-milleruptonI}, \cite{kramer-milleruptonII}.
Via the crystalline interpretation of $L$-functions ({\em e.g.}~\cite{KatzMessing}, \cite{KatzCrystalline}) and the canonical identification of the first crystalline cohomology of $X_n$ with the Dieudonn\'e
module of $\Gscr_n$ \cite{MazurMessing}, it follows from these papers that 
the {\em slopes} of the isocrystal associated to $\Gscr_n$ exhibit astonishing regularity
as $n\rightarrow \infty$, 
and in particular become uniformly distributed in the unit interval.  
By the Dieudonn\'e--Manin classification, the $p$-divisible group
$\Gscr_n$ is determined up to isogeny over $\overline{k}$ by the set of slopes of its associated isocrystal, 
so these papers may be interpreted as providing an analogue of Iwasawa theory for the {\em isogeny type}
of $p$-divisible class group schemes in such $\Z_p$-extensions.

Unfortunately,
the isogeny type of a $p$-divisible group is a somewhat coarse invariant,
as it loses touch with finite, {\em torsion} subgroup schemes, whose
structure can be exceedingly subtle.
In \cite{boohercais21},
the first two authors have studied
the isomorphism type of the Frobenius-torsion subgroup schemes $\Gscr_n[F]=J_{X_n}[F]$
in $\Z_p$-extensions of $k(x)$ with $S=\{\infty\}$.  By Oda's Theorem \cite{Oda} and Dieudonn\'e theory,
 $\Gscr_n[F]$ is determined up to isomorphism by the
sequence of nonnegative integers
\begin{equation}
        a_n^{(r)}\colonequals a_{X_n}^{(r)} \colonequals  \ker\left( V^r : H^0(X_n,\Omega^1_{X_n/k})\rightarrow H^0(X_n,\Omega^1_{X_n/k})\right), \quad r \ge 1\label{anumdef}
\end{equation}
where $V$ is the {\em Cartier operator} acting on the space of global, holomorphic differential forms on $X_n$ over $k$.
When $r=1$, the integer $a_n^{(1)}$ is called the {\em $a$-number} of the curve $X_n$, 
and is a subtle and mysterious quantity that has been 
 studied extensively in many %
contexts 
\cite{CMHyper,VolochChar2,Fermat,ReBound,ElkinPriesanum1,ElkinCyclic,Suzuki,Dummigan,FermatHurwitz,Frei}.
While there are upper and lower bounds for $a$-numbers \cite{BCanum} and their
``higher" (i.e. $r>1$ above) counterparts \cite{groen} in arbitrary branched $\Z/p\Z$-covers of curves,
these bounds leave room for significant variation, and  
there can be no exact formula for the $a$-number of a $\Z/p\Z$-cover in general
\cite[Example 7.2]{BCanum}.  
Via extensive numerical computation for small values of $p$ and $n$, as well as some theoretical results when $p=2$ and $r=1$,
\cite{boohercais21} provides support for an array of conjectures which, for each fixed $r\ge 1$, 
predict that the higher $a$-numbers \eqref{anumdef} behave with striking regularity as $n\rightarrow \infty$. 

In this paper, we prove the main conjectures of \cite{boohercais21} for a broad class of $\Z_p$-extensions of $k(x)$
by establishing an explicit formula that gives 
the higher $a$-numbers \eqref{anumdef} up to bounded (and sometimes zero) error as $n\rightarrow \infty$. 
In order to state it, we fix some notation and terminology.  
We say that a geometric $\Z_p$-extension $L/K$ of $K=k(x)$ with $S=\{\infty\}$ 
has {\em minimal break ratios} if the $n$-th break in the upper numbering ramification
filtration at $\infty$ is equal to $dp^{n-1}$, for some positive integer $d$
that 
we call the {\em ramification invariant} of $L/K$.
Given such an $L/K$, for an integer $i>0$ we write $\frac{p+1}{d}i = \sum_{m\in \Z} i_m p^m$ in base $p$
and define
\begin{equation*}  
    \mu_i \colonequals  \begin{cases}
            \left\lfloor \frac{p+1}{d}i \right\rfloor & \text{if}\ (i_m)_{m\ge 0} \succ (i_{-1-m})_{m\ge 0} \\ \\
            
            \left\lceil \frac{p+1}{d}i \right\rceil & \text{otherwise}
    \end{cases}\quad\text{and}\quad
  t(n) \colonequals  \begin{cases}
            \left\lfloor \frac{dp^n}{(r+1)p - (r-1)}\right\rfloor & \text{if}\ d | (p-1) \\ \\
            d\left\lfloor \frac{p^n-1}{(r+1)p - (r-1)}\right\rfloor & \text{otherwise}
    \end{cases}
\end{equation*}
where the definition of $\mu_i$ uses the lexicographic ordering ``$\succ$'' on the set of integer sequences.

\begin{theorem}\label{MT}
    Let $L/K$ be a geometric $\Z_p$-extension of the rational function field $K=k(x)$
    totally ramified over $S=\{\infty\}$ and unramified elsewhere, with minimal break ratios and ramification invariant $d$.
    For each $r$, there exists a constant $C_{p,d,r}\ge 0$ depending only on $d,p$, and $r$,
    such that
        \begin{equation}
                0 \le \frac{r (p-1) t(n) (t(n)+1)}{2d} + \#\{(i,j)\in \Z^2\ :\ i > t(n),\ p^n >  j \ge \mu_i \} - a_n^{(r)}
                \le C_{p,d,r}\label{anumformula}
        \end{equation}
    for all $n > 0$.  When $d|(p-1)$, we may take $C_{p,d,r}=0$.  When $d\le p+1$ and $r\not\equiv 1\bmod p$,
    \eqref{anumformula} holds with $C(p,d,r)=0$ whenever $n\equiv 0\bmod m$, with $m$ the order
    of $p$ modulo $(p-1)r+(p+1)$.
\end{theorem}

Estimating the lattice-point count in \eqref{anumformula}, we deduce a case of \cite[Conjecture 3.4]{boohercais21}:

\begin{corollary} \label{cor:asymptoticmain}
    With the hypotheses of Theorem \ref{MT}, for each $r\ge 1$
    \begin{equation*}
         a_n^{(r)} = \frac{r}{r + \frac{p+1}{p-1}} \frac{dp^{2n}}{2(p+1)} + O(p^n) \quad\text{as}\ n\rightarrow \infty,
    \end{equation*}
    with an implicit constant depending only on $d,p,r$. In particular, if $g_n$ is the genus of $X_n$ then
    \begin{equation}\label{eq:a-genus asymptotic}
         \lim_{n\rightarrow \infty} \frac{a_n^{(r)}}{g_n} = \frac{r}{r + \frac{p+1}{p-1}}.
    \end{equation}
\end{corollary}

When $r=1$, $p>2$, and $d|(p-1)$, we are able to compute the lattice-point count in \eqref{anumformula} exactly:

\begin{corollary}\label{MC2:formula}
    With the hypotheses of Theorem \ref{MT}, if $p>2$ and $d|(p-1)$ the $a$-number of $X_n$ is
    \begin{equation}
        a_n^{(1)} = \frac{p-1}{2}\frac{d}{2(p+1)}(p^{2n-1}+1) - \begin{cases} 0 & d\ \text{even} \\ \frac{p-1}{4d} & d\ \text{odd}\end{cases}
    \end{equation}
\end{corollary}

\begin{remark}
When $p=2$, an analogous formula for {\em arbitrary} $d$ is given by\footnote{
Although \cite[Corollary 8.12]{boohercais21} is only asserted for the smaller class of ``basic'' $\Z_2$-towers in the sense of
\cite[Definition 2.12]{boohercais21}, the proof given works for all $\Z_2$-extensions of $k(x)$
with $S=\{\infty\}$ that have minimal break ratios.
}
\cite[Corollary 8.12]{boohercais21}.
The special case $n=1$ of Corollary \eqref{MC2:formula} follows from \cite[Theorem 1.1]{FarnellPries},
which gives a formula for the $a$-number of any branched $\Z/p\Z$-cover $X_1\rightarrow \PP^1$ ramified at $S$
when the least common multiple of the ramification breaks at all points in $S$ is a divisor of $p-1$.  
\end{remark}

When $r>1$ and $d|(p-1)$, we do not have a simple explicit formula.  However, 
in this situation, \cite{promys} shows that the lattice-point count in \ref{MT} 
obeys a beautiful Iwasawa--style formula:

\begin{corollary} \label{r>1result}
 Under the hypotheses of Theorem \ref{MT}, if $d|(p-1)$ then for each $r$ there exists $\lambda_r\in \Q$
and a periodic function $\nu_r: \Z\rightarrow \Q$ such that for all $n$ sufficiently large
    \begin{equation}
        a_n^{(r)} = \frac{r}{r + \frac{p+1}{p-1}} \frac{dp^{2n}}{2(p+1)} + \lambda_r n + \nu_r(n).
    \end{equation}
    Furthermore, writing $D_r$ for the prime to $p$ part of the denominator of $d/(r(p-1)+(p+1))$ in lowest terms,
     the period of $\nu_r$ divides the least common multiple of $2$ and the order of $p$ modulo $D_r$.
\end{corollary}

Corollary~\ref{MC2:formula} and Corollary~\ref{r>1result} establish \cite[Conjecture 3.7 and 3.8]{boohercais21} for $\{X_n\}$ when $d | (p-1)$.

\begin{remark}
  The restrictions to $K= k(x)$ and $S=\{\infty\}$ are unnatural, and we expect to be able to remove them.
  Indeed, a primary external input in this paper
    is the work \cite{kramer-milleruptonI}, \cite{kramer-milleruptonII}
    of the second two authors, which studies Hasse--Weil $L$-functions
    in $\Z_p$-extensions of function fields, and allows for
    arbitrary ordinary base curves, as well as ramification
    over {\em any} (finite) number of points under the hypothesis
    of (eventually) minimal break ratios.  It is moreover worth pointing out that
    the asymptotic formula \eqref{eq:a-genus asymptotic} makes no reference to the ramification invariant $d$, so might be expected to hold 
    under considerably relaxed hypotheses on the ramification
    (cf.~\cite[\S6.4--6.5]{boohercais21}).
\end{remark}

    Theorem \ref{MT} and its corollaries may also be understood as Iwasawa--Theoretic refinements of the Riemann--Hurwitz and Deuring--Shafarevich formulae.  To explain this perspective, suppose for the moment that $k=\overline{k}$ is {\em any} algebraically closed field, and $X$ is a smooth, projective and connected curve over $k$
    of
    genus 
    $g_X:=\dim_k H^0(X,\Omega^1_{X/k})$.
    In any branched Galois cover of curves
    $Y\rightarrow X$ with group $G$, %
    the {\em Riemann--Hurwitz} formula is an exact expression for $g_Y$ in terms of $g_X$, the order of $G$, 
    and the ramification filtration of $G$ over the branch points.
    When $k$ has positive characteristic $p$, the Cartier operator $V$ gives
    $H^0(X,\Omega^1_{X/k})$ the structure of a $k[V]$-module, and it follows from Fitting's Lemma
    and Lang's Theorem \cite[Theorem 1]{Lang-algebraic_groups_over_finite_fields}
    that
    one has an isomorphism of $k[V]$-modules
    $$
        H^0(X,\Omega^1_{X/k})\simeq \left(\frac{k[V]}{(V-1)}\right)^{f_X} \oplus \bigoplus_{i\ge 1} \left(\frac{k[V]}{(V^i)}\right)^{m_X(i)}
    $$
    for a nonnegative integer $f_X$ (called the $p$-rank of $X$) and an eventually zero sequence
    of nonnegative integers $\{m_X(i)\}_{i\ge 1}$.  %
    It is then natural to ask for analogues
    of the Riemann--Hurwitz formula for the numerical invariants $f_{\bullet}$ and $m_{\bullet}(i)$ in branched $G$-covers
    of curves.  For general $G$, simple examples show there can be no direct analogue for $f_{\bullet} $ \cite[Remark 1.8.1]{CrewEtale}.  However, 
    when $G$ is a $p$-group, the {\em Deuring--Shafarevich} formula provides an exact expression for $f_Y$ in terms
    of $f_X$, the order of $G$, and the ramification indices of the branch points.
    On the other hand,
    there can be no exact formula for $m_{\bullet}(i)$ in general, even for $p$-group covers \cite[Example 7.2]{BCanum}.  
    Nevertheless, it follows from the relation $m_{X_n}(i)=2a_n^{(r)}-a_n^{(r-1)} - a_n^{(r+1)}$
    that there {\em is} a formula for $m_{X_n}(i)$
    in any $\Z_p$-extension satisfying the hypotheses of Theorem \ref{MT}, with bounded error as $n\rightarrow \infty$.
    In particular:

    \begin{corollary}\label{c: asymptotic decomposition of k[V]-structure}
        For each fixed $i\ge 1$ 
    $$
        \lim_{n\rightarrow \infty} \frac{m_{X_n}(i)}{g_{X_n}} = \frac{2 \tau}{(i+\tau)^3-(i+\tau)}\quad \text{where}\ \tau:=\frac{p+1}{p-1}.
    $$
    When moreover $d|(p-1)$, Corollary \ref{r>1result} yields an exact Iwasawa--style formula
    \begin{equation*}
            m_{X_n}(i) = \frac{d}{\left(i+\tau\right)^3 - \left(i+\tau\right)}\cdot \frac{p^{2n}}{p-1}
            + \lambda_i \cdot n + \nu_i(n)
    \end{equation*}
    for all $n$ sufficiently large, where $\lambda_i\in \Q$ is a constant (depending only on $p$, $d$, and $i$),
    and $\nu_i:\Z\rightarrow \Q$ is a periodic function (again depending only on $p$, $d$, and $i$).
    \end{corollary}

\begin{remark}
We emphasize that, by the Deuring--Shafarevich formula,
{\em any} $\Z_p$-extension of $k(x)$
with $S=\{\infty\}$ has 
$| \Cl_{K_n}[p^\infty]| =1$ for all $n$.  Thus, the
$p$-primary torsion in the
``physical'' class group $\Cl_{K_n}$ of $K_n$ entirely misses all of the interesting Iwasawa-theoretic phenomena in these towers: such phenomena appear {\em only} when considering the $p$-primary class group {\em schemes} $J_{X_n}[p^\infty]$.
\end{remark}    

Theorem~\ref{MT} and its corollaries establish an entirely new %
instance of Iwasawa's philosophy that $p$-power torsion in class groups of $\Z_p$-extensions of global fields behaves regularly, and represent the beginning of an exciting new trajectory in geometric Iwasawa theory.
Ultimately, one would like an Iwasawa-theoretic understanding of the full $p$-divisible group $\Gscr_n$
as $n\rightarrow \infty$.  Using Dieudonn\'e theory, \cite{ID} develops a systematic approach to the study of
these {\em $p$-divisible class group schemes} in arbitrary, pro-$p$ extensions of function fields.
A main result of this work is that the ``Iwasawa--Dieudonn\'e'' modules attached to 
the \'etale and multiplicative components of $p$-divisible class group schemes in pro-$p$
extensions of function fields are {\em always} finitely generated over the %
Iwasawa algebra, and satisfy a {\em control} theorem
enabling the recovery of each finite level from the limit object; it follows that both the \'etale and
multiplicative components grow in a regular and controlled manner.  
This provides in particular a direct generalization
of the work \cite{MWAnalogies} of Mazur--Wiles to arbitrary pro-$p$ extensions of function fields.
In the case of everywhere unramified pro-$p$ extensions, which exist over 
sufficiently large (necessarily infinite) base fields, \cite{ID} shows that the 
Iwasawa--Dieudonn\'e module attached to local--local $p$-divisible class group schemes
is {\em also} finitely generated over the Iwasawa algebra, and satisfies control.  It follows from this
that the $p$-divisible groups in \'etale pro-$p$ extensions exhibit remarkable regularity;
for example, \cite[Corollary 1.4]{ID} gives an analogue of Corollary \ref{cor:asymptoticmain} in this setting.
Unfortunately, the Iwasawa--Dieudonn\'e module associated to the local--local components
of $p$-divisible class group schemes in {\em ramified} pro-$p$ extensions
is {\em never} finitely generated \cite[Theorem E]{ID}, %
so understanding the growth and behavior of these local-local $p$-divisible groups
lies beyond the scope and methods of traditional Iwasawa theory.  
The present paper provides compelling evidence
that these $p$-divisible class group schemes nevertheless
grow with astonishing regularity, and 
strongly suggests that the Iwasawa--Dieudonn\'e modules attached to them
are {\em Banach modules} over the Iwasawa algebra, and should be studied in this context.

\section*{Acknowledgments}
Jeremy Booher and Joe Kramer-Miller would like to thank Richard Crew for valuable discussions. Jeremy Booher was supported by the Marsden Fund administered by the Royal Society of New Zealand. Bryden Cais was supported by NSF grants DMS-2302072 and DMS-1902005.

	\section{Overview of the article}
 \label{s: outline of the article and sketch of the proof}

 As in Section \ref{sec:intro}, we fix a geometric $\Z_p$-extension $L/K$ of $K=k(x)$, totally ramified over $S=\{\infty\}$
 and unramified elsewhere We write $K_n$ for the unique subextension of $L/K$ of degree $p^n$
 over $K$, and $X_n$ for the corresponding smooth projective curve; in particular, $X_0=\PP^1$. 
We set $A:=k\powerseries{T}$, which we identify with the characteristic-$p$ Iwasawa algebra $k[\![\Gal(L/K)]\!]$
by sending a fixed choice of topological generator of $\Gal(L/K)$ to $1+T$. %
Writing $P_n$ for the unique point of $X_n$ over $\infty$,
in Section~\ref{s: towers of curves and iwasawa modules of differentials} we study the 
inclusion of 
$A$-modules:
\begin{equation*}
    \xymatrix{
        {M_n:=H^0(X_n,\Omega^1_{X_n/k})} \ar@{^{(}->}[r] & H^0(X_n-P_n,\Omega^1_{X_n/k})=: Z_n
    }
\end{equation*}
Writing $R_n$ %
for the affine coordinate ring of $X_n-P_n$, so $R_0=k[x]$, we have a natural identification of $k[x][T]/(T^{p^n})$-modules
$Z_n=R_n dx$.
Passing to inverse limits along the canonical the trace mappings, we obtain an inclusion of topological $A$-modules
\begin{equation}
    \xymatrix{
        {M:=\varprojlim_n M_n} \ar@{^{(}->}[r] & \varprojlim Z_n=: Z.
    }\label{eq:MinZ}
\end{equation}
Setting $R:=\varprojlim_n R_n$ we have an identification $Z =Rdx$ of modules over the Tate algebra $A\langle x\rangle =\varprojlim_n k[x][T]/(T^{p^n})$.
The starting point for our work (Proposition \ref{prop: control type results}) is that $R$ is {\em free} of {\em rank one} as an $A\langle x\rangle$-module; in particular, there exists $w_0 \in R$ such that $R = A\langle x \rangle w_0$. %
Writing $w_0 dx$ for the corresponding element of $Z$, we have $Z=A\langle x\rangle \cdot w_0 dx$
as an $A\langle x\rangle$-module, and 
the $E:=\Frac(A)$-vector space $Z\otimes_A E$ is
naturally an $E$-Banach space, isomorphic as such to the $T$-adic Tate algebra $E\langle x\rangle$.
The Cartier operator $V$ induces a semilinear map on $Z$ which preserves $M$,
and we show using the fact that $Z\simeq R$ is of rank one
 over  $A\langle x\rangle$ that $V$ is completely determined by a single 
 power series $\alpha\in A\langle x \rangle$ (see Definition~\ref{alphadef} and Corollary \ref{cor:rankone}).
Our strategy to understand $V$ on $M_n$ is to analyze the $T$-adic properties of $V$ on $M$
via the inclusion \eqref{eq:MinZ} and a careful analysis of $\alpha$, and then pass to finite level using the $A$-module identification
\begin{equation}\label{eq: intro 1}
		M_n \simeq \frac{M}{T^{p^n} Z \cap M}.
\end{equation}
We emphasize that although $Z$ satisfies {\em control} in the sense that $Z/T^{p^n}Z\simeq Z_n$,
recovering $M_n$ from $M$ is more complicated.
	
	In Section \ref{ss: functions and differentials in the limit} we use explicit Artin--Schreier--Witt theory
 to analyze the inclusion \eqref{eq:MinZ}; this analysis depends on the fact that our base field
 is the rational function field (see Proposition \ref{proposition:standardformB} and its proof).
 For $i\ge 1$, we construct unit power series 
 $r_i=r_i(x) \in A\langle x\rangle^{\times}$, as well as
a sequence of positive integers $\mu_1\le \mu_2 \le  \dots$
 determined by the ramification filtration at $\infty$ such that 
        \begin{equation}
            u_i \colonequals r_{\mu_i}  x^{i-1}w_0 dx \quad\text{for}\ i\ge 1\label{eq:uelt def}
        \end{equation} 
satisfy the following:
\begin{enumerate}
        \item[i] The set $\{u_i\}_{i\ge 1}$ is an orthonormal basis of the Banach space $Z\otimes_A E$ over $E$.
        \item[ii] The set $\{T^{\mu_i}u_i\}_{i\ge 1}$ is a $A$-basis\footnote{In the sense 
        that every element of $M$ is a unique (possibly infinite) $A$-linear combination
        of the elements $T^{\mu_i}u_i$.} of $M$.
\end{enumerate}
This provides a precise understanding of the way in which $M$ sits inside of $Z$ as an $A$-submodule,
which we will exploit to transfer information about $V$ on $Z$ to an understanding of $V$ on $M$.

An important point is that we do not {\em directly} analyze $V$ on $M$, but rather infer 
information about $V$ on $M$ by studying its action on certain ``overconvergent'' $A$-submodules
of $Z$ defined in terms of the $u_i$.  In order for this strategy to be effective, 
we must impose sufficient regularity on the sequence $\mu_i$ (which can in general
be highly irregular) to gain control of $V$ on these submodules.
With this in mind, in Section \ref{s: tadic growht for minimally ramified towers} we restrict our attention to extensions 
$L/K$ having minimal break ratios and ramification invariant $d$.
The key result is in this setting is Corollary \ref{cor: Frobenius element growth},
    which shows that the power series $\alpha \in A\langle x\rangle$
    {\em overconverges} to the open disc $v_T(x)> -\frac{1}{d}$, with values
    bounded by $1$.  To fully exploit this fact, we work over the extension
    $\calA:=k[\![T^{1/d}]\!]$ of $A$ obtained by adjoining a $d$-th root of $T$. 
    For $m>0$, we define the $\calA$-submodules of $\calA\langle x\rangle$ and $\calZ:=Z\otimes_A \calA$
    \[
        \calL^m\colonequals  \Bigg\{ \sum_{i=1}^\infty a_ix^i ~|~ a_i\in \calA\ \text{and}\ v_T(a_i) \geq \frac{i}{m} \Bigg \},\quad
        \calZ^m\colonequals  \Bigg\{ \sum_{i=1}^\infty a_iu_i ~|~ a_i\in \calA\ \text{and}\ v_T(a_i) \geq \frac{i}{m} \Bigg \}.
    \]
    Note that $\calL^m$ is precisely the $\calA$-submodule of $\calA\langle x\rangle$ consisting of functions that 
    vanish at $x=0$ and
    converge on the open disc $v_T(x)>-\frac{1}{m}$ with values bounded by 1.  
In Corollary \ref{cor:functionsqb}, we prove that $r_{\mu_i}-r_{\mu_i}(0)\in \calL^{d/p}$ for all $i$,
from which we obtain an identification of $\calA$-modules 
\begin{equation}\label{eq:Lm eq Zm}
    x^{-1}\calL^m w_0 dx=\calZ^m\
\end{equation}
for any $m\ge d/p$ (see Lemma \ref{l: growth property of Z^m for normal basis}).  
Using the overconvergence of $\alpha$, in Lemma \ref{l: action of of V on Z^m} 
we deduce from \eqref{eq:Lm eq Zm}
the key fact that $V$ improves $T$-adic convergence by a factor of $p$: precisely, $V(\calZ^d) \subseteq \calZ^{d/p}$.
In Corollary \ref{c: V-estimate for basis eliment}, we apply this key fact
to obtain good $T$-adic estimates for the action of $V$ on $\calZ^{d/p}$; it follows, in particular,
that the $T$-adic {\em Hodge Polygon} of $V$ on $\calZ^{d/p}$ (defined along the lines of \cite[\S1.2]{Katz})
lies on or above the polygon $\HP(d)$ with slopes (each of multiplicity 1)
\[
    \frac{p-1}{d}, 2\frac{(p-1)}{d}, 3\frac{(p-1)}{d}, \dots.
\]

To proceed further, in Section \ref{ss:cartierestimates} we sharpen these $T$-adic estimates for $V$ on $\calZ^{d/p}$ by applying ideas inspired from Dwork theory. 
For simplicity of exposition, we describe only the case $k=\F_p$, so that $V$ is actually {\em linear} over $\calA$; the general case
is similar, presenting only minor technical difficulties.
After recalling some basic facts in the theory of Banach spaces in Section~\ref{ss:recollectionsbanach},  
we show that our overconvergence results imply that
 the {\em Fredholm determinant} $\Det_{\calA}(1-sV ~|~\calZ^{d/p})$ of $V$ acting on $\calZ^{d/p}$ is a well-defined
 power series in $s$ with coefficients in $\calA$. 
    Write
\begin{equation*}
\xymatrix{
    {\chi: \pi_1^{\et}(\mathbf{A}^1_{\overline{k}})} \ar@{->>}[r] & {\Gal(L/K)} \ar@{^{(}->}[r] & {k[\![\Gal(L/K)]\!]^{\times}\simeq {A^{\times}}}
}
\end{equation*}
for the $A$-valued character associated to our $\Z_p$-extension $L/K$
and our identification of the Iwasawa algebra of $\Gal(L/K)$ with $A$.
Using an equicharacteristc version of the Dwork trace formula, we deduce from
this %
that the Artin $L$-function of $\chi$ coincides with the Fredholm determinant of $V$:
	\[L(\chi,s)=\Det_{\calA}(1-sV ~|~\calZ^{d/p}). \] 	 
In particular, the $T$-adic Newton polygon of $L(\chi,s)$
and the Newton polygon of $V$ are {\em equal}.  
By the usual ``Newton over Hodge''-type formalism (e.g.~\cite[Theorem 1.4.1]{Katz}),
one knows that the Newton polygon of $V$ lies on or above the
Hodge polygon of $V$, %
which we have seen lies on or above the polygon $\HP(d)$.
On the other hand, 
work of the last two authors \cite{kramer-milleruptonII} gives strong estimates for the Newton polygon of $L(\chi,s)$,
showing in particular that it {\em coincides} with the polygon $\HP(d)$ every $d$-slopes (i.e.~the points with $x$-coordinate $dk$ for $k\in \Z$ are the same on both polygons).  When moreover $d|(p-1)$, they show that the two polygons are {\em identical}. 
It follows that the Hodge polygon of $V$ on $\calZ^{d/p}$ coincides with $\HP(d)$ every $d$-slopes for general $d$,
and when $d|(p-1)$, is equal to $\HP(d)$ on the nose.  Although we do not explicitly use the formalism 
of Hodge polygons, this periodic coincidence is essential to our analysis, and is exploited in 
Section \ref{ss: consequences of NP and HP coincidence} to deduce refined estimates for the behavior of $V$
on the basis elements $u_i$.  

These refined estimates give us sufficiently strong control over the $T$-adic 
properties of $V$ and its iterates on the $u_i$ that we are able to deduce our main theorem from them
in Section \ref{s:calculation of higher a numbers}.  As one might expect, the actual analysis %
is both technical and nuanced: 
in general---and in contrast to Newton polygons---the Hodge polygon of $V^r$ on $\calZ^{d/p}$ is not
easily related to the Hodge polygon of $V$.  Nevertheless, our precise control of the $T$-adic properties of $V$ on 
the $u_i$ is just strong enough for our desired application to estimating the higher $a$-numbers \eqref{anumdef},
providing an exact formula in the most favorable situation $d|(p-1)$.  Section~\ref{ss:toy} presents a ``toy model'' of this argument which illustrates the essential behavior.

Finally, in Section \ref{s: formulae and estimates} we estimate the lattice count in Theorem \ref{MT} to obtain the asymptotic results stated in Corollary \ref{cor:asymptoticmain}: it is essentially the area of triangle above and to the right of the diagonal in Figure~\ref{fig:picture} below. When $r=1$ and $d|p-1$ we obtain precise results, giving the exact formula in Corollary \ref{MC2:formula}.

\subsection{A simple ``toy model''} \label{ss:toy}
It follows easily from definitions that the set $\calB=\{T^{ip/d}u_i\}_{i\ge 1}$
is an $\calA$-basis of $\calZ^{d/p}$.
When $d|(p-1)$, so that the Hodge and Newton polygons of $V$ coincide, an analogue of ``Hodge--Newton decomposition'' (e.g.~\cite[Theorem 1.6.1]{Katz})
implies the existence of an $\calA$-basis $\calB'$
 of $\calZ^{d/p}$ in which the matrix of $V$ is {\em diagonal}.
Unfortunately, we have little control over the change of basis matrix
from $\calB$ to $\calB'$, which prevents us from directly 
transferring the simple behavior of $V$ and its iterates
on $\calB'$
to useful information about $V$ on the elements $u_i$.
Nevertheless, our argument in Section \ref{s:calculation of higher a numbers} was inspired by, and in many ways reflects,
the heuristic {\em toy} argument founded on the premise
that the matrix of $V$ relative to $\calB$ is already diagonal,
i.e.~that $Vu_i = \lambda_i u_i$ for some $\lambda_i\in \calA$.
Although this assumption is incorrect,
it allows us to illustrate
the main conceptual ideas of Section \ref{s:calculation of higher a numbers}
in a way that is not only faithful to the spirit of the argument, 
but even predicts the higher $a$-numbers exactly when $d|p-1$. 

For simplicity, we continue to assume that $k=\F_p$, so that $V$ is $k$-linear.
We allow arbitrary $d$, but we assume that each $u_i$ is an eigenvector of $V$ whose eigenvalue\footnote{We remark that the heuristic works with the milder assumption that the eigenvalue has $T$-adic valuation $b_i$, but we make the simplifying assumption for convenience.} is $T^{b_i}$ with $b_1 \le b_2 \le\dots$. Via \eqref{eq: intro 1}
and the fact that the $u_i$ are an $A$-basis of $Z$, we see that
a $k$-basis for $M_n$ is given by the images of $T^j u_i$ for $(i,j)\in \Z^2$
satisfying $1\le i \le i(n)$ and $\mu_i \leq j \leq p^{n}-1$, where $i(n)$ is the largest integer
with $\mu_{i(n)} < p^n$.  We represent this $k$-basis of $M_n$ graphically in
Figures \ref{fig:picture} and \ref{fig:picture 2} by associating the lattice point $(i,j)$ to the
$k$-basis element corresponding to $T^j u_i$. 
We may then view $V$ as acting on lattice points by sending $(i,j)$ to $(i,j + b_i)$.
Again using \eqref{eq: intro 1}, the differential given by $T^ju_i$ lies in the kernel
of $V^r$ on $M_n$ if and only if $V^r(T^ju_i) \in T^{p^n}Z$; since the $u_i$ are an $A$-basis
of $Z$, this translates to the condition
\begin{equation}\label{eq: being in the kernel}
    j+rb_i \geq p^n,
\end{equation}
so that the lattice point $(i,j)$ corresponds to a $k$-basis element of $M_n$ that lies in the kernel of $V^r$ if and only if $V^r$ moves it on or above the line $y=p^n$. 
Since the $\mu_i$'s and the $b_i$'s are both non-decreasing sequences, there exists an exact ``cutoff point'' $\tau(n)$ satisfying the following:
\begin{enumerate}
    \item For $i\leq \tau(n)$ we have $ \mu_i < p^n-rb_i$. This means that the lattice points $(i,p^n-j)$ for $rb_i \le j < p^n$
     correspond to $k$-independent global differentials in $M_n$
     that lie in the kernel of $V^r$ by \eqref{eq: being in the kernel}, which contributes
    \begin{equation}\label{eq: being in the kernel 1}
        \sum_{i=1}^{\tau(n)} rb_i
    \end{equation}
    to the dimension of  $\ker(V^r)$.
    \item For $i > \tau(n)$ we have $\mu_i\geq p^n-rb_i$. 
    Again by \eqref{eq: being in the kernel}, the lattice points
    $(i,j)$ for $\mu_i \le j < p^n$ correspond to $k$-linearly independent global differentials in the kernel of $V^r$, contributing
    \begin{equation}\label{eq: being in the kernel 2}
        \sum_{i=\tau(n)+1}^\infty \max(0, p^n-\mu_i)
    \end{equation}
    to the dimension of $\ker(V^r)$.
\end{enumerate}

When $d|(p-1)$, we know the Newton polygon of $V$ exactly: the slopes are $\frac{p-1}{d}i$ for $i\ge 1$, each with multiplicity $1$, so $b_i=i\frac{p-1}{d}$. As such, the points $(i,p^n-rb_i)$ all lie on the line $y=p^n-ir\frac{p-1}{d}$, 
displayed in Figure \ref{fig:picture} as the downward sloping red line. In this case,
the exact ``cutoff point'' is
\begin{equation*}
     t'(n)=\left \lfloor \frac{dp^n}{(r+1)p-(r-1)} \right \rfloor,
\end{equation*}
which is represented in Figure \ref{fig:picture} by the vertical line $x=5$. Combining \eqref{eq: being in the kernel 1} and \eqref{eq: being in the kernel 2}, we get
\begin{equation*}
\dim \ker V^r|_{M_n} = \frac{r (p-1) t'(n) (t'(n)+1)}{2d} + \#\{(i,j)\in \Z^2\ :\ i > \tau(n),\ p^n >  j \ge \mu_i \}.
\end{equation*}

\begin{figure}[ht]
\includegraphics[height=3in]{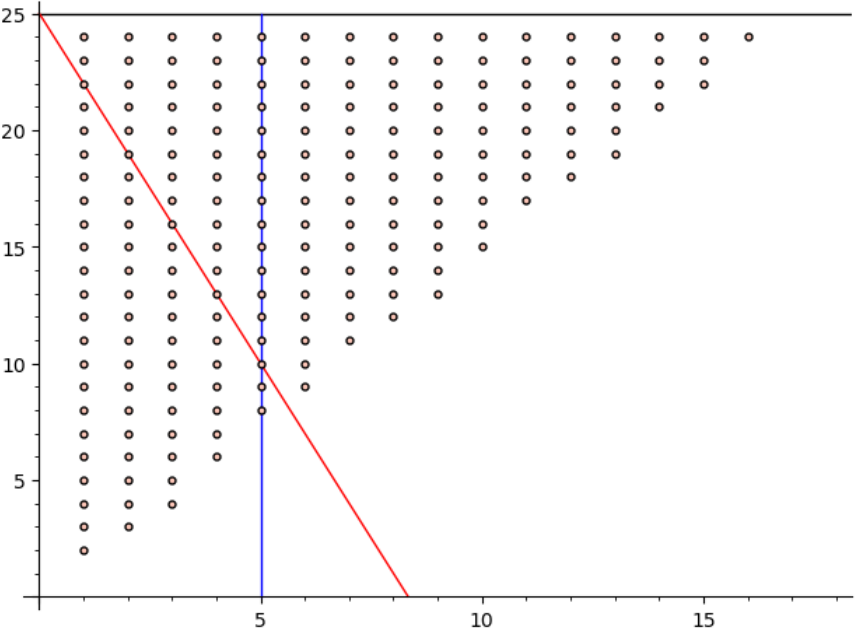}
\caption{A $k$-basis of $M_2$ when $p=5$, $d=4$, $r=3$.  The lattice point $(i,j)$ represents the basis element $T^j u_i$ for $M_2$.  Lattice points on or to the right of the diagonal line are sent to $0$ by $V^3$.
The vertical line represents the cutoff $\tau(n)=t'(n)$.}
\label{fig:picture}
\end{figure}

When $d\nmid (p-1)$, we only know that the Newton and Hodge polygons coincide periodically (i.e.~every $d$ slopes). This translates to the following information about the $b_i$'s:
\begin{enumerate}
    \item For $dk < i \leq d(k+1)$ we have $k(p-1) < b_i \leq (k+1)(p-1)$.\label{item: gend1}
    \item The average value of $b_i$ for $i=dk+1,\dots, d(k+1)$ is $\sum_{i=1}^d \frac{p-1}{d} (dk+i)$.\label{item: gend2}
\end{enumerate}
We can then repeat the argument above, with the caveat that not knowing the exact value of $b_i$ introduces uncertainty,
and we can not determine the exact `cutoff point'
in this case.  Instead, our analysis is optimized when we choose 
the closest approximate `cutoff point' that is a multiple of $d$:
\begin{equation*}
    t(n) = d \left \lfloor \frac{p^n-1}{(r+1)p-(r-1)}  \right\rfloor;
\end{equation*}
this is represented in Figure \ref{fig:picture 2} by the vertical line $x=6$.
When $i \leq t(n)$, the points $(i,p^n-j)$ for $r b_i\le j< p^n$ 
correspond to $k$-independent global differentials in the kernel of $V^r$, which contributes
\begin{equation*}
    \sum_{i=1}^{t(n)}rb_i=\frac{r(p-1)t(n)(t(n)+1)}{2d}
\end{equation*}
to $\dim \ker (V^r)$. Here, we make use of 
the known average value of the
the $b_i$'s when computing the summation.
There is some uncertainty %
when $i > t(n)$: we only have a range of possibilities for $b_i$, so it is possible that $p^n-rb_i>\mu_i$. 
In other words, not all lattice points to the right of the cutoff point need correspond to $k$-independent global differentials
in the kernel of $V^r$. 
Nevertheless, using our bounds on $b_i$ and the 
regularity that our hypothesis of minimal break ratios imposes on the $\mu_i$,
we show that 
the number of lattice points $(i,j)$ with $i>t(n)$
that {\em do not} correspond to $k$-independent global differentials
in the kernel of $V^r$ is bounded independently of $n$, 
which yields
\begin{align*}
    0 \leq  \frac{r(p-1)t(n)(t(n)+1)}{2d} + \#\{(i,j)\in \Z^2\ :\ i > \tau(n),\ p^n >  j \ge \mu_i \}-\dim\ker V^r|_{M_n}< C,
\end{align*}
where $C$ is some constant independent of $n$.

\begin{figure}[ht]
\includegraphics[height=3in]{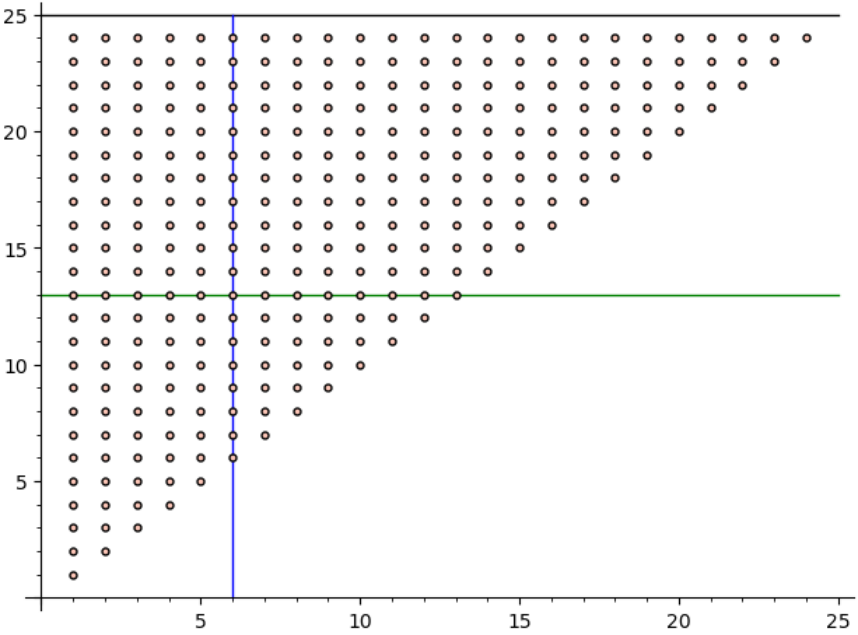}
\caption{
A $k$-basis of $M_2$ when $p=5$, $d=6$, $r=3$.  The lattice point $(i,j)$ represents the basis element $T^j u_i$ for $M_2$.  The vertical cutoff line is $y=t(2)=6$,
 and the horizontal line is $y=p^n-rt(n)(p-1)/d=13$.
 Every point with $i> 6$ and $j\ge 13$ contributes
 contributes to $\dim \ker V^3$, and exactly $42$ of the $72$
 points with $i\le 6$ and $j \ge 13$
 contribute.  No lattice point with $i < 6$ and $j < 13$
 contributes, while whether a lattice point  $(i,j)$
 with $i>6$ and $j<13$ contributes is uncertain.
}
\label{fig:picture 2}
\end{figure}

Figure \ref{fig:picture 2} gives a precise example with $p=5$, $d=6$, $r=3$ and $n=2$.
For $0<i\leq 6$ (i.e. on or left of the vertical line $x=6$) we know that $0<b_i\leq 4$
by \ref{item: gend1} above, so that $25-3b_i \geq 13$
Thus, any lattice point $(i,j)$ corresponding to a basis element $T^j u_i$ of $M_2$
with $0 < i \le 6$ and $25-3b_i \leq j < 25$ lies in the kernel of $V^3$ on $M_2$;
such lattice points all lie on or above the horizontal line $y=13$, though
for any specific value of $i$, we do not know exactly which lattice points 
satisfy this condition, as we do not know the exact value of $b_i$.
Nevertheless, our knowledge of the average values of the $b_i$ in \ref{item: gend2}
above tells us that $\sum_{i=1}^6 3b_i=42$, so that exactly 42 of the 72 lattice points
lying on or left of $x=6$ and on or above $y=13$ (i.e.~the top left region, including the boundary)
contribute to the $k$-dimension of $\ker V^3$ on $M_2$.

On the other hand, any differential in $M_2$ corresponding to  lattice point $(i,j)$ with $0<i \le 6$ and $j < 13$
is sent to something {\em nonzero} by $V^3$, as $j+3b_i < 25$.
It follows that {\em no} lattice point on or left of $x=6$ and below $y=13$
(i.e.~the bottom left region) contributes to $\dim \ker V^3$ on $M_2$.

For $i> 6$, we have $b_i>4$, so when $j\ge 13$ we get $j+3b_i \ge 25$
and the lattice point $(i,j)$ is sent by $V^3$ to $(i, j+3b_i)$,
which lies above $y=25$.  Thus, {\em every}
such lattice point $(i,j)$ corresponding to a basis element of $M_2$
contributes to $\dim \ker V^3$ on $M_2$.
That is, all $150$ lattice points in the top right region (including the bottom boundary $y=13$
but excluding the left boundary $x=6$) contribute to $\dim \ker V^3$,
for a running total of $42+150=192$.

What happens below the horizontal line $y=13$ and to the right of the vertical cutoff line $x=6$ is opaque without knowledge of the exact values of $b_i$. 
There are $21$ lattice points in this region of uncertainty, which yields an upper bound
on $a_2^{(3)}$ of $192+21=213$. 

In fact, the bound is \emph{sharp} in this case: for $p=5$, $d=6$, and $r=3$,
Table~\ref{table: anums} gives $a_2^{(r)}$ for the $\Z_p$-extension of $k(x)$ given by the Artin-Schreier-Witt equation
$(y_0^p,y_1^p,\ldots)-(y_0,y_1,\ldots) = f([x])$ for varying $f\in W(k)[X]$, where $[x]\in W(k(x))$
is the Teichm\"uller lift of $x$; see Remark \ref{rem:ASW}.

\begin{table}[ht]
\centering
\renewcommand{\arraystretch}{1.5}
\captionsetup{justification=centering}
\caption{$a_2^{(3)}$ for $\Z_5$-extensions of $\F_5(x)$}
\begin{tabular}{|c|c|} 
 \hline
  $f$& $a_2^{(3)}$ \\ 
 \hline\hline
 $X^6 + X^4 + 2X^3 + X^2 + X$ & $210$ \\
 \hline
 $X^6 + X^4 + 2X^2$ & $210$ \\ 
 \hline 
 $X^6 + X^3 + X^2 + 3X$ & $211$ \\
 \hline 
 $X^6 + 4X$ & $213$  \\ 
 \hline 
 $ X^6$ & $213$\\
 \hline
\end{tabular}
 \label{table: anums}
\end{table}

\section{Towers of Curves and Iwasawa modules of differentials} \label{s: towers of curves and iwasawa modules of differentials}

\newcommand{\fullystable}{{minimally ramified}} %
\newcommand{\isfullystable}{{has minimal monodromy}} %

We maintain the notation of the introduction, and fix a finite field $k$ of %
cardinality $p^{\fielddegree}$ with prime field $\F \colonequals \F_p$. 
Let $L$ be a $\Z_p$ extension of the rational function field $K = k(x)$ that is {\em geometric} in the sense that $k$ is algebraically closed in $L$.  We fix an isomorphism of topological groups $\Z_p\simeq \Gal(L/K)$, 
and write $\gamma$ for the image of $1$ under this identification. 
We denote by $K_n$ the unique subextension of $L/K$ of degree $p^n$ over $K$,
and by $X_n$ the unique smooth, projective, and
geometrically connected curve over $k$ whose function field is $K_n$.  Thus, $X_0=\PP^1_k$, and 
$\pi:X_n\rightarrow X_0$ is a branched $\Z/p^n\Z$-cover of the projective line.   We call the collection $\{X_n\}_{n\ge 0}$ a {\em $\Z_p$-tower} of curves over $X_0=\PP_k^1$.   We are concerned with $\Z_p$-towers where $L$ is unramified outside $S=\{\infty\}$ and totally ramified there.  In this paper, a $\Z_p$-tower over $\PP^1$ will always satisfy this additional condition even if not explicitly stated.  

\begin{notation} \label{notation:towers}
Given a $\Z_p$-tower of curves over $\PP^1_k$, we write:
\begin{enumerate}[(i)]
    \item $P_n$ for the unique point of $X_n$ over infinity.
    
    \item  $\ord_n$ for the discrete valuation on $K_n$ taking the value 1 on a uniformizer at $P_n$.

    \item  $d_n$ for the unique break in the ramification filtration of the cover $X_{n} \to X_{n-1}$ at $P_n$.

   \item $s_n$ for the $n$th break in the upper numbering ramification filtration of $L/K$ over infinity.
    
\end{enumerate}
\end{notation}

Note that $s_n\in \Z$, and that
 $p\nmid s_1$ and $s_{n+1}\ge p s_n$ with equality if and only if $p|s_{n+1}$ \cite[Chap IV, Exercise 3]{SerreLF}.
Conversely, given {\em any} sequence of integers $\{s_n\}_{n\ge 1}$ satisfying these conditions, there
exists a $\Z_p$-tower over $\PP^1_k$ with $n$th upper ramification break $s_n$ \cite[Proposition 3.3]{Keating:numerology}.
In particular, the sequence of upper ramification breaks in a $\Z_p$-tower over $\PP^1_k$ can grow {\em arbitrarily}
fast.  To prove our main results, we will need to restrict our attention to a certain class
of $\Z_p$-towers with controlled ramification:

\begin{definition}\label{d:minimal-break-ratios}
    A $\Z_p$-tower $\{X_n\}_{n\ge 0}$ over $X_0=\PP_k$ has {\em minimal break ratios} if 
    $s_{n+1}=ps_n$ for all $n\ge 1$, {\em i.e.}~
    $s_{n+1}/s_n$ is as small as theoretically possible.
    In this case, $s_n=d p^{n-1}$ for all $n\ge 1$ and some integer $d >0$ that is not divisible by $p$. We call $d$ the {\em ramification invariant} of the tower.
\end{definition}

\begin{remark}\label{rem: embr}
    It follows from the Riemann--Hurwitz formula
    that any $\Z_p$-tower over $\PP_k^1$ with {\em eventually minimal break ratios},
    {\em i.e.}~which satisfies $s_{n+1}=ps_{n}$ for all $n\gg 0$, is     
    {\em genus stable} in the sense that
    the genus $g_n$ of $X_n$ is given by $g_n = ap^{2n} + bp^n + c$
    for fixed rational numbers $a,b,c$ and all $n$ sufficiently large. 
    Genus stability is a {\em strictly weaker} condition than having eventually minimal break ratios \cite[Theorem 0.1]{KWErrata},
    and there are $\Z_p$-towers over $\PP_k^1$ with $g_n$ a fixed quadratic polynomial in $p^n$ for {\em all} $n\ge 0$ which do {\em not}
    have (eventually) minimal break ratios. It is expected that all towers of `geometric origin' are genus stable \cite{Kramer-Miller-monodromy}.
\end{remark}

\begin{remark}\label{rem:ASW}
 There are an enormous number of $\Z_p$-towers with minimal break ratios, even when the ramification invariant $d$ is fixed.
Indeed, fix a separable closure $K^{\sep}$ of $K\colonequals k(x)$ and $\alpha\in k$ having nonzero trace in $\FF$,
and write $[\cdot]: k(x)\rightarrow W(k(x))$ for the Teichm\"uller section. 
By Artin--Schreier--Witt theory and \cite[Lemma 4.8]{KW},
to each pair $(L,\varphi)$ with $L$ a subextension of $K^{\sep}/K$ unramified outside $S=\{\infty\}$
and $\varphi:\Z_p\simeq  \Gal(L/K) $ an isomorphism of topological groups, there is a unique $c\in \Z_p$
and a unique primitive ({\em i.e.}~nonzero modulo $p$), %
 convergent power series
 $f(X)\colonequals \sum_{i\ge 1, (i,p)=1} c_i X^i \in W(k)[\![X]\!]$ such that $K_n= k(x,y_0,y_1,\ldots, y_{n-1})$ is determined
 by the Witt vector equation
 \begin{equation}
    (y_0^p,y_1^p,\ldots) - (y_0,y_1,\ldots) = c [\alpha] + f([x]).\label{ASWeqn}
 \end{equation}
 Forgetting the identification $\varphi$, we see that the Witt vector $c[\alpha]+ f([x])$
 is determined uniquely by $L$ up to multiplication by $\Z_p^{\times}$. Conversely,
any such equivalence class of pairs $(c,f)$ as above determines a unique $\Z_p$-extension via \eqref{ASWeqn}.
 As $f$ is primitive and convergent, $d\colonequals \deg (f(X) \bmod p)$
is a positive integer, necessarily prime to $p$.  
 Due to \cite[Proposition 3.3]{KW}, the $\Z_p$-extension $L/K$ has minimal break ratios with ramification invariant $d$
 if and only if %
 $
        |c_i|_p \le \frac{d}{i}
 $
 for all $i> 0$. 
\end{remark}

The usual formula expressing the lower ramification breaks in terms of the upper breaks
and the Riemann--Hurwitz formula give (e.g. \cite[\S2.2]{boohercais21}):

\begin{proposition}\label{prop:genusbreaks}
    Let $\{X_n\}$ be a $\Z_p$-tower over $\PP_k^1$ and write $g_n$ for the genus of $X_n$.  Then
    $$d_{n+1}-d_n = (s_{n+1}-s_n)p^n \qquad\text{and}\qquad
        2g_n = 1 - p^n + \sum_{i=1}^n \varphi(p^{n+1-i})d_i.$$ 
    In particular, if ${X_n}$ has minimal break ratios and ramification invariant $d$, then
        $$d_n = d\displaystyle\frac{p^{2n-1}+1}{p+1}\qquad\text{and}\qquad
        2g_n = d\displaystyle\frac{p^{2n}-1}{p+1} +1 - p^n.$$
\end{proposition}

\subsection{Iwasawa modules}\label{ss:iwasawa-modules}

For a $\Z_p$-tower $\{X_n\}$ over $\PP_k^1$, the covering maps $\pi: X_{n+1}\rightarrow X_n$
are finite morphisms of smooth curves, so in particular are finite and flat.  It follows that
$\pi_*\calO_{X_{n+1}}$ is a finite, locally free $\calO_{X_n}$-module, so there is a canonical
{\em trace} morphism $\pi_* : \pi_*\calO_{X_{n+1}}\rightarrow \calO_{X_n}$.
On the other hand, Grothendieck's theory associates to the finite morphism $\pi$
a trace map on sheaves of differentials $\pi_*: \pi_*\Omega^1_{X_{n+1}/k} \rightarrow \Omega^1_{X_n/k}$
which, via Grothendieck--Serre duality, is dual to the usual pullback map on structure sheaves;
likewise, the dual of the trace map on structure sheaves
is pullback of differential forms; see e.g.~\cite[Proposition 4.5]{cais09}.  Using these maps and the canonical
action of $\Z_p$, we will associate to the tower $\{X_n\}$ Iwasawa modules of differentials
and functions whose study is the main object of this paper.

\begin{notation} \label{notation: iwasawa}
We introduce some notation for this Iwasawa theory.
\begin{enumerate}[(i)]
\item  Let  $A \colonequals k\powerseries{T}$ and $A_{\F} \colonequals \F \powerseries{T}$.
\item  We denote the fields of fractions by $E \colonequals \Frac{A}$ and $E_{\F} \colonequals \Frac(A_\F)$.
\item  For $n \geq 0$, let $A_n \colonequals k[T]/(T^{p^n})$.
\item  Let $A\langle x \rangle \colonequals \varprojlim_n A_n[x]$ be the Tate algebra over $A$.
\end{enumerate}
\end{notation}

Note that $A = \varprojlim_n A_n$.

\begin{notation}  Let $\{X_n\}$ be a $\Z_p$-tower over $\PP^1_k$, and fix a coordinate $x$ on $\PP^1_k$.
\begin{enumerate}[(i)]
    \item Let $R_n \colonequals \calO_{X_n}(X_{n}-P_n)$ be the affine coordinate ring of $X_{n}-P_n$;
    in particular, $R_0=k[x]$.

    \item  Write $M_n \colonequals H^0(X_n,\Omega^1_{X_n/k})$ for the space of regular differentials on $X_n$.

    \item  Let $Z_n \colonequals H^0(X_n-P_n, \Omega^1_{X_n/k})$ be the space of differentials on $X_n$ regular outside $P_n$.
    
\end{enumerate}
\label{notation: functions and differentials}
\end{notation}

Via our fixed isomorphism $\Z_p\simeq \Gal(L/K)$ sending $1$ to $\gamma$, we obtain an action
of $\Z_p$ on $R_n$, $M_n$, and $Z_n$, with $p^n\Z_p$ acting trivially.
We identify the completed group ring $\F[\![\Z_p]\!]$
with the power series ring $A_\F = \F[\![T]\!]$ via the map carrying $a\in \Z_p$ to $(1+T)^a$,
and we write 
\begin{equation}
    \xymatrix{
        \chi: \Gal(L/K) \ar[r]^-{\simeq} & \Z_p \ar@{^{(}->}[r] & \F[\![\Z_p]\!]^{\times} \ar[r]^-{\simeq} &  \F[\![T]\!]^{\times}= A_\F ^\times 
        }\label{eq:galoischar}
\end{equation}
for the corresponding (injective) homomorphism of topological groups.  By a slight abuse of
notation, we will frequently view $\chi$ as a continuous character $\chi: \Gal(L/K) \rightarrow A = k\powerseries{T}$.
In this way, 
each of $R_n$, $M_n$, and $Z_n$ become $A_n$-modules, i.e.~an $A$-module whose annihilator is generated by $T^{p^n}$. 
In fact, both $R_n$ and $Z_n$ are naturally $R_0$-modules, and hence are modules over $A\langle x\rangle = R_0 \powerseries{T}$.
These structures are  compatible with the trace mappings, so we may form projective limits.

\begin{definition} \label{definition: limiting objects}
    We write $M\colonequals \varprojlim_n M_n$, with the inverse limit taken via the trace mappings.
    We similarly define $R\colonequals \varprojlim_n R_n$ and $Z\colonequals \varprojlim_n Z_n$,
    again using the trace mappings.  Each of these is naturally a topological $A$-module, via the inverse
    limit topology, and both $R$ and $Z$ are naturally $A\langle x\rangle$-modules.
    We denote by $\sigma$ the unique continuous endomorphism of $A\langle x\rangle$
    that is the $p$-power map on $R_0$, and sends $T$ to $T$; this
    restricts to an automorphism of $A$.
\end{definition}

Passing to inverse limits, the inclusions $M_n\hookrightarrow Z_n$ yield
a continuous  inclusion of $A$-modules
\begin{equation}
    \iota: M\hookrightarrow Z\label{eq:iotamap}
\end{equation}
and we henceforth identify $M$ as a $A$-submodule of $Z$ using $\iota.$

The explicit realization of $\pi_*:R_{n+1}\rightarrow R_n$ as a sum of Galois conjugates (see below) 
makes it clear that $\pi_*$ commutes with the Frobenius endomorphism on each $R_n$; as such,
and keeping in mind the fact that Frobenius commutes with any ring map in characteristic $p$, 
the inverse limit $R$ is equipped with an additive Frobenius map $F$ that 
is $\sigma$-semilinear over $A\langle x\rangle$. Likewise, since the absolute Frobenius morphism $F:X_n\rightarrow X_n$ is finite and flat, 
it induces a trace morphism of abelian sheaves
$F_*: F_*\Omega^1_{X_n/k}\rightarrow \Omega^1_{X_n/k}$ which is dual, via Grothendieck--Serre duality,
to pullback by absolute Frobenius on structure sheaves.  This yields an additive 
map $V: Z_n\rightarrow Z_n$, called the {\em Cartier operator}, that preserves the subspace $M_n$ of regular differential forms and satisfies $V(F(f)\omega)=f V(\omega)$ for any $f\in R_n$.
The Cartier operator is equivariant with respect to the action of the Galois group $\Z_p$, as well as the trace morphisms
$\pi_*$ on differentials; as such, it induces a  $\sigma^{-1}$-semilinear endomorphism $V$ of 
the $A$-module $Z$
that preserves $M$ and satisfies $V(F(f)\omega)=fV(\omega)$ for any $f\in R$ and $\omega\in Z$.
For more on the Cartier operator, see \cite{cartier57,serre58,achterhowe}.

Using the following general fact, we will prove that the structure of $Z$ is particularly simple:

\begin{lemma} \label{lem:torsors}
    Let $\pi: V\rightarrow U$ be a finite \'etale $G$-torsor over $k$, with $G$ a (constant) $p$-group and $U=\Spec(R)$ and $V=\Spec(S)$ smooth of finite type and affine.%
    \begin{enumerate}[(i)]
        \item The canonical trace map $\pi_* : S\rightarrow R$ is surjective.\label{tracesurj}
        \item If $u\in S$ is any element with $\pi_*(u)=1$, then the map 
        \begin{equation*}
                \xymatrix{\psi:R[G]\ar[r] & S}\quad\text{determined by}\quad \psi\left(\sum_{g\in G} r_g g\right) \colonequals  \sum_{g\in G} r_g gu
        \end{equation*}
        is an isomorphism of (left) $R[G]$-modules.\label{torsorfree}
        \item Writing $\varepsilon: R[G]\rightarrow R$ for the augmentation map, the following diagram commutes
        \begin{equation*}
        \xymatrix{
                R[G] \ar[d]_-{\varepsilon} \ar[r]^{\psi} & S \ar[d]^-{\pi_*} \\
                R \ar@{=}[r] & R
        }
        \end{equation*}
        In particular, the kernel of $\pi_*: S\rightarrow R$ is $I_GS$, with $I_G = \ker(\varepsilon)$.     
        \label{torsordiag}
    \end{enumerate}
\end{lemma}

\begin{proof}
    The map $\pi_*$ is described in \cite[\href{https://stacks.math.columbia.edu/tag/03SH}{Tag 03SH}]{stacks-project},
    and it follows from the characterization given there that $\pi_*$ coincides with the $R$-module
    homomorphism $\sum_{g\in G} g $.
    Surjectivity may be checked locally on $U$, where by Nakayama's lemma it is enough to check
    that the induced local trace maps $S\otimes_{R} k(\frakp)\rightarrow k(\frakp)$ are surjective for all $\frakp\in \Spec(R)$.
    Since $\pi$ is finite \'etale, each fiber ring $S\otimes_R k(\frakp)$ is a finite disjoint union
    of finite, separable extensions of $k(\frakp)$  \cite[\href{https://stacks.math.columbia.edu/tag/02G7}{Tag 02G7}]{stacks-project}, and the local trace map is simply the sum
    of the field-theoretic trace maps over these finite separable extensions.  We are thereby reduced to 
    checking surjectivity of trace in the case that $S=L$ and $R=K$ are fields, where it is well-known to be equivalent to 
    the separability of $L/K$.

    The map $\psi$ of \ref{torsorfree} is easily seen to be a map of left $R[G]$-modules, 
    so to check that it is an isomorphism we may work locally on $U$, which reduces us to the case
    that $R$ is a finite, separable extension of $k$ and $S$ is a finite \'etale $R$-algebra of rank $|G|$. 
    In this case, the map $\psi$ is a homomorphism of $R$-vector spaces of the same dimension, 
    so it suffices to prove it is injective, which amounts to 
    the $R$-linear independence of the set $\{gu\}_{g\in G}$.
    To see such independence, consider an $R$-linear relation
    $
        \sum_g r_g gu = 0.
    $
    Applying each $h\in G$ to this relation yields the equation $Av=0$
    where $v=(r_g)\in R^{|G|}$ and $A$ is the $|G|\times |G|$ matrix with $A_{ h g} = hgu $.
    Now the matrix $A$ is circulant and $|G|$ is a power of $p$, so by the usual formula for the determinant of a circulant matrix 
    (which holds in the ``universal'' case by standard arguments)
    one has
    $
        \det(A) = %
        \pi_*(u)^{|G|} = 1,
    $
    and it follows that left multiplication by $A$ is injective on $R^{|G|}$, whence
    $v=0$ as desired.  
    Finally, one checks readily that the diagram in \ref{torsordiag}
    commutes; as $\psi$ is an isomorphism of $R[G]$-modules, 
    we then have $\ker(\pi_*)=\psi(I_G)=I_G S$.
\end{proof}

For $m>n$, the covering map $\pi: X_{m}\rightarrow X_n$ is ramified only over $P_n$, so the restriction
$\Spec(R_{m})\rightarrow \Spec(R_n)$ is an \'etale $\Z/p^{m-n}\Z$-torsor.
Thanks to Lemma \ref{lem:torsors} \ref{tracesurj}, there exists $u=\{u_n\}_{n\ge 0}\in R$
 with $u_0=1$, and we fix any such element.

\begin{proposition}\label{prop: control type results}
    Let $\{X_n\}$ be a $\Z_p$-tower over $\PP_k^1$ and $u=\{u_n\}$ as above.
    \begin{enumerate}[(i)]
            \item The map $R\rightarrow Z$ given by $(f_n)_{n\ge 0}\mapsto (f_n dx)_{n\ge 0}$
            is an isomorphism of $A\langle x\rangle$-modules.\label{fndiffisom}

            \item For each $n$, projection to level $n$ yields isomorphisms of $A\langle x\rangle$-modules 
            $R/T^{p^n}R \simeq R_n$ and $Z/T^{p^n}Z\simeq Z_n$
            that are compatible with the identification of \ref{fndiffisom}.\label{control}

            \item The map $\psi: A\langle x\rangle\rightarrow R$ given by $\psi(f)\colonequals f\cdot u$ is an isomorphism
            of $A\langle x\rangle$-modules.\label{uisom}

            \item The image of the composite map
            \begin{equation}
                \xymatrix{
                        M \ar@{^{(}->}[r]^-{\iota} & Z \ar@{->>}[r] & Z / T^{p^n}Z \simeq Z_n 
                }\label{eq:iotalimitmap}
            \end{equation}
            is equal to $M_n$.\label{regularcontrol}
    \end{enumerate}
\end{proposition}

\begin{proof}
    Since $R_{n}$ is an \'etale $R_0$-algebra, we have $Z_n=\Omega^1_{R_n/k} = R_n \cdot \Omega^1_{R_0/k} = R_n\cdot dx$
    as $R_0=k[x]$: passing to inverse limits gives \ref{fndiffisom}.  
   Under our fixed identification $A\simeq k[\![\Gal(L/K)]\!]$,
   the augmentation ideal of the quotient $k[\Gal(K_m/K_n)]$ is generated by $T^{p^n}$,
   for all $m\ge n$, and \ref{control} follows easily from 
   Lemma \ref{lem:torsors} \ref{tracesurj} and \ref{torsordiag}.
   To prove \ref{uisom}, it suffices to do so modulo $T^{p^n}$ for each $n$, where (using \ref{control})
   it reduces to an instance of Lemma \ref{lem:torsors} \ref{torsorfree}.
   Finally, the very definition of \eqref{eq:iotamap} as an inverse limit 
   shows that \eqref{eq:iotalimitmap} factors through the projection $M\rightarrow M_n$,
   which is {\em surjective} since each transition map $\pi_*:M_{m+1}\rightarrow M_m$ is surjective,
   due to the fact that the cover $\pi:X_{m+1}\rightarrow X_m$ is totally ramified (see \cite[Proposition 2.26]{ID}).
   It follows that \eqref{eq:iotalimitmap} has image precisely $M_n$.
\end{proof}

\begin{notation}\label{dxabuse}
If $f=(f_n)_{n\ge 0} \in R$, then by a slight abuse of notation, we will write simply $f\cdot dx$
for the element $(f_n dx)_{n\ge 0}$ of $Z$ corresponding to $f$ under the identification
\ref{fndiffisom}. 
\end{notation}

\begin{remark}
 Although projection induces a surjection $M / T^{p^n} M\twoheadrightarrow M_n$,
it is {\em far} from an isomorphism: by Proposition~\ref{prop:omega as T module}
below, $M / T^{p^n} M$ is infinite dimensional as a $k$-vector space.
Said differently, the $A$-submodule of $Z$ given by $M$
is far from being $T$-adically saturated.
\end{remark}

Using Lemma~\ref{lem:torsors} \ref{uisom}, we make the following definition:

\begin{definition}\label{alphadef}
    For $u\in R$ with $u_0=1$, the {\em Frobenius element} corresponding to $u$
    is the unique $\alpha\in A\langle x\rangle$ determined by $Fu = \alpha\cdot u$.
\end{definition}

\begin{remark}\label{rem:alphaunit}
    As $Fu$ {\em also} has $Fu_0=1$, it follows from Lemma \ref{lem:torsors} \eqref{uisom}
    that $\alpha\in 1+TA\langle x\rangle$.  If $u' \in R$ also has $u'_0=1$
    and corresponding Frobenius element $\alpha'$, then  $u' = cu$
    for a unique $c\in 1+TA\langle x\rangle$, and we have $\alpha' = c^{-1}\sigma(c)$.
\end{remark}

\begin{notation} \label{notation:L}
Let $L \colonequals x A\langle x\rangle $.  It is a principal ideal of $A\langle x\rangle $, and multiplication by $x$ gives an isomorphism of $A\langle x\rangle$-modules
$A\langle x\rangle \xrightarrow{\simeq} L = xA\langle x\rangle$.
\end{notation}

\begin{definition}
 We denote by $V$
the unique continuous and $\sigma^{-1}$-linear $A$-module endomorphism of $L$ determined by
\begin{equation}
    V(x^{i+1})\colonequals  \begin{cases} x^{\frac{i+1}{p}} & \text{if}\ i\equiv -1\bmod p \\
    0 & \text{otherwise}
    \end{cases}\label{eq:VonL def}
\end{equation}
For $\beta\in A\langle x\rangle$, we will write $V_{\beta}$ for the endomorphism given by $V_{\beta}(f)\colonequals V(\beta f)$.
\end{definition}

\begin{corollary}\label{cor:rankone}
    Let $u\in R$ with $u_0=1$ have corresponding Frobenius element $\alpha$.
    The map
    \begin{equation} 
                {\eta: L} \rightarrow Z\quad\text{given by}\quad \eta(f)\colonequals f\cdot u \frac{dx}{x}\label{etaisom}
    \end{equation}
    is an isomorphism of $A\langle x\rangle$-modules.  It satisfies the
    intertwining relation $V\circ \eta = \eta \circ V_{\alpha^{-1}}$.
\end{corollary}

\begin{proof}
    That \eqref{etaisom} is an isomorphism of $A\langle x\rangle$-modules
    follows immediately from Proposition \ref{prop: control type results}.
    Now for any $f\in A\langle x\rangle$ we compute %
    $$
        V(f u \frac{dx}{x}) = V(f \alpha^{-1} F(u) \frac{dx}{x}) = u V(\alpha^{-1} f \frac{dx}{x}).
    $$
    As $V$ is additive and commutes with $T$, the claimed
    intertwining relation then amounts to the assertion that the Cartier operator on $Z_{0}$
    carries $x^{j} dx$ to $x^{\frac{j+1}{p}-1} dx$ when $j\equiv -1\bmod p$, and to $0$ otherwise,
    which is standard.
\end{proof}

To each finite place $v$ of $K$ with corresponding maximal ideal $\frakm_v\subset R_0$
is associated the usual ``evaluate at $v$'' homomorphism of rings 
$\ev_v: R_0\twoheadrightarrow R_0/\frakm_v=\kappa(v)$, with $\kappa(v)$ the residue field of $v$.
This admits a unique continuous extension to a homomorphism of rings
$A\langle x\rangle\rightarrow \kappa(v)[\![T]\!]$ that we again denote by $\ev_v$.
For $f\in A\langle x\rangle$, we often write $f(v)$ in place of $\ev_v(f)$.

\begin{proposition}\label{prop: Frobenius }
  Let $\chi: \Gal(L/K)\rightarrow A_\F^\times$ be as in \eqref{eq:galoischar}, and $v$ a finite place of $K$ of degree $d$ with corresponding Frobenius element $\Frob_v$ in (the abelian!) $\Gal(L/K)$.  Then 
  \begin{equation}
    \chi(\Frob_v) = \prod_{i=0}^{d\nu-1} \sigma^i(\alpha)(v).\label{eq:Frobalpha}
  \end{equation}
\end{proposition}

\begin{proof}
    Note that $R_n \otimes_{R_0} \kappa(v)$ is isomorphic to the product of residue fields extensions of all places of $K_n$ lying over $v$. Thanks to Lemma \ref{lem:torsors},
    we have the commutative diagram of $A\langle x\rangle$-modules
    \begin{equation*}
        \xymatrix@C=50pt{
            R_0[T]/(T^{p^n}) \ar@{->>}[d]_-{\ev_v} \ar[r]^-{\simeq}_-{f\mapsto f\cdot u_n} & R_n \ar@{->>}[d]^-{} \\
            \kappa(v)[T]/(T^{p^n}) \ar[r]^-{\simeq}_-{f\mapsto f\cdot u_n(v)} & R_n \otimes_{R_0} \kappa(v)
        }
    \end{equation*}
    whose horizontal maps are isomorphisms.
    Writing $\chi_n:\Gal(K_n/K)\rightarrow k[T]/(T^{p^n})$ for the character induced from $\chi$
    by reduction modulo $T^{p^n}$, the $A\langle x\rangle$-equivariance of the bottom identification
    implies that, for any $g\in \Gal(K_n/K)$, multiplication by $\chi_n(g)$ on $\kappa(v)[T]/(T^{p^n})$ is intertwined by the action of $g$ on $R_n \otimes_{R_0} \kappa(v)$.  In particular,
    multiplication by $\chi_n(\Frob_v)$ is intertwined with $\Frob_v=F^{d\nu}$ on $R_n \otimes_{R_0} \kappa(v)$,
    where $F$ is the $p$-power map.  On the other hand, the vertical arrow $R_n \twoheadrightarrow \kappa(w)$
    is a homomorphism of characteristic $p$ rings, so the endomorphism $F^{d\nu}$ of $R_n$
    lifts $F^{d\nu}$ on $R_n \otimes_{R_0} \kappa(v)$.  By Definition \ref{alphadef},
    the top horizontal identification intertwines $F^{d\nu}$ on $R_n$
    with multiplication by $\prod_{i=0}^{d\nu-1} \sigma^i(\alpha)$ on $R_0[T]/(T^{p^n})$.
    Commutativity of the diagram then yields \eqref{eq:Frobalpha} modulo $T^{p^n}$,
    and passing to the limit completes the proof.
\end{proof}

\subsection{Artin-Schreier-Witt Theory} \label{ss:towers ASW}
Artin-Schreier-Witt theory gives an explicit description of $\Z_p$-towers $\{X_n\}_{n \geq 0}$.
Letting $W(K_0)$ denote the Witt-vectors of $K_0$, there exists $w \in W(K_0)$ such that $K_n=K_0(y_0,\ldots,y_{n-1})$
is determined by the Witt vector equation
\[
(y_0^p,y_1^p,\ldots) - (y_0,y_1,\ldots) = w.
\]
Making the Witt vector subtraction explicit yields layer-by-layer descriptions $K_{n+1} = K_n(y_n)$ where $y_n^p-y_n=b_n$ for some
$b_n \in K_n$ uniquely determined up to addition of elements of the form $a^p-a$ for $a\in K_n$.  
Fixing a $\Z_p$-tower $\{X_n\}$ over $X_0=\PP_k^1$, we have $K_0=k(x)$.

\begin{definition}
    We say that the equation $y_n^p - y_n = b_n$ presents the branched cover $X_{n+1} \to X_{n}$ in standard form if $b_n\in R_n=H^0(X_n-P_n,\calO_{X_n})$ has $\ord_n(b_n) = -d_{n+1}$, and $\gamma^{p^n}y_n=y_n+1$.
    The tower is presented in standard form if each level is presented in standard form.
\end{definition}

Writing $K_{n+1} = K_n(y'_n)$ with $(y'_n)^p - y'_n = b'_n \in K_n$, standard arguments show 
$\ord_n(b'_n) \leq -d_{n+1}$ with equality if and only if $p \nmid \ord_n(b'_n)$;  thus standard form minimizes the order of the pole of $b'_n$.

\begin{proposition} \label{proposition:standardformB}
    For all $n$, there exists 
    $u_n\in K_{n}$ such that $b_n' \colonequals b_n + u_n^p - u_n$ lies in $R_n$
    and $b'_n$ has pole at $P_n$ of order prime to $p$.
\end{proposition}

\begin{proof}
    We claim that $d_{n} > 2p g_{n-1}$ for all $n$.  This estimate requires working over the projective line; to clarify this suppose for a moment 
    that our base curve $X_0$ had genus $g_0$.  Expressing $d_n$ and $g_{n-1}$ in terms of the upper breaks 
    (e.g.~ via Proposition \ref{prop:genusbreaks}) and remembering the fundamental inequality $s_n \geq p^{n-i} s_i$, we estimate
   \begin{align*}
d_{n} - 2p g_{n-1} &\geq  p^{n-1}s_n - \sum_{i=1}^{n-1} \varphi(p^i) \frac{s_n}{p^{n-i}} - \sum_{i=1}^{n-1} p \varphi(p^i)\left(\frac{s_n}{p^{n-i}} +1\right)   - 2p - p^n (2 g_0 -2) \\
& \geq s_n \left ( p^{n-1} - \sum_{i=1}^{n-1} (p^2-1) p ^{2j-1-n} \right)  - \sum_{i=1}^{n-1} p^i (p-1)  - 2p - p^n(2g_0-2)\\
& \geq s_n p^{-(n-1)} - (p^n -p) - 2p - p^n(2g_0-2) \\
& \geq s_n p^{-(n-1)}  - p - p^n(2g_0-1).
   \end{align*} 
   If this is strictly positive, then a straightforward application of Riemann-Roch gives functions to modify the local expansion of $b_n$ at $P_n$ as in \cite[Lemma 7.3]{garnek}.  When $g_0 =0$, it is clear that $d_{n} - 2p g_{n-1} >0$ but this may not hold for all $n$ in other situations.
   \end{proof}
    
\begin{corollary}\label{cor:fninRn}
There exist $f_n \in R_n$, $n\geq 1$, with $y_n^p-y_n = f_n$ presenting $\{X_n\}$ in standard form.
When this is the case, we have $y_n\in R_{n+1}$ and $\ord_{n+1}(y_n)=-d_{n+1}$.
\end{corollary} 

\begin{proof}
    The existence of $f_n\in R_n$ satisfying $\ord_{n}(f_n)=-d_{n+1}$ and $K_{n+1}=K_n(y_n)$
    with $y_n^p-y_n=f_n$ follows immediately from Proposition \ref{proposition:standardformB}.
    As $\gamma^{p^n}$ generates $\Gal(K_{n+1}/K_n)$, by Artin--Schreier theory we have $\gamma^{p^n}y_n=y_n+c_n$
    for some $c_n\in \F^{\times}$.
    Replacing $y_n$ with $c_n^{-1}y_n$ and $f_n$ with $c_n^{-1}f_n$
    then presents $X_{n+1}\rightarrow X_n$ in standard form.
   The final assertions about $y_n$ follow from the equation $y_n^p-y_n=f_n$, and the fact that 
    $f_n\in R_n$ has $\ord_n(f_n)=-d_{n+1}$.
\end{proof}

From now on fix $(f_0,f_1,\cdots,)$ and $(y_0,y_1,\cdots)$ presenting the tower $\{X_n\}$ in standard form.  

\begin{proposition}\label{prop:regular}   
   There is an equality $R_n = k[x,y_0,\ldots,y_{n-1}]$ of subrings of $K_n$.
\end{proposition}
To prove Proposition \ref{prop:regular}, we first require:

\begin{lemma} \label{lemma:tracecomputation}
Let $\pi : Y \to X $ be an Artin-Schreier cover of curves over $k$ corresponding to an extension of function fields given by adjoining a root of $y^p - y = f$.  Then
\[
\pi_*(y^i) =   \begin{cases}
-1 & \textrm{if }i = p-1\\
0 & \textrm{if } 0\leq i < p-1
\end{cases}.
\]
\end{lemma}

\begin{proof}
The Galois group of $k(Y)/k(X)$ is generated by the automorphism sending $y$ to $y+1$ and fixing $k(X)$.
 Using the description of trace as sum of Galois conjugates (as in the proof of Proposition \ref{lem:torsors}), we compute  
\[
\pi_*\left(y^{i} \right) = \sum_{a=0}^{p-1} (y+a)^{i}   = \sum_{j=0}^i \left(\binom{i}{j}\sum_{a=0}^{p-1} a^{i-j}\right)y^j.
\]
The sum over $a$ is preserved by multiplication by $\alpha^{i-j}$ for any $\alpha\in \FF^{\times}$,
so is zero unless $i=p-1$ and $j=0$, where it is equal to $-1$ by Fermat's little theorem.
\end{proof}

\begin{proof}[Proof of Proposition \ref{prop:regular}]
    As the presentation is
    in standard form, we know that $y_i\in R_{i+1}\subseteq R_n$ for $i < n$ thanks to Corollary \ref{cor:fninRn}, and clearly $x\in R_0\subseteq R_n$,
    so one containment is clear.  To establish the reverse containment, we proceed
    by induction on $n$,
    with base case $X_0=\PP^1$ obvious.
    Any $g\in K_{n+1}$ may be written
    \begin{equation}
    g = b_0 + b_1 y_n + \cdots + b_{p-1}y_n^{p-1}\label{gexp}
    \end{equation}
    for functions $b_0,b_1,\ldots, b_{p-1} \in K_{n}$.  If $g\in R_{n+1}$, then since $y_n\in R_{n+1}$,
    the same is true of
    $y_n^i g$ for $i=0,\ldots, p-1$.  The trace mapping $\pi_*$
    carries $R_{n+1}$ into $R_n$, so $\pi_*(y_n^i g)\in R_n$ for all $i$. 
    Applying Lemma~\ref{lemma:tracecomputation} to \eqref{gexp}, we compute
    $$
        \pi_*(y_n^i g) = \begin{cases}
            -b_{p-1-i} & \textrm{if } 0 \le i < p-1 \\
            -b_0-b_{p-1} & \textrm{if } i = p-1.
            \end{cases}
    $$
    We conclude that $b_{i} \in R_{n}$ for all $i$, so by
    our inductive hypothesis must be a polynomial in $x,y_0,\ldots, y_{n-1}$.  
    It then follows from equation \eqref{gexp} that $g \in k[x,y_0,\ldots, y_n]$, as desired. 
\end{proof}

\begin{remark}
In particular, note that $k[x,y_0,\ldots, y_{n-1}] \subset K_n$ is Galois stable.
\end{remark}

\begin{definition}\label{not:basep}
For a nonnegative integer $a$, we write $\{a_i\}_{i\ge 0}$ for the sequence of $p$-adic digits of $a$; {\em i.e.}~$a = \sum_{i\ge 0} a_i p^i$ with $0 \leq a_i <p$ for all $i$. We define $\rev(a)\colonequals\sum_{i\ge 0} p^{-i-1}a_i$, and put
 $$\y^a \colonequals \prod_{i\ge 0} y_i^{a_i}
\qquad\text{and}\qquad \xi_a \colonequals \sum_{i \ge 0} p^{-i-1} a_{i} d_{i+1} .
$$
When $a<p^n$, we have $\y^a \in R_{n}$ and $p^n\xi_a\in \Z$.
\end{definition}

\begin{corollary} \label{cor:valuations}
For $a < p^{n}$ we have 
\[
\ord_{n}(\y^a) = -p^n\xi_a.
\]
\end{corollary}

\begin{proof}
For $i<n$ we have $\ord_{n}(y_i) = p^{n-i-1}\ord_{i+1}(y_i) = -p^{n-1-i}d_{i+1}$
by Corollary \ref{cor:fninRn}.
\end{proof}

\begin{lemma}\label{lem:diffequiv}
    The function $a\mapsto p^n\rev(a)$ is an involution on $[0,p^n)\cap \Z$.
    For $a<p^n$, we have 
    \begin{equation}
        p^n\xi_a \equiv d_n \cdot p^n\rev(a)\bmod p^n\label{pncong}
    \end{equation}
    In particular, the set of integers $\{p^n\xi_a\ :\ 0\le a < p^n\}$ is a complete set of representatives of $\Z/p^n\Z$.
\end{lemma}

\begin{proof}
    If $a<p^n$, we have $a_i=0$ for $i \ge n$ and $p^n\rev(a) \in \Z$ is the integer obtained by reversing the
    base-$p$ digits of $a$; the first statement follows immediately.  
    For the second, we observe that $d_{i+1} \equiv d_i \bmod p^i$ by Proposition \ref{prop:genusbreaks} and the fact that
    the upper breaks $s_i$ are integers.  It follows from this that $p^{n-1-i}a_id_{i+1} \equiv p^{n-1-i}a_i d_n\bmod p^n$
    for $i < n$, which yields \eqref{pncong}.  Finally, we have $d_i\equiv d_1 = s_1 \bmod p$, and $p\nmid s_1$
    so $d_i$ is a unit modulo $p^n$ for all positive integers $i$ and $n$, and the final statement follows from the first.
\end{proof}

\begin{corollary}\label{cor:valuationsdistinct}
    If $g\in R_n$ with $\displaystyle g = \sum_{a=0}^{p^{n}-1} g_a \y^a$ for $g_a \in k[x]$,
    then 
    \[
    \ord_n(g) = \min_{a} \left \{ \ord_n(g_a \y^a)  \right \} = -p^n\max_a\left\{ \deg(g_a) + \xi_a\right\}.
    \]
    Moreover, the set $\{\y^a\ :\ 0\le a < p^n\}$ is a $k[x]$-module basis of $R_n$.
\end{corollary}

\begin{proof}
    By Corollary~\ref{cor:valuations} and Lemma \ref{lem:diffequiv},
    the integers $\ord_n(g_a\y^a)=p^n \deg(g_a) + \ord_n(\y^a)$ for $0\le a < p^n$ are distinct modulo $p^n$, and therefore distinct, which gives the first statement and the $k[x]$-linear independence
    of the set $\{\y^a\ :\ 0\le a < p^n\}$.  That this set also spans follows from an inductive argument using Proposition \ref{prop:regular}  and the equations $y_i^p - y_i = f_i$, with $f_i\in R_{i}$.
\end{proof}

\subsection{Functions and differentials in the limit} \label{ss: functions and differentials in the limit}  
Next we turn to writing down explicit elements of $R$ and $Z$, 
which will enable us to describe $M$ as a $A$-module in Proposition \ref{prop:omega as T module}.

\begin{definition} \label{defn:omegab}
For nonnegative integers $n$ and $a$ with $a < p^n$, we define
\[
w_a^{(n)} \colonequals (-1)^n \y^{p^n-1-a} = (-1)^n \prod_{i=0}^{n-1} y_i^{p-1-a_i}
\]
where $\{a_i\}_{i\geq 0}$ are the $p$-adic digits of $a$.
Note that $w_0^{(0)}=1$ and $w_a^{(n)}\in R_n$. For $a \geq p^n$ we define $w_a^{(n)}$ to be zero.
\end{definition}

\begin{lemma} \label{lem:trace}   
Let $\pi_*: R_{n+1}\rightarrow R_n$ be the trace map.  For $0\le a< p^{n+1}$ we have $\pi_*(w_a^{(n+1)})=w_a^{(n)}$.

\end{lemma}

\begin{proof}
Write $a=a_0+a_1p+\cdots+a_np^n$ as per Definition \ref{not:basep} and let $a'=a-a_np^n$. We have $w_a^{(n+1)} = -w_{a'}^{(n)}y_n^{p-1-a_n}$ with $w_{a'}^{(n)}$ in $R_n$.  
Since $\pi_*: R_{n+1}\rightarrow R_n$ is a map of $R_n$-modules we have $\pi_*(w_a^{(n+1)}) = - w_a^{(n)} \pi_*(y_n^{p-1-a_n})$. The lemma follows from Lemma \ref{lemma:tracecomputation}.
\end{proof}

\begin{definition}\label{def:wa}
For each non-negative integer $a$, we define $w_a$ as the element of $R$ whose projection
onto $R_n$ is $w_a^{(n)}$, i.e., we have
$$
    w_a\colonequals (\cdots , w_a^{(2)}, w_a^{(1)} , w_a^{(1)}).
$$
\end{definition}

\begin{remark}\label{omegaconverge}
    The element $w_0\in R$ is a viable choice of $u$ as in %
    Corollary \ref{cor:rankone}.
    Note that $w_a\rightarrow 0$ in the inverse limit topology
    as $a\rightarrow \infty$, as the projection of $w_a$ to $R_n$ is zero
    when $a\ge p^n$; equivalently, thanks to Proposition \ref{prop: control type results} \ref{fndiffisom},
    we have $w_a \in T^{p^n}{R}$ for $a\ge p^n$.
\end{remark}

\begin{proposition} \label{cor:basis omega_n}
A $k$-basis for $M_n$ is
\[
\left \{ x^\nu w^{(n)}_a \frac{dx}{x}\ :\ 0 \leq a \leq p^n-1\ \text{and}\ 1 \leq \nu <  \xi_a  \right\}.
\]

\end{proposition}

\begin{proof}
    This goes back at least to Boseck~\cite{boseck} (c.f. \cite[Lemma 5]{madden}
    and \cite[pp. 112--113]{valentinimadan81}); we sketch a proof for convenience of the reader.
    For $a<p^n$ we have $\ord_n(w_a^{(n)}) = -p^n \xi_{p^n-1-a}$ by Corollary \ref{cor:valuations},
    and $\ord_n(dx) = 2g_n-2$ where $g_n$ is the genus of $X_n$; using Proposition \ref{prop:genusbreaks} and simplifying then gives
    \begin{align}
        \ord_n\left (x^{\nu}w_a^{(n)}\frac{dx}{x}\right ) 
        = -1  + p^n (\xi_a - \nu).\label{basisorder}
    \end{align}
    It follows from Lemma \ref{lem:diffequiv} that the association $(a,\nu)\mapsto \ord_n\left (x^{\nu}w_a^{(n)}\frac{dx}{x}\right ) $ is one to one; in particular, the set $B\colonequals\{x^{\nu} w_a^{(n)}\frac{dx}{x}\ :\ 0\le a<p^n\ \text{and}\ \nu\ge 1\}$ is $k$-linearly independent.
    On the other hand, by Proposition \ref{prop: control type results} \ref{fndiffisom} and Corollary \ref{cor:valuationsdistinct},
    the collection of $xw_{a}^{(n)}\frac{dx}{x}$ for $0\le a < p^n$ spans $Z_n$ as a $k[x]$-module, whence 
    $B$ spans $Z_n$ as a $k$-vector space.  Letting $\omega=\sum_{a,\nu}  b_{a,\nu} x^{\nu}w_{a}^{(n)}\frac{dx}{x}$ be any element of $Z$, we have (cf. Corollary \ref{cor:valuationsdistinct})
    \[
        \ord_n(\omega) = \min_{a,\nu} \ord_n( b_{a,\nu} x^{\nu}w_a^{(n)}\frac{dx}{x}),
    \]
    so $\omega$ is regular if and only $x^{\nu}w_a^{(n)}\frac{dx}{x}$ is regular whenever $b_{a,\nu}\neq 0$.
    As $\nu$ is an integer and $\xi_a$ is never integral when $a>0$, one checks that \eqref{basisorder} is nonnegative if and only if $\nu < \xi_a$, and the result follows.
\end{proof}

\begin{lemma}\label{lem:Taunit}
    For each integer $a=a_0+a_1p+\cdots+a_np^n >0$ as in Definition \ref{not:basep}, $T^a \y^a = \prod_{i\ge 0} a_i!$. In particular, $T^a\y^a\in \F^{\times}$.
\end{lemma}

\begin{proof}
    We first treat the special case $a=a_mp^m$ with $0\le a_m < p$.
    As the tower is presented in standard form, 
    we have $T^{p^m}y_m=1$ for all $m >0$.  Inductively assuming that $T^{ip^m} y_m^i=i!$ for $i<b< p$, we have
    $$
        T^{bp^m}y_m^b = T^{(b-1)p^m}((y_m+1)^b-y_m^b)=T^{(b-1)p^m}\sum_{i=1}^b \binom{b}{i} y_m^{b-i}.
    $$
    On the other hand, $T^{p^m}$ kills $k[x]$, so our inductive hypothesis implies that $T^{(b-1)p^m}$ kills $y_m^i$ for $i< b-1$,
    and the sum above reduces to $T^{(b-1)p^m} b y_m^{b-1} = b\cdot (b-1)!=b!$ by our inductive hypothesis.

    Now assuming that the result holds for $a'<p^m$, let $a=a'+a_mp^m$.  As $T^{p^m}$ is zero on $K_m$ by definition, 
    it is $K_m$-linear as an endomorphism of $K_{m+1}$. Thus
    $$
        T^a \y^a = T^{a'} T^{a_mp^m} \y^{a'} y_m^{a_m} =
        T^{a'} \y^{a'} T^{a_mp^m} y_m^{a_m} = T^{a'} \y^{a'} a_m!,
    $$
    which gives the claimed result by induction.
\end{proof}

Recalling Remark~\ref{omegaconverge}, note that Proposition \ref{prop: control type results} \eqref{uisom} gives an isomorphism of
$A\langle x\rangle$-modules
\begin{equation}
    \xymatrix{
        {A\langle x\rangle\ar[r]^-{\simeq}} & R
    }\quad\text{given by}\quad 
    f\mapsto f\cdot w_0.
    \label{psirecall}
\end{equation}

\begin{definition}\label{def:rhoa}
    For each integer $a>0$, we write $\rho_a $ for the element of $A\langle x\rangle$ corresponding to $w_a\in R$ under the identification \eqref{psirecall}.
\end{definition}

\begin{lemma}\label{lemdef:ra}
    We have $\rho_a = T^ar_a$ with $r_a\in A\langle x\rangle$. Furthermore, $r_a \bmod T$ lies in $k^{\times}$.
\end{lemma}

\begin{proof}
    By definition, the projection of $T^{p^n-1-a}\rho_a w_0=T^{p^n-1-a}w_a$ to $R_n$
    is $(-1)^nT^{p^n-1-a} \y^{p^n-1-a}$, which lies in $k^{\times}$ thanks to Lemma \ref{lem:Taunit}.  It follows that $T^{p^n-a}\rho_a$ projects to zero in $R_0[T]/(T^{p^n})$, whence $\rho_a \in T^a A\langle x\rangle.$ Writing $r_a$ for the unique element of $A\langle x\rangle$ with $\rho_a=T^a r_a$, we know from the above that
    $T^{p^n-1}r_a $ projects to an element of $k^{\times}$ in $R_0[T]/(T^{p^n})$, so $r_a\bmod T$ lies in $k^{\times}$. 
\end{proof}

\begin{definition} 
\label{defn:muj}
\begin{enumerate}[(i)]
\item  For a nonnegative  integer $i$ define
\[
\mu_i \colonequals \min \{ b \ge 1 : \xi_b > i \} \quad \text{and}\quad e_i \colonequals x^i w_{\mu_i}  \frac{dx}{x} = T^{\mu_i}r_{\mu_i}x^i w_0 \frac{dx}{x}.
\]

\item  Let $i(n)$ be the largest integer such that $\mu_{i(n)} < p^n$.
\item \label{defn:Deltan} For a non-negative integer $n$ define 
\begin{align*}
\Delta_n \colonequals \left \{ ( i,j ) \in \Z^2 : i \geq 1 \, \text{ and } \, \mu_i \leq j < p^n \right\} 
=  \left \{ (i,j) \in \Z^2 :  0 \leq j < p^n \, \text{ and } \, 1 \leq i < \xi_j \right\}.
\end{align*}

\item  Define
\[
\Delta_\infty  \colonequals \left \{ ( i,j ) \in \Z^2 : i \geq 1 \, \text{ and } \, \mu_i \leq j  \right\} .
\]
\end{enumerate}
\end{definition}

The equivalence of the two descriptions of $\Delta_n$ follows from the definition of $\mu_i$,
while the asserted equality in the definition of $e_i$ is Definition \ref{def:rhoa} and Lemma \ref{lemdef:ra}.
Note that for $i > i(n)$ we have $e_i \in T^{p^n} Z$, 
so that $e_i \rightarrow 0$ in the inverse limit topology on $Z$;
as such, any infinite sum $\sum_{i> 0} c_i e_i$ with $c_i\in A$ converges
in $Z$.
Given $\omega \in Z$, to simplify notation we let $\overline{\omega}$ denote the image of $\omega$ in $Z_n \simeq Z / T^{p^n} Z$ when $n$ is clear from context.

\begin{proposition} \label{prop:omega as T module}
For each positive integer $n$, the set $\{\overline{e}_1, \overline{e}_2, \ldots, \overline{e}_{i(n)}\}$
is a minimal set of generators for $M_n$ as an $A_n =k[T]/(T^{p^n})$-module, and
\[
    M_n = \bigoplus_{i=1}^{i(n)} k \powerseries{T}/ (T^{p^n} ) \cdot \overline{e}_i \simeq  \bigoplus_{i=1}^{i(n)} k \powerseries{T}/ (T^{p^n-1-\mu_i}).
 \]
 Furthermore, the map $\displaystyle \prod_{i \geq 1} k\powerseries{T} \to M$ sending $(a_i)_{i \geq 1}$ to $\sum_{i\ge 1} a_ie_i$ is an isomorphism of $A=k\powerseries{T}$-modules (note the sum converges in $Z$ by Remark \ref{omegaconverge}.)
\end{proposition}

\begin{proof}
First note that $\overline{e}_i \in M_n$ and $e_i \in M$ by Proposition~\ref{cor:basis omega_n} and the definition of $\mu_i$.
We claim
\[
\mathscr{B}_n = \left \{ T^j \overline{e_i} : 1 \leq i \leq i(n), \, 0\leq j \leq p^n-1-\mu_i \right \}
\]
is a $k$-basis for $M_n$.  We have $\mathscr{B}_n\subseteq M_n $ as $M_n$ is stable under the action of $A$.
Suppose that
\[
\sum_{i=1}^{i(n)} \sum_{j=0}^{p^n-1-\mu_i} c_{i,j}T^j \overline{e}_i =0
\]
with $c_{i,j} \in k$.  By Proposition \ref{prop: control type results}, we must then have
\begin{equation}
\sum_{i=1}^{i(n)} \left(\sum_{j=0}^{p^n-1-\mu_i} c_{i,j}T^{j+\mu_i}\right)r_{\mu_i}x^i \equiv 0 \bmod T^{p^n}
\label{eq:Tpnvanish}
\end{equation}
in $A\langle x\rangle$.  We will prove by induction on $i$ that this forces $c_{ij}=0$ for all $i,j$.
By Lemma \ref{lemdef:ra}, we may write $r_a =\sum_{i\ge 0} c_i x^i$ for $c_i\in A$ with $c_0 \in A^{\times}$.  It follows that the coefficient of $x$ in \eqref{eq:Tpnvanish} is a unit (in $A$) multiple of 
\begin{equation}
    \sum_{j=0}^{p^n-1-\mu_i} c_{i,j} T^{j+\mu_i}\label{eq:xicoeff}
\end{equation}
so this sum must be congruent to 0 modulo $T^{p^n}$, forcing $c_{ij}=0$ for $i=1$ and all $j$ as $j+\mu_i < p^n$
for all $j$ in the range of summation.  Now assuming that $c_{ij}=0$ for all $i < \ell$ and all $j$, 
the coefficient of $x^{\ell}$ in \eqref{eq:Tpnvanish} is a unit (in $A$) multiple of \eqref{eq:xicoeff}
with $i=\ell$, which must then vanish modulo $T^{p^n}$.  As before, this forces $c_{ij}=0$ for $i=\ell$.
Thus $\mathscr{B}_n$ is linearly independent over $k$. 

To show that $\mathscr{B}_n$ is a basis for $M_n$ it suffices to check that $\# \mathscr{B}_n = \dim_k M_n$.  Note there is a natural bijection between
$\mathscr{B}_n$ and $\Delta_n$
sending $T^j \overline{e}_i$ to $(i,j+\mu_i)$.  
Similarly, there is a natural bijection between the basis of $M_n$ in Proposition~\ref{cor:basis omega_n} and $\Delta_n$
sending $x^\nu w_b^{(n)} \frac{dx}{x}$ to $(\nu,b)$.
The first claim follows immediately, and the second is a straightforward consequence of this.
\end{proof}

\section{\texorpdfstring{$T$}{T}-adic Growth Properties for Towers with Minimal Break Ratios} \label{s: tadic growht for minimally ramified towers}

We continue with the notation of Section \ref{s: towers of curves and iwasawa modules of differentials},
and we fix a $\Z_p$-tower $\{X_n\}_{n\ge 0}$ over $X_0=\PP_k^1$.
By 
Proposition \ref{prop: control type results} \ref{fndiffisom} and \ref{uisom},
each $A\langle x\rangle$-module $R$ and $Z$ is free of rank 1.  Using the choice of generator specified by Remark \ref{omegaconverge}, each may be
identified with $A\langle x\rangle$.  Viewing $A\langle x\rangle= R_0\powerseries{T}$
as the ``integral'' Tate algebra of rigid analytic functions on the closed $T$-adic unit disc
with nonnegative Gauss valuation, it is natural to ask for a $T$-adic analytic 
description of the $A$-submodule $M$ of $Z$.
Via Proposition \ref{prop:omega as T module}, such a description
amounts to understanding the $T$-adic properties of the functions $r_a\in A\langle x\rangle$
of Definition \ref{def:rhoa} and Lemma \ref{lemdef:ra}.  

For general towers, this seems
difficult.  However, when the tower has minimal break ratios with invariant $d$, we prove 
in Corollary \ref{cor:functionsqb} that each $r_a$ {\em overconverges} on the $T$-adic disc
$v_T(x)\ge -p/d$.  Using this, we show (Corollary \ref{cor: Frobenius element growth}) that the element $\alpha \in A\langle x\rangle$ of Definition \ref{alphadef}, which determines the Frobenius operator on $R$
likewise overconverges on the $T$-adic disc $v_T(x)>-1/d$.  The latter is crucial for the analysis in Section~\ref{sec:lfunctions}.

\subsection{\texorpdfstring{$T$}{T}-adic properties of functions}  We begin by studying the $T$-action
on the rings of functions $R_n$. For any $f\in K_n$, it is obvious that $\ord_n(Tf)=\ord_n(\gamma f - f) \ge \ord_n(f)$, and since $T^{p^n}=0$ on $K_n$, the inequality must be strict.
This can be made precise:

\begin{lemma} \label{lem:ordTyngeneral}
   For $f \in K_n$ we have $\ordn(Tf)\geq \ordn(f)+d_1$ with equality if and only if $p\nmid \ordn(f)$. In particular, for $m \leq n$ we have $\ordn(Ty_m)=(-d_{m+1} + d_1)p^{n-m-1}$.
\end{lemma}

\begin{proof}
Let $t$ be a uniformizer at $P_n$;  it suffices to prove the result for $f=t^j$. As the smallest break in the lower ramification filtration for $X_{n}/X_0$ is $d_1$, we have
\[
T t = \gamma t -t  = c t^{d_1+1} + \textrm{O}\left(t^{d_1+2}\right)
\]
with $c\in k^{\times}$.
Hence
\begin{align*}
T t^j &=  (t + ct ^{d_1+1} + \ldots) ^{j}  - t^j
= j c t ^{j  + d_1} + \textrm{O}\left(t^{j + d_1 + 1} \right), 
\end{align*}
which immediately gives $\ordn(Tt^j) \geq j + d_1$. This will be an equality if and only if $p\nmid j$. To prove the second statement, first note that $p\nmid d_m$ and $\ord_{m+1}(y_m)=-d_m$. By the first part of the lemma, we have $\ord_{m+1}(Ty_m)=-d_m + d_1$. The result follows by noting that for $n \geq m+1$, we have $\ord_n = p^{n-m-1}\ord_{m+1}$ as $\Z$-valued functions on $K_{m+1}$.
\end{proof}

Recall that $R_n$ is generated as a $k[x]$-module by the set $\{\y^a\}_{0\le a < p^n}$,
and that $\ord_{n}(\y^a)=-p^{n}\xi_a$ by Corollary \ref{cor:valuations}.
To better understand the $T$-action on $R_n$, we must analyze the function $\xi_a$.

\begin{lemma}\label{lem:xiprops}
    Let $a ,b <p^n$ be positive integers, 
    and maintain the notation of Definition \ref{not:basep}.
    \begin{enumerate}[(i)]
        \item      
        If $a_i+b_i < p $ for all $i$, then $\xi_a+\xi_b=\xi_{a+b}$.  \label{pcarry}
        \item $\displaystyle \xi_a - \xi_{a-1} = \frac{d_{m+1}}{p^{m+1}} - \xi_{p^m-1} = \frac{s_1}{p^{m+1}}+\sum_{i=1}^{m} p^{i-m-1}(s_{i+1}-ps_i)$ where $m\colonequals v_p(a)$.\label{xidif}
    \end{enumerate}    
    In particular, $\xi_{a}$ is an increasing function of $a$.  When the tower has minimal break ratios, we have
    \begin{enumerate}[(i),resume]
        \item $\displaystyle \frac{ad}{p+1} < \xi_a \le \frac{ad}{p}$ with $d$ the ramification invariant. 
        \label{xineq}
    \end{enumerate}
    Moreover, the tower has minimal break ratios with ramification invariant $d$ if and only if, for all $a$
    \begin{enumerate}[(i),resume]
        \item $\displaystyle\xi_a = \frac{d}{p+1}\left(a + \rev(a)\right)$. \label{xifor}
        
        \item $\displaystyle\xi_a - \xi_{a-1} = dp^{-m-1}$ with $m\colonequals v_p(a)$.\label{xidifmbr}
    \end{enumerate}
\end{lemma}

\begin{proof}
    The first claim is clear from the definition of $\xi_{\star}$ (Definition \ref{not:basep}),
    and the second follows from this definition after expressing each $d_i$ in terms of the $s_j$
    via Proposition \ref{prop:genusbreaks}.  When the tower has minimal break ratios, \ref{xifor} follows
    easily from definitions and Proposition \ref{prop:genusbreaks}, and this is easily
    seen to imply \ref{xineq}.  Imposing \ref{xifor} for all $a$
    forces $d_i = d\frac{p^{2i-1}+1}{p+1}$ for all $i$, which is equivalent to 
    minimal break ratios with invariant $d$.  The final assertion follows at once
    from \ref{xidif}.
\end{proof}

\begin{remark}
    Note that, by Lemma \ref{lem:xiprops} \ref{xidif}, we may interpret the condition of minimal break ratios
    as stipulating that the increasing function $\xi_a$ grows as {\em slowly} as possible.
\end{remark}

\begin{corollary}\label{cor:ordTmbr}
    The tower has minimal break ratios if and only if $\ord_{n+1}(Ty_n^{a_n})=-p^{n+1}\xi_{a_np^n-1}$.
    for all nonnegative integers $n$ and $a_n <p$. 
\end{corollary}

\begin{proof}
    By Corollary \ref{cor:valuations} and Lemma \ref{lem:ordTyngeneral}, the left 
    side of the claimed equality is $-p^{n+1}\xi_{a_np^n}+d_1$, which equals the right
    side if and only if the tower has minimal break ratios by Lemma \ref{lem:xiprops} \ref{xidifmbr}.    
\end{proof}

\begin{definition} \label{defn:wa}
For each nonnegative integer $a$ and real number $m$ 
we define
 $$W_a^m \colonequals \left \{ \sum_{b=0}^{a} g_b \y^b : g_b \in k[x], \deg(g_b) \leq (a-b)m \right\},$$ 
viewed as a $k$-subspace of $R_n$ for any $n$ with $p^n>a$. 
\end{definition}

\begin{proposition} \label{prop:T is triangular}
If $\{X_n\}$ has minimal break ratios with invariant $d$ and $m\colonequals v_p(a)$, then
\[
T\y^a = (-1)^m a_m \y^{a-1} + \sum_{b<a-1} g_b \y^b %
\]
for unique $g_b\in k[x]$ with $\deg(g_b) \leq (a-1-b)d/p$;
in particular, $T\y^a \in W_{a-1}^{d/p}$.  %
\end{proposition}

To prove Proposition \ref{prop:T is triangular}, we require:

\begin{lemma}\label{lem:key}
    If the tower has minimal break ratios with invariant $d$, then for $a' <p^n$ and $a_n<p$ 
    \begin{equation}
        \y^{a'} T(y_n^{a_n}) = \sum_{0\le b < a_np^n} g_b \y^b\label{gsum}
    \end{equation}
    for unique $g_b\in k[x]$ with $\deg g_b \le \xi_{a'}+\xi_{a_np^n-1-b}\le \frac{d}{p}(a'+a_np^n-1-b)$.
\end{lemma}

\begin{proof}
    As $R_{n+1}$ is $T$-stable, we have
    $
        T(y_n^{a_n}) = \sum_{0\le c < p^{n+1}} g_c \y^c
    $
    for unique $g_c\in k[x]$ by Corollary \ref{cor:valuationsdistinct}.
    Applying $\ord_{n+1}$, it follows 
    from Corollaries \ref{cor:valuationsdistinct} and \ref{cor:ordTmbr} that
    $
         \deg(g_c) + \xi_c \le \xi_{a_np^n-1}.
    $
    As $\xi_{\star}$ is increasing by Lemma \ref{lem:xiprops}, we conclude that $g_c=0$ when $c \ge a_np^n$.
    For $c < a_np^n$, we have $\y^c = \y^{c'} y_n^{b_n}$ with $b_n < a_n$ and $c' < p^n$.
    Since $a'< p^n$ and $R_n$ is a ring, each product $\y^{a'} \y^{c'}$ may be written as
    a $k[x]$-linear combination of $\y^{b'}$ with $b'<p^n$ by Corollary \ref{cor:valuationsdistinct}, 
    and it follows that $\y^{a'} T(y_n^{a_n})$ is a $k[x]$-linear combination of $\y^{b}$ with 
    $b =b'+b_np^n < a_np^n$ as in \eqref{gsum}.  
    Taking $\ord_{n+1}$ of both sides of \eqref{gsum} and appealing again
    to Corollaries \ref{cor:valuationsdistinct} and \ref{cor:ordTmbr}, we find
    \begin{equation}
        \deg(g_b) + \xi_b \le \xi_{a'} + \xi_{a_np^n-1}.\label{gbbound}
    \end{equation}
    Now for any $b < a_np^n$, the sum of $b$ and $a_np^n-b-1$ involves no $p$-adic carries,
    so rearranging \eqref{gbbound} and applying Lemma \ref{cor:ordTmbr} \ref{pcarry} and \ref{xineq}
    completes the proof.
\end{proof}

\begin{proof}[Proof of Proposition \ref{prop:T is triangular}]
    Taking $a'=0$ in Lemma \ref{lem:key} gives the result when $a=a_np^n$ and $a_n<p$.  
    Inductively assuming the result for all $b<p^n$, and letting $a=a'+a_np^n$ with $a'<p^n$,  
    we compute
    \begin{equation}
        T(\y^a) = T(\y^{a'} y_n^{a_n}) = T(\y^{a'})T(y_n^{a_n})+\y^{a'}T(y_n^{a_n})+y_n^{a_n}T(\y^{a'}).
        \label{Tdiff}
    \end{equation}
    By our inductive hypothesis and Lemma \ref{lem:key}, we have 
    \begin{equation*}
        T(\y^{a'})T(y_n^{a_n})=\sum_{b< a'} g_b \y^b T(y_n^{a_n}) 
        =\sum_{b< a'} g_b \sum_{c < a_np^n} h_c(b) \y^c
    \end{equation*}
    for unique $g_b,h_c(b) \in k[x]$ with $\deg g_b\le \frac{d}{p}(a'-b-1)$
    and $\deg h_c(b) \le \frac{d}{p}(b+a_np^n-1-c)$.  Rearranging the 
    double sum and collecting terms shows that $T(\y^{a'})T(y_n^{a_n}) \in W^{d/p}_{a-1}$.
    The term $\y^{a'}T(y_n^{a_n})$ lies in $W_{a-1}^{d/p}$ by Lemma \ref{lem:key}.
    Again by the inductive hypothesis, we have 
    $$
        y_n^{a_n}T(\y^{a'}) =\sum_{b'<a'} g_{b'} \y^{b'+a_np^n}
    $$
    with $\deg(g_{b'}) \le \frac{d}{p}(a'-1-b')$.
    Writing $b=b'+a_np^n$, so $a-b=a'-b'$, and setting $g_b\colonequals g_{b'}$ then shows
    that the third and final term lies in $W_{a-1}^{d/p}$ as well.
    Writing $c_b \in k[x]$ for the coefficient of $\y^{b-1}$ in $T\y^b$,
    a straightforward induction using Lemma \ref{lem:Taunit} shows 
    that $\prod_{b=1}^a c_b = \prod_{i\ge 0} a_i!$, which 
    yields the explicit formula $c_a = (-1)^m a_m$ for $m=v_p(a)$.
\end{proof}

    \begin{remark}
        The natural analog of Proposition~\ref{prop:T is triangular} is false if the tower does not have minimal break ratios. Consider a tower where $s_2>ps_1$ and $s_{2+k}=s_2p^k$ for $k\geq 0$. Further suppose that $s_1 \equiv -s_2 \mod p$; lots of such towers exist, with $s_2$ arbitrarily large compared to $s_1$. Then the analog of ``$d$'' is $\frac{s_2}{p}$ and by Lemma~\ref{lem:ordTyngeneral}
        \[\ord_2(Ty_1)=-d_2+d_1 = -p(s_2-s_1).\]
        In particular, $\ord_2(Ty_1) \equiv 2s_1p \mod p^2$.  But as $\ord_2(\y^{p-2})=-(p-2)ps_1$,  we have $\ord_2(\y^{p-2})\equiv \ord_2(Ty_1) \mod p^2$. In particular,
        if we write \[Ty_1=g_0 + g_1\y^{1} + \dots g_{p-1}\y^{p-1},\]
        we must have $\ord_2(Ty_1)= \ord_2(g_{p-2}) + \ord_2(\y^{p-2})$ by Corollary~\ref{cor:valuationsdistinct}.
        We compute 
        \begin{align*}
            -\ord_2(g_{p-2}) &= \ord_2(\y^{p-2}) - \ord_2(Ty_1) \\
            &=p(s_2-ps_1) + ps_1 \\
            &= s_2 + (p-1)(s_2 - ps_1).
        \end{align*}
Hence $\displaystyle \deg(g_{p-2}) = \frac{s_2 + (p-1)(s_2-ps_1)}{p^2} =   \frac{d}{p} + \frac{(p-1)(s_2-ps_1)}{p^2}$.
        Since $s_2-ps_1>0$, we see that the degree of $g_{p-2}$ is larger than $\frac{d}{p}$. In particular, $g_{p-2} \y^{p-2}$ is not in $W_{p-1}^{d/p}$.
    \end{remark}

\subsection{Functions and Differentials as \texorpdfstring{$A$}{A}-modules}\label{ss: functions and differentials as T modules}

In order to understand the $T$-adic convergence properties of the functions $r_a\in A\langle x\rangle$
specified by Definition \ref{def:rhoa} and Lemma \ref{lemdef:ra}, in view of Proposition \ref{prop:T is triangular},
we must understand the images of the subspaces $W_a^{d/p}$ of $R_n$ under the $A$-module
identification $R_n\simeq A_n[x]\cdot w_0^{(n)}$ induced by Proposition \ref{prop: control type results} \ref{uisom}.  To do so, we introduce the following $A$-submodules of $A_n[x]$.

\begin{definition}\label{d:Lm}
    For a rational number $m$, define $ A_n[x]^m  \subset A_n[x]$ by
    \begin{align*}
    A_n[x]^m &\colonequals \left \{\sum_{0 \leq i < p^n m} c_i x^i : c_i \in A_n, \, v_T(c_i) \geq \frac{i}{m}  \right \}
    = \left \{ \sum_{0 \leq i < p^n} b_i T^i : b_i \in k[x], \, \deg(b_i) \leq im \right \}.
    \end{align*}
    Reduction modulo $T^{p^n}$ induces a surjection $A_{n+1}[x]^m\rightarrow A_n[x]^m$. We define
    $A\langle x\rangle ^m\colonequals \varprojlim_n A_n[x]^m$.
\end{definition}

It is straightforward to check the two sets in the definition of $A_n[x]^m$ are equal,
and that there is a natural action of $A_n[x]^m$ on $K_n$ as $T^{p^n}$ acts trivially.
Note that $A\langle x\rangle^m$ is precisely the $A$-submodule of $A\langle x \rangle$
consisting of those power series in $x$ that are convergent with values bounded by $1$ on the open $T$-adic disc $v_T(x)> -1/m$, i.e.
\[
A\langle x\rangle^m =  \left \{ \sum_{i\ge 0} c_i x^i \in A \langle x \rangle \ :\ v_T(c_i) \ge i/m \right\}.
\]

\begin{definition} \label{definition: Lm}
For a rational number $m$, let $L^m$ be the ideal of $A\langle x \rangle^m$ consisting of functions vanishing at $x=0$, i.e.
\[
L^m = \left \{ \sum_{i\ge 1} c_i x^i \in L\ :\ v_T(c_i) \ge i/m \right \}.
\]
\end{definition}

\begin{corollary} \label{cor:functions fa}
For any $0 \leq a < p^n$ and $m\geq d/p$ 
we have $T^{p^n-1-a} A_{n}[x]^{m} w_0^{(n)} = W_a^{m}$ as $k$-subspaces of $R_n$. In particular, $A_n[x]^m w_0^{(n)} = W_{p^n-1}^m$.  Furthermore, $T W_a^m = W_{a-1}^m$.
\end{corollary}

\begin{proof}
We first claim that $TW_a^m \subseteq W_{a-1}^m$ as $k$-subspaces of $R_n$.  
To see this, it suffices to prove that $h T\y^b$ lies in $W_{a-1}^m$
for any $h\in k[x]$ with $\deg(h)\le (a-b)m$.  By Proposition \ref{prop:T is triangular},
we know that $T\y^b = \sum_{c<b} g_c \y^c$ with $\deg(g_c)\le (b-1-c)\frac{d}{p}$.
Since $m\ge d/p$ by hypothesis, we have 
$$\deg(hg_c)\le (a-b)m+(b-1-c)\frac{d}{p} \le (a-b)m+(b-1-c)m=(a-1-c)m$$
and it follows that $hT\y^b \in W_{a-1}^m$, as desired.  To see that the inclusion
$TW_a^m\subseteq W_{a-1}^m$ of $k$-subspaces of $R_n$ is an equality, we check that
they have the same dimension.  As the kernel of multiplication by $T$ on $R_n$
is exactly $k[x]$, the kernel of multiplication by $T$ on $W_a^m$ consists of
polynomials $h\in k[x]$ with $\deg(h)\le am$, which has dimension $1+\lfloor am\rfloor$.
Since 
$$\dim_k W_{a}^m=\sum_{0\le b\le a} (\lfloor (a-b)m\rfloor +1),$$
it follows that $\dim_kW_a^m - \dim_k W_{a-1}^m = 1+\lfloor am\rfloor$. By comparing dimensions we have $TW_a^m=W_{a-1}^m$.

For ease of notation, let us put $U_a^m\colonequals T^{p^n-1-a}A_n[x]^mw_0^{(n)}$.
Since $T^{p^n-1}\y^{p^n-1}\in k^{\times}$ by Lemma \ref{lem:Taunit},
we have $U_0^m=k=W_0^m$ inside $R_n$.
Inductively assuming that $U_{a-1}^m=W_{a-1}^m$,
we have a commutative diagram of exact sequences of $k$-subspaces of $R_n$
\begin{equation*}
        \xymatrix{
            0 \ar[r] & k[x] \cap U_a^m \ar[r] & U_a^m \ar[r]^-{\cdot T} & U_{a-1}^m \ar[r]\ar@{=}[d] & 0 \\
            0 \ar[r] & k[x] \cap W_a^m \ar[r] & W_{a}^m \ar[r]_-{\cdot T} & W_{a-1}^m \ar[r] & 0 \\
        }
\end{equation*}
so to prove that $U_a^m=W_a^m$ it is enough to prove that $k[x]\cap W_a^m=k[x]\cap U_a^m$,
which is clear from the definition of $W_a^{m}$ and the second description of $A_n[x]^m$.
\end{proof}

\begin{corollary} \label{cor:functionsqb}
For all $a$, the element $r_a \in A\langle x\rangle$ of Lemma \ref{lemdef:ra} lies in $A\langle x\rangle^{d/p}$.
\end{corollary}
\begin{proof}
    Fixing $a$, we have %
    $T^a r_a w_0=w_a$ using Definition \ref{def:rhoa} and Lemma \ref{lemdef:ra}. 
    Since $w_a\bmod T^{p^n}=w_a^{(n)}=(-1)^n \y^{p^n-1-a}$ when $p^n>a$ 
    (Definitions \ref{defn:omegab} and \ref{def:wa}), 
    we conclude that $T^a r_a w_0 \bmod T^{p^n} \in W_{p^n-1-a}^{d/p}=T^aA_n[x]^{d/p} w_0^{(n)}$
    thanks to Corollary \ref{cor:functions fa}, so that
    $r_a$ is congruent to an element of $A_n[x]^{d/p}$ modulo $T^{p^n-a}$.
    For varying $n$, these elements are compatible and uniquely determined modulo $T^{p^n-a}$,
    so yield a unique element of $A\langle x\rangle^{d/p}$ congruent to $r_a$ modulo every power of $T$.
\end{proof}

With Corollary \ref{cor:functionsqb} in hand, we now use the 
description of $M$ as a $A$-module provided by 
Proposition \ref{prop:omega as T module} to give strong $T$-adic estimates 
for $M$ when the tower has minimal break ratios.  To do so, it remains only to provide an
explicit description of the function 
$\mu_j$ of Definition \ref{defn:muj} under this hypothesis.  Our description uses
the lexicographic ordering on integer sequences: $(a_n)_{n\ge 0} \succ (b_n)_{n\ge 0}$
    if and only if there exists $m\ge 0$ with $a_i=b_i$ for $i<m$ and $a_i>b_i$.

\begin{proposition} \label{prop:muestimates}
    Assume that $\{X_n\}$ has minimal break ratios with ramification invariant $d$, and for $j\in \Z_{> 0}$ let 
    $\displaystyle \frac{p+1}{d} j=\sum_{n\in \Z} j_n p^{n}$ be the base-$p$ expansion of  $\frac{p+1}{d}j$. Then $\mu_0=1$ and for $j>0$
    \begin{equation*}
            \mu_j = \begin{cases}
                        \left\lfloor \frac{p+1}{d}j\right\rfloor & \text{if}\ (j_n)_{n\ge 0} \succ (j_{-1-n})_{n\ge 0} \\ \\
                        \left\lfloor \frac{p+1}{d}j\right\rfloor + 1 =\left\lceil \frac{p+1}{d}j\right\rceil  & \text{otherwise}
            \end{cases}
    \end{equation*}
\end{proposition}

\begin{proof}
        Set $b=\mu_j$, %
        so $b\ge 1$ and
        $\xi_b > j \ge \xi_{b-1}$.  By Lemma \ref{lem:xiprops} \ref{xifor}, this is equivalent to 
        \begin{equation}
            b + \rev(b) > \frac{p+1}{d}j \ge b-1 + \rev(b-1).\label{bjineq}
        \end{equation}
        Taking the floor of \eqref{bjineq} and remembering that 
        $0 < \rev(b)<1$ when $b>0$ gives 
        $$
            1+\left\lfloor \frac{p+1}{d}{j}\right\rfloor \ge b \ge \left\lfloor \frac{p+1}{d}{j}\right\rfloor.
        $$
        Let $c=\lfloor \frac{p+1}{d}{j}\rfloor$, so 
        $\frac{p+1}{d}j = c+\varepsilon$ with $0\le \varepsilon < 1$ and $c\le b$.
        Then
        $\rev(c) > \varepsilon$ if and only if $c+\rev(c) > \frac{p+1}{d}j$
        if and only if $c\ge b$ if and only if $c=b$.  The result follows
        upon expressing $\rev(c)$ and $\varepsilon$ in terms of the $p$-adic digits
        of $\frac{p+1}{d}{j}$.
\end{proof}

\begin{corollary} \label{cor:muestimate}
For $j>0$ we have $| \mu_j - \frac{p+1}{d} j | <  1$.  When $d | (p+1)$, $\mu_j = \frac{p+1}{d} j$.
\end{corollary}

The following results are not used in the later sections of the article, but we believe they are of intrinsic interest. Indeed, %
Corollary \ref{cor:global pro-differential growth result} %
helped us to conceive of the strategy of this article.
\begin{lemma} \label{l: lower bound of mu_j}
    For all $j\geq 0$ we have $\mu_j > \frac{p}{d}j$.
\end{lemma}

\begin{proof}
     Setting $b = \mu_j$, it follows from the definition of $\mu_j$ and Lemma \ref{lem:xiprops} \ref{xifor} that
\[b - \frac{pj}{d} > \frac{j}{d} - \rev(b).\]
If $j \geq d$, the right hand side is greater than zero, which gives the result. Assume $j < d$.  If $b\geq p$ there is nothing to show, so we may assume $b < p$.  
In that case, $\rev(b) = b/p$, and we have

\[b - \frac{pj}{d} > \frac{j}{d} - \frac{b}{p} = \frac{1}{p}\Big(\frac{pj}{d} - b\Big),\]
which implies  $b - \frac{pj}{d}$ is positive.
\end{proof}

\begin{corollary} \label{cor:global pro-differential growth result}
    We have $M \subset L^{d/p} w_0 \frac{dx}{x}$.
\end{corollary}

\begin{proof}
    By Proposition \ref{prop:omega as T module}, it is enough to show that $T^{\mu_j}r_{\mu_j}x^j \in A\langle x\rangle^{d/p}$ for $j\geq 1$.
    As $\mu_j > \frac{pj}{d}$, we have $T^{\mu_j}x^j \in A\langle x\rangle^{d/p}$, so the result follows
    from Corollary \ref{cor:functionsqb} and the fact that $A\langle x\rangle^{d/p}$ is a ring.
\end{proof}

\subsection{\texorpdfstring{$T$}{T}-adic properties of Frobenius}
We conclude this section by analyzing the Frobenius operator on $W_a^m$ 
and the $T$-adic growth of the Frobenius element $\alpha$ defined in Definition \ref{alphadef}. 

\begin{proposition}\label{prop:Frobconv}
For any $a$, $F(W_a^{d/p}) \subset W_a^d$.  
\end{proposition}

\begin{proof}
    We will proceed by induction on $a$, with base case $a=0$ immediate.
    Assume the asserted containment holds with $a-1$ in place of $a$, 
    and let $z\in W_a^{d/p}$.  Choosing $n$ with $p^{n+1}>a$, since 
    $R_n$ is clearly $F$-stable, we have $F(z)\in R_n$.
     From  
Corollary \ref{cor:functions fa} we know $TW_a^{m} =W_{a-1}^{m}$ for $m\ge d/p$, so  
    $Tz\in W_{a-1}^{d/p}$, and the inductive hypothesis gives
    $T(Fz)=F(Tz) \in W_{a-1}^d$ as $T$ commutes with $F$.  Using again
    the fact that $TW_a^m=W_{a-1}^m$ for $m\ge d/p$, there exists $f\in W_a^d$
    with $F(z)-f \in R_n$ killed by $T$.  As the kernel of $T$ on $R_n$ is $k[x]$,
    we conclude that
    $$
        F(z) = \sum_{b=0}^a g_b \y^b \quad \text{with}\quad \deg(g_b) \le (a-b)d\quad \text{for}\ b>0.
    $$
    and it remains to check that $g_0 \in W_a^d$, i.e.~$\deg(g_0) \le ad$.
    As $z\in W_a^{d/p}$ by hypothesis, we have 
    $$\ord_{n+1}(z) \ge -p^{n+1}\left((a-b)\frac{d}{p}+\xi_b\right) \ge -adp^n$$
    by Corollary \ref{cor:valuationsdistinct} and Lemma \ref{lem:xiprops} \ref{xineq},
    and it follows that $\ord_{n+1}(Fz) \ge -adp^{n+1}$, which forces
    $\deg(g_0) \leq ad$ by Corollary~\ref{cor:valuationsdistinct}.  
\end{proof}

\begin{corollary} \label{cor: Frobenius element growth}
Let $\alpha$ be as in Definition \ref{alphadef}. Then both $\alpha$
and $\alpha^{-1}$ lie in $A\langle x\rangle^{d}$.
\end{corollary}

\begin{proof}
    Combining Proposition \ref{prop:Frobconv} and Corollary \ref{cor:functions fa} gives 
    $F(A_n[x]^{d/p} w_0^{(n)}) \subset A_n[x]^d w_0^{(n)}$.  Passing to the limit
    yields $F(A\langle x\rangle^{d/p} w_0) \subset A\langle x\rangle^d w_0$, which suffices to give $\alpha \in \calA\langle x\rangle^d$.
    By Remark \ref{rem:alphaunit}, we know that $\alpha \in 1+TA\langle x\rangle$, so that $\alpha$ is a unit in $A\langle x\rangle$ \cite[\S 5.1.3 Proposition 1]{BGR}.
    Writing $\alpha^{-1}=\sum_{i\ge 0} c_i x^i$ with $c_i\in A$
    and $\alpha = \sum_{i\ge 0} b_i x^i$ with $v_T(b_i) \ge i/d$,
    we prove by induction that $v_T(c_i)\ge i/d$.  Since $\alpha \in 1 + TA\langle x \rangle$ we deduce $b_0\equiv 1 \mod T$. In particular, we know $b_0 \in A^\times$, and thus $c_0=b_0^{-1} \in A^\times$ which gives the base case.
    For $i > 0$,  comparing  coefficients of $x^i$ in $1=\alpha\cdot \alpha^{-1}$ gives
    \[
        c_i =  - \sum_{j<i} b_{i-j} c_j.
    \]
    Assuming that $v_T(c_j)\ge j/d$ for $j < i$, we have 
    $v_T(b_{i-j}c_j) \ge (i-j)/d + j/d = i/d$ for $j < i$,
    so every term in the above sum has valuation at least $i/d$,
    which completes the inductive step.
\end{proof}

\section{\texorpdfstring{$L$}{L}-functions and spectral properties of Cartier} \label{sec:lfunctions}

\renewcommand*{\b}{\beta}
\newcommand*{\m}{\mathfrak{m}}

We continue with the notation from Section \ref{s: tadic growht for minimally ramified towers}. In particular, the tower $X_\infty/X$ has minimal break ratios with invariant $d$. Recall from equation \eqref{eq:galoischar} that $\chi$ is a Galois character
$\chi:\Gal(L/K) \to A^\times$ that sends $a \in \Z_p \simeq \Gal(L/K)$ to $(1+T)^a$. The $L$-function associated to $\chi$ is defined by the following equivalent Euler products
\begin{align}
\begin{split}\label{eq: L-function definition}
    L(\chi,s) &= \prod_{\mathrm{places }~v}\left( 1 - \chi(Frob_v) s^{\deg(v)}\right) \\
    &= \prod_{\mathrm{places }~v}\left( 1 - \alpha(v) \dots \sigma^{d\fielddegree - 1}(\alpha)(v) s^{\deg(v)}\right),
\end{split}
\end{align}
where the product is over all places of $K$ away from $\infty$, and the equivalence of the two Euler products follows from Proposition \ref{prop: Frobenius }. 
In this section we use properties of $L(\chi,s)$, essentially proven in \cite{KostersZhu}  and refined in \cite{kramer-milleruptonII},  to study the Cartier operator on $Z$.

\subsection{Recollections on Banach spaces and Fredholm series} \label{ss:recollectionsbanach}
For clarity of exposition, in this section only we fix an arbitrary complete, valued field 
$(\calE, v:\calE^{\times} \to \R)$ with
valuation ring $\calA$ having maximal ideal $\m$ and residue field $k$.
 A {\em Banach space} over $\calE$ is an $\calE$-vector space  
 $\calV$ equipped with a valuation $\upnu$ that is compatible with the
 valuation on $\calE$ and with respect to which $\calV$
 is complete.
Given $\calV$, we write $\calD$ for the $\calA$-submodule of $\calV$ consisting of all elements
with nonnegative  valuation, 
and set $\overline{\calD}:=\calD/\m \calD$, which is naturally a $k$-vector space.

\begin{definition} \label{defn:schauderbasis}
    A {\em Schauder basis} of an $\calE$-Banach space $\calV$
    is a sequence $\calB=\{\b_1,\b_2,\ldots\}$ in $\calV$ with the property that,
    for every element $m \in \calV$, there exists a unique sequence $\{s_1,s_2,\ldots\}$
    in $\calE$ such that $\lim_{N\rightarrow \infty} \sum^N_{i=1} s_i \b_i$
    exists and equals $m$.  We will henceforth refer to $\calB$ simply as a {\em basis}
    of $\calV$.  If in addition we have $s_i\rightarrow 0$ as $i\rightarrow \infty$ and $\upnu(m)=\inf_i v(s_i)$, we say $\calB$ is an \emph{orthonormal basis}. 
\end{definition}

\begin{definition}\label{def: fildefgeneral}
Let $\calV$ be an $\calE$-Banach space with orthonormal basis $\calB=\{\b_i\}$.
For $t\in \Z_{\ge 0}$, define
		\begin{align*}
			\Fil{> \filindex}_{\calB} \calV	 \colonequals	\left\{	\sum_{i = 1}^\infty a_i \b_i \in \calV: a_i = 0\text{ for }i \leq \filindex	\right\}
   \quad\text{and}\quad
            \Fil{\leq \filindex}_{\calB} \calV	 \colonequals	\left\{	\sum_{i = 1}^\infty a_i \b_i \in \calV: a_i = 0\text{ for }i > \filindex	\right\}.
	\end{align*}
In other words, $\Fil{\le \filindex}_{\calB} \calV$ is the $\calE$-subspace of $\calV$ spanned by $\b_1,\ldots,\b_t$, and $\Fil{>\filindex}_{\calB} \calV$ is the closure of the span of $\{\b_i\}_{i>t}$ in $\calV$.  Note that $\Fil{\le \filindex}_{\calB} \calV$
(respectively $\Fil{>\filindex}_{\calB} \calV$)
is an increasing (respectively decreasing) filtration of $\calV$ by $\calE$-subspaces.
If $\calW$ is any subquotient of $\calV$, we will write $\Fil{\star}_\calB{}\calW$ for the induced filtrations on $\calW$
\cite[\href{https://stacks.math.columbia.edu/tag/0120}{Tag 0120}]{stacks-project}.
\end{definition}

For our applications, we will need to analyze {\em semilinear} maps of Banach spaces.  Rather than develop
the theory of such maps from scratch, we will allow ourselves sufficient flexibility in the usual theory of
linear maps as follows: let $k'\subseteq k$ be any subfield with $k/k'$ of finite degree $\nu$
and let $\calE'$ be the unique corresponding subfield of $\calE$ with $\calE/\calE'$ unramified.
If $\calV$ is any Banach space over $\calE$, we may consider it as a Banach space over $\calE'$,
and we will study maps of Banach spaces over $\calE$ that are $\calE'$-linear.  Fixing
a $k'$-basis $\varsigma_1,\ldots,\varsigma_{\nu}$ of $k$, each $\varsigma_j$
is a root of unity of order prime to $p$, so by Hensel's lemma 
uniquely lifts to a root of unity in $\calE$, which we again denote by $\varsigma_j$.
Note that if $\calB=\{\b_i\}$ is any basis of a Banach space $\calV$ over $\calE$,\
then $\calB_{k'}:=\{\varsigma_j \b_i\}$ is a basis of $\calV$ over $\calE'$, and is orthonormal
if and only if $\calB$ is.

    \begin{definition}
        Let $\calB=\{\b_i\}_{i\ge 1}$ be a basis of a Banach space $\calV$ over $\calE$ and $\Flin: \calV \to \calV$ an $\calE'$-linear map.
        Let $A=(a_{ij})$ be the matrix of $\Flin$ with respect to $\calB_{k'}$.
        For each $n$, consider
        \begin{equation}\label{eq: leibniz type formula}
				c_n	\colonequals 	(-1)^n \sum_{\substack{S\subseteq \calB_{k'} \\ |S| = n}} \sum_{\sigma \in \mathrm{Sym}(S)} \mathrm{sgn}(\sigma) \prod_{\b \in S}    a_{\sigma(\b),\b}
        \end{equation}
        in which the first (infinite) sum is over all subsets $S$ of $\calB_{k'}$ of cardinality $n$, and the second sum
        is over all permutations of $S$.  Provided the infinite sum in equation \eqref{eq: leibniz type formula} converges for all $n$, 
        we say that the {\em Fredholm determinant of $\Flin$ relative to $\calB_{k'}$} exists, and define it to be
        \begin{equation*}
            \Det_{\calE}(1 - s \Flin|\calB_{k'}) \colonequals  1 + c_1 s + c_2 s^2 + \cdots \in \calE\llbracket s \rrbracket.
        \end{equation*}
    \end{definition}
    In general, the existence and value of the Fredholm determinant depends both on $\calB_{k'}$ and $\Flin$. 
    Under suitable ``overconvergence'' hypotheses on $\Flin$ relative to $\calB_{k'}$, the Fredholm determinant of $\Flin$ exists, and is independent of
    $\calB_{k'}$ up to ``overconvergent'' change of basis, as the following
    Definitions and Lemmas make precise:

    \begin{definition}\label{def:overconvergent balls}
        Let $\calV$ be an $\calE$-Banach space with orthonormal basis $\calB=\{\b_i\}$.  For $r> 0$, define
        \[
            \calD^r_{\calB}:=\left\{ \sum_{i\ge 1} a_i \b_i \in \calD\ :\ rv(a_i) \ge i\right\}
            \quad\text{and}\quad
             \calD^{[r]}_{\calB}:=\left\{ \sum_{i\ge 1} a_i \b_i\in \calD^r_{\calB}\ :\ rv(a_i)-i\rightarrow \infty\right\}.
        \]
        We equip $\calD_{\calB}^{[r]}$ with the valuation $\upnu_{r}\left(\sum a_i\b_i\right):=\inf_i (rv(a_i)-i)$,
        which makes $\calV^{[r]}_{\calB}:=\calD_{\calB}^{[r]}\otimes_{\calA}\calE$ into a Banach space over $\calE$.
        If $r^{-1}=v(\delta)$ for some $\delta \in \calA$, we set $\calB^r:=\{ \delta^i \b_i\}$; it is an orthonormal basis of $\calV^{[r]}_{\calB}$.

    \end{definition}

    \begin{definition}
        Let $\calV$ and $\calB$ be as in Definition 
        \ref{def:overconvergent balls}.
        A sequence $B=\{b_i\}_{i > 0}$ in $\calV$
        is {\em $r$-overconvergent with respect to $\calB$} if, for all $i$,
        the $\calB$-expansion of $b_i = \sum_{j\ge 0} a_{ij} \beta_j$ has $rv(a_{ij}) \ge j-i$
        for all $j$, with equality when $j=i$. An $\calE'$-linear map $\Flin:\calV\rightarrow \calV$
        is {\em $r$-overconvergent with respect to $\calB$} if $\Flin(\calD^r_{\calB}) \subseteq \calD^{\rho}_{\calB}$
        for some $\rho < r$.
    \end{definition}

    \begin{remark}\label{rem:B resp B}
        Tautologically, $\calB$ is $r$-overconvergent with respect to $\calB$ for any $r>0$.
    \end{remark}

    \begin{lemma}
        \label{l: open discs with changing }
        Let $\calV$ be an $\calE$-Banach space with orthonormal basis $\calB=\{\beta_i\}$.
        If $B=\{b_i\}\subseteq \calD$ is $\rho$-overconvergent, then for any $r> \rho$ (respectively $r \ge \rho$), 
        we have $\calD^{[r]}_B=\calD^{[r]}_{\calB}$ (respectively $\calD^{r}_B=\calD^{r}_{\calB}$).
        If $r^{-1}=v(\delta)$ with $\delta\in \calA$, then $B^{r}$
        is an orthonormal basis of $\calV^{[r]}_{\calB}$.
    \end{lemma}

    \begin{proof}
        Let $\calE^{''}$ be an extension of $\calE$ that is complete with respect to a valuation $v:(\calE^{''})^\times \to \R$ extending that on $\calE$ and let $\calA^{''}$ be the valuation ring of $\calE^{''}$. 
        Then the completed tensor product $\calV^{''}:=\calE^{''} \widehat{\otimes}_{\calE} \calV$ is a Banach space over $\calE^{''}$
        and $\calD^{''}:=\calA^{''} \widehat{\otimes}_{\calA} \calD$ is the $\calA^{''}$-submodule of $\calV^{''}$ consisting of all elements with nonnegative valuation. Viewing $\calB$ and $B$ as orthonormal bases of $\calV^{''}$, it is immediate that $\calD_\star^{[r]}=(\calD^{''}_\star)^{[r]} \cap \calV$ where $\star$ is $\calB$ or $B$.
        Thus, to prove $\calD_B^{[r]}=\calD_{\calB}^{[r]}$, we are reduced to the case where $r^{-1}=v(\delta)$ for some $\delta \in \calA$. 
        As $\calB$ is an orthonormal basis and $B\subseteq \calD$, for each $i$ we may write $b_i=\sum_{j > 0} a_{ij} \beta_j$
        with $a_{ij}\in \calA$.  The hypothesis that $B$ is $\rho$-overconvergent gives $v(a_{ij}) \ge \max\{0,(j-i)/\rho \}$
        for all $j$, so 
        \begin{equation}
            r v(\delta^i a_{ij})-j \ge \max\left\{i-j, \left(\frac{r}{\rho}-1\right)(j-i)\right\},\label{eq:valcalc}
        \end{equation}
        which is nonnegative and tends to $\infty$ with $j$ when $r>\rho$.  Thus, $B^{r}\subseteq \calD^{[r]}_{\calB}$.
        By hypothesis, we have $a_{ii}\in \calA^{\times}$, so
        \[
            \upnu_{r}(\delta^i\beta_i - \delta^ia_{ii}^{-1} b_i) =  \inf_{j\neq i} (r v(\delta^i a_{ii}^{-1}a_{ij})-j))
        \]
        is {\em positive} thanks to \eqref{eq:valcalc}; we conclude that $\delta^i \beta_i$ and $\delta^i a_{ii}^{-1} b_i$
        have the same image in $\overline{\calD}_{\calB}^{[r]}$.
        Since $a_{ii}\in \calA^{\times}$ and $\calB^{r}$ is an orthonormal basis of $\calV^{[r]}_{\calB}$, so also is $B^{r}$ by \cite[Lemme 1]{SerreOperators}, and it follows that $\calD^{[r]}_{B}=\calD^{[r]}_{\calB}$.
        As $\calD^{r}_{\calB}$ is the intersection over all $r' > r$ of $\calD^{[r']}_{\calB}$,
        we deduce from this the equality $\calD^{r}_{\calB}=\calD^{r}_B$ for all $r\ge \rho$.
    \end{proof}

    A linear map of Banach spaces is {\em completely continuous}
    if it is a limit of maps of {\em finite} rank.  By \cite[Proposition 7]{SerreOperators},
    the Fredholm determinant of a completely continuous endomorphism exists relative to {\em any} orthonormal basis, is independent
    of the choice of basis, and is entire as a function of $s$ (i.e.~converges for all $s$). 
    The following result builds upon these key facts to establish the existence and independence of the Fredholm determinant
    in the context of overconvergence:

    \begin{lemma}\label{lem: technical}
        Let $\calV$ be a Banach space over $\calE$ with orthonormal basis 
        $\calB=\{\beta_i\}$, and $\Flin:\calV\rightarrow \calV$ an $\calE'$-linear map.
        Assume that $\Flin$ is $r$-overconvergent with respect to $\calB$ for some
        $r>0$.   Then $\Flin$ restricts to a completely continuous
        $\calE'$-linear endomorphism of $\calV^{[r]}_{\calB}$. %
        In particular, if $r^{-1}=v(\delta)$ and $B=\{b_i\}\subseteq \calD$ is $\rho$-overconvergent with respect to $\calB$ for some $\rho < r$,
        then $\Det_{\calE'}(1-s\Flin|B_{k'})$ and $\Det_{\calE'}(1-s\Flin|\calB^r_{k'})$
        exist and are equal.
    \end{lemma}

    \begin{proof} 
            For $t>0$, consider the composite mapping
            \begin{equation*}
                    \xymatrix{
                        {\Flin_t: \calV} \ar[r]^-{\Flin} & \calV \ar@{->>}[r] &{\Fil{\le t}_{\calB} \calV} \ar@{^{(}->}[r] & {\calV}.
                        }
            \end{equation*}
            As $\calD^{[r]}_{\calB} \subseteq \calD^r_{\calB}$, 
            a straightforward computation gives
            \[
                (\Flin-\Flin_t)(\calD^{[r]}_{\calB}) \subseteq (\Flin-\Flin_t)(\calD^{r}_{\calB}) \subseteq \Fil{>t}_{\calB} \calD^{\rho}_{\calB} \subseteq 
                 \left\{m \in \calD^{[r]}_{\calB}\ :\ \upnu_r(m) \ge (t+1)\left(\frac{r}{\rho}-1\right)\right\}
            \]
            for some $\rho< r$, so that $\Flin = \lim_{t\rightarrow\infty} \Flin_t$.
            As each $\Flin_t$ is visibly of finite rank, it follows that $\Flin$
            is a completely continuous $\calE'$-linear endomorphism of $\calV^{[r]}_{\calB}$.

            When $r^{-1}=v(\delta)$ and $B=\{b_i\}$ is $\rho$-overconvergent for some $\rho < r$,
            then both $B^r$ and $\calB^r$ are orthonormal bases of $\calV^{[r]}_{\calB}$ by Lemma \ref{l: open discs with changing }
            (and Remark \ref{rem:B resp B}),
            whence $B^r_{k'}$ and $\calB^r_{k'}$ are orthonormal bases of $\calV^{[r]}_{\calB}$ over $\calE'$.
            By \cite[Proposition 7]{SerreOperators}, the Fredholm determinant of a completely continuous $\calE'$-linear map is 
            well-defined with respect to---and independent of the choice of---{\em any} orthonormal basis,
            so that $\Det_{\calE'}(1-s\Flin | \star)$ exists for $\star=\calB^r_{k'}$ and $\star=B^r_{k'}$,
            and both choices give the same power series.
            On the other hand, $B$ is certainly a basis of $\calV$, and if $(a_{ij})$
            is the matrix of $\Flin$ relative to $B^r$, then $(\delta^{(i-j)}a_{ij})$
            is its matrix relative to $B$.  It follows easily that, for the induced $\calE'$-bases
            $B^r_{k'}$ and $B_{k'}$, the products occurring in the definition \eqref{eq: leibniz type formula} of $c_n$ 
            are {\em the  same}, 
            which shows that $\Det_{\calE'}(1-s\Flin | B_{k'})$ exists and equals $\Det_{\calE'}(1-s\Flin | B^r_{k'})$.
    \end{proof}

    \begin{definition}
        Let $f(s) = 1 + a_1s + a_2s^2 + \dots$ be an element of $\calE\powerseries{s}$ that is entire in $s$. We define
        the {\em $v$-adic Newton polygon of $f(s)$}, denoted by $\NP(f)$,
        to be the lower convex hull of the points 
        \[(0,0), (1,v(a_1)), (2,v(a_2)), \dots .\]
        Note that the multisets of slopes of $\NP(f)$ 
        and valuations of the reciprocal roots of $f(s)$ coincide.
    \end{definition}
    \begin{definition}
        Let $\Flin:\calV \to \calV$ be an $\calE'$-linear map and $\calB$ a basis of $\calV$ over $\calE$.
        Assume that 
        the Fredholm determinant of $\Flin$ relative to $\calB_{k'}$ exists and is an entire function.
        Then we define 
        \begin{align*}
            \NP(\Flin|\calB_{k'}) \colonequals  \NP(\Det_{\calE'}(1-s\Flin|\calB_{k'})).
        \end{align*}
    \end{definition}

    	\begin{theorem}[Dwork's Linearization Trick]
     \label{t: Dwork's linearlization trick}
        Let $\calV$ be an $\calE$-Banach space with orthonormal basis $\calB$ and $\Flin: \calV \to \calV$ an 
        $\calE'$-linear map that is $r$-overconvergent with respect to $\calB$, for some $r>0$ with $r^{-1}=v(\delta)$. 
        If $\klin=\Flin^\fielddegree$ is $\calE$-linear %
        then
        \begin{align*}
            \frac{1}{\fielddegree} \NP(\klin|\calB)^{\times \fielddegree} &= \NP(\Flin|\calB_{k'}),
        \end{align*}
        where scaling $\NP(\klin|\calB)$ by $\frac{1}{\fielddegree}$ means each $y$-coordinate of the polygon is scaled by $\frac{1}{\fielddegree}$ and the $\fielddegree$ in the superscript means each slope is repeated $\fielddegree$ times. 
		\end{theorem}
		\begin{proof}
            Thanks to Lemma \ref{lem: technical}, we may replace $\calV$ with $\calV^{[r]}_{\calB}$ and $\calB$
            with $\calB^r$ to assume that $\Flin$ is completely continuous,
            and thereby reduce to proving the claim under the assumption that $\calV$ is finite dimensional over $\calE$.
            In this situation, we have the well-known identity
			\begin{equation}\label{eq: Dwork trick 1}
				\Det_{\calE'}(1-s\klin|\calB_{k'})	=	\mathrm{Nm}_{\calE/\calE'} \left(\Det_{\calE}(1-s\klin|\calB)\right).
			\end{equation}
            Let $\calE_0$ be a finite extension of $\calE'$ containing all roots of $X^{\nu}-1$ and let $\calV_0=\calV\otimes _{\calE'}\calE_0$.
            Then $\calV_0$ is a Banach space over $\calE_0$ and $\calB_{k'}$ is
            an orthonormal basis of $\calV_0$ over $\calE_0$. 
            The $\calE'$-linear map $\Flin$ extends to an $\calE_0$-linear map $\Flin_0=\Flin \otimes 1$ and we have 
            \begin{equation} \label{eq: Dwork trick 2}
                \Det_{\calE'}(1-s\Flin^i|\calB_{k'})=\Det_{\calE_0}(1-s\Flin_0^i|\calB_{k'}),
            \end{equation}
            for any $i\geq 1$.
            Writing $\zeta_1,\dots,\zeta_\nu\in \calE_0$ for the roots of $X^\nu-1$ counted with multiplicity, 
            we have
            \begin{equation}\label{eq: Dwork trick 3}
                \Det_{\calE_0}(1-s^\nu \Flin_0^\nu |\calB_{k'}) = \prod_{i=1}^\nu \Det_{\calE_0}(1-s\zeta_i\Flin_0|\calB_{k'}). 
            \end{equation}
            Replacing $s$ by $s^{\nu}$ in \eqref{eq: Dwork trick 1} and the $i=\nu$ instance of \eqref{eq: Dwork trick 2},
            and combining these with \eqref{eq: Dwork trick 3} gives
            \begin{equation*}
                \mathrm{Nm}_{\calE/\calE'} \left(\Det_{\calE}(1-s^\nu \klin|\calB)\right) =\prod_{i=1}^\nu \Det_{\calE_0}(1-s\zeta_i\Flin_0|\calB_{k'}).
            \end{equation*}
            Each conjugate
            of $\Det_\calE(1 - s^\fielddegree \klin|\calB)$ under the natural action of $\mathrm{Aut}(\calE/\calE')\simeq \Gal(k/k')$
            has the same Newton polygon.   Since $v(\zeta_j)=0$ for all $j$, it follows from \eqref{eq: Dwork trick 2} with $i=1$
            that
            $\det_{\calE_0}(1-s\zeta_i\Flin_0|\calB_{k'})$ has the same Newton polygon as $\Det_{\calE'}(1-s  \Flin|\calB_{k'})$, 
            which completes the proof.
            \end{proof}

    \subsection{Applications to the Cartier operator} \label{ss:cartierestimates}
    We now apply the preceding considerations to study the Cartier operator on the $A$-module
    $Z$ of Definition \ref{definition: limiting objects}.  Recall that by Corollary \ref{cor:rankone}, the element $w_0$ of Definition \ref{def:wa}
    induces an identification of $A\langle x \rangle$-modules
    \begin{equation}
        Z = L \cdot w_0  \frac{dx}{x}\label{eq:Z eq L}.
    \end{equation}
    The Cartier operator $V$ on $Z$ is identified with the map sending $f\cdot w_0 \frac{dx}{x}$
    to $V(\alpha^{-1} f) \cdot w_0 \frac{dx}{x}$, where $V$ on $L=x A \langle x \rangle$ is given by equation \eqref{eq:VonL def} and $\alpha$ is
    the unique element (necessarily a unit) of $A\langle x\rangle $ determined by $Fw_0 = \alpha w_0$.
    The identification \eqref{eq:Z eq L} will be fundamental to our analysis below; however, in order
    to proceed, we must first extend scalars to obtain a $d$-th root of $T$.

\begin{notation} \label{notation: calligraphic}
We set $\calA \colonequals k\powerseries{T^{1/d}}$ and $\calA_\F\colonequals  \F\powerseries{T^{1/d}}$,
equipped with the $T$-adic valuation $v_T$, and write
$\calE\colonequals \Frac(\calA)$ and $\calE_\F\colonequals \Frac(\calA_\F)$ for their fraction fields.
(c.f. Notation~\ref{notation: iwasawa}). 

\end{notation}

We will frequently use the fact that, for an $A$-module $N$ of finite dimension over $k$, we have
\begin{align}\label{eq: dimension compare with d-th root of T}
    \dim_k(N \otimes_{A} \calA) &= d \cdot \dim_k(N).
\end{align}

    \begin{definition} 
    We define $\calZ \colonequals Z \otimes_{A} \calA$,
    and write $\calZ_{\calE}=\calZ\otimes_{\calA}\calE$ for the associated Banach space over $\calE$,
    equipped with the $T$-adic valuation for which $\calZ$ is the set of elements of nonnegative valuation.
    We write $\calA\langle x\rangle$ for the Tate algebra in $x$
    over $\calA$, equipped with its usual Gauss valuation induced by $v_T$, and denote by $\calL \colonequals  L\otimes_A \calA = x\calA\langle x \rangle$
    the ideal of functions vanishing at $x=0$.

    \end{definition}

    \begin{definition}\label{def:Bbasis}
        For each integer $i>0$ and rational $m>0$ with $1/m \in \frac{1}{d}\Z$, we define
        $$
        u_i	 \colonequals r_{\mu_i} x^i w_0 \frac{dx}{x},\quad\text{and}\quad
        u_i^m\colonequals T^{\frac{i}{m}} u_i,
        $$
        where $r_{\mu_i}$ is as in Lemma \ref{lemdef:ra}, and we set
        $\calB \colonequals \{u_i\}_{i >0}$ and $\calB^m \colonequals \{u_i^m\}_{i>0}$.
        We also define $B:=\{x^i w_0 \frac{dx}{x}\}_{i\ge 1}$
        and $B^m:=\{T^{\frac{i}{m}}x^i w_0 \frac{dx}{x}\}_{i\ge 1}$.
    \end{definition}

    In terms of the $\{e_i\}$ introduced in Definition~\ref{defn:muj}, note $T^{\mu_i} u_i = e_i$.

\begin{lemma}\label{l: B's are Schauder basis} %
The set $\calB$ is an orthonormal basis of $\calZ_{\calE}$ over $\calE$.
    \end{lemma}
    \begin{proof}
        Thanks to Lemma \ref{lemdef:ra}, we have
        $u_i \equiv a_ix^i w_0\frac{dx}{x} \bmod T^{1/d} \calZ$
        with $a_i \in k^{\times}$, so $\calB$ is an orthonormal basis of $\calZ_{\calE}$
        by \cite[Lemme 1]{SerreOperators}. 
    \end{proof}

To ease the notational burden, we make the following abbreviations:
    
\begin{definition}\label{def: fildef}
For any subquotient $\calW$ of $\calZ_{\calE}$, we define  $\Fil{\star} \calW := \Fil{\star}_{\calB} \calW$
as in Definition \ref{def: fildefgeneral}.
For $m>0$ we abbreviate $\calZ^m:=\calZ^m_{\calB}$ 
and $\calL^m:=\calL^m_{\mathscr{B}}$ for $\mathscr{B}=\{x^i\}_{i\ge 1}$
as in Definition \ref{def:overconvergent balls}.

\end{definition}

    \begin{remark}
        Writing $\calA\langle x\rangle^m$ for the ring of rigid analytic functions over $\E$ that converge on the open disc $v_T(x)> -1/m$ and have values bounded by 1, we note that $\calL^m\subseteq \calA\langle x\rangle^m$ is the ideal of functions
        vanishing at $0$.
        While $\calA\langle x\rangle^m$  contains $A\langle x\rangle^m \otimes_{A}\calA$, they are not equal.
    \end{remark}

    \begin{lemma}
        \label{l: growth property of Z^m for normal basis}
        The basis $\calB$ is $\frac{d}{p}$-overconvergent with respect to $B$ and for $m\geq \frac{d}{p}$ with $1/m \in \frac{1}{d}\Z$ we have 
        \begin{align*}
            \calZ^m= \calZ^m_B = \calL^m \cdot w_0\frac{dx}{x}.
        \end{align*}
    \end{lemma}
    \begin{proof}
    By Lemma \ref{lemdef:ra} and Corollary \ref{cor:functionsqb}
    we see that $\calB$ is $\frac{d}{p}$-overconvergent with respect to $B$. The second part of the lemma follows from Lemma \ref{l: open discs with changing }.

    \end{proof}

    \begin{lemma}
        \label{l: action of of V on Z^m}
        We have $V(\calZ^d) \subset \calZ^{d/p}$, and $V$ is $d$-overconvergent with respect to $B$ and $\calB$.
    \end{lemma}
    \begin{proof}
        Recall that $V$ is defined on $A\langle x\rangle$ by \eqref{eq:VonL def}, 
        and we extend $V$ to $\calA\langle x\rangle = A\langle x\rangle \otimes_A \calA$
        by letting $V$ act on $\calA$ through $\sigma^{-1}$ on $k$ and the identity map on $T^{1/d}$.
        It follows easily from this definition that $V(\calL^m) \subset \calL^{m/p}$ for all $m>0$ with $1/m\in \frac{1}{d}\Z$. 
        Corollary \ref{cor: Frobenius element growth} gives $\alpha^{-1}\in \calA\langle x\rangle^d $,
        so since 
        $\calL^d$ is an ideal of $\calA\langle x\rangle^d$,
        we conclude in particular that 
         $\alpha^{-1} \calL^d \subseteq \calL^d$. Then using Corollary \ref{cor:rankone} we have
        \[
            V(\calL^d w_0\frac{dx}{x}) = V(\alpha^{-1} \calL^d) w_0 \frac{dx}{x} \subseteq V(\calL^d)w_0\frac{dx}{x} \subseteq \calL^{d/p} w_0 \frac{dx}{x}. \]
        The lemma then follows at once from Lemma \ref{l: growth property of Z^m for normal basis}.
    \end{proof}

    \begin{corollary}\label{c: V-estimate for basis eliment}
        For $j >0$ we have $V(u_j^{d/p}) \in T^{\frac{(p-1)j}{d}}\calZ^{d/p}$.
    \end{corollary}
    \begin{proof}
        Note that $u_j^{d/p}= T^{\frac{(p-1)j}{d}}u_j^{d}$ and apply Lemma \ref{l: action of of V on Z^m}.
    \end{proof}

    \begin{remark}
        Corollary \ref{c: V-estimate for basis eliment} 
        gives bounds on the columns of the matrix of $V$ with respect to $\calB_{\F}^{d/p}$. This in turn gives a bound on the Hodge polygon of $V$ on $\calZ^{d/p}$ in the sense of \cite{Katz}.
    \end{remark}

    \begin{corollary}\label{c:V-estimate}
			For $t \ge 0$, we have
			\begin{equation*}
				\bfV(\Fil{> \filindex} \calZ^{d/p}) \subseteq T^{\frac{(p-1)(\filindex+1)}{d}} \calZ^{d/p}.
			\end{equation*}
	\end{corollary}

        We now prove the following theorem:
        \begin{theorem}\label{t: usable trace formula}
            Let $\calB^{d/p}$ be the basis of $\calZ_{\calE}$ in Definition \ref{def:Bbasis} and $\calB^{d/p}_\F$ 
            as in Definition \ref{defn:schauderbasis}.  
            Then
            \begin{align*}
                \frac{1}{\fielddegree}\NP(L(\chi,s))^{\times \fielddegree} &= \NP(V|\calB_\F^{d/p}).
            \end{align*}
        \end{theorem}

        The main input is the Dwork trace formula applied to $\chi$.

        \begin{remark}
            \label{remark: matrix of V in terms of coefficients of alpha} Write $\displaystyle\prod_{i=0}^{\nu-1} \sigma^{i}(\alpha^{-1})=\displaystyle\sum_{i=0}^\infty b_i x^i$, where $b_i \in A$, and set $b_i=0$ for $i< 0$. Then 
            by the intertwining relation in Corollary \ref{cor:rankone}
            we see that the matrix of $V^\nu$ in terms of the basis 
            $B$ is the matrix $(b_{p^\fielddegree i-j})_{i,j \geq 1}$.
        \end{remark}
        \begin{theorem}[Dwork trace formula]
            \label{t: trace formula}
            We have
            \begin{align*}
                L(\chi,s) &= \Det_{\calE}(1-sV^\nu|B).
            \end{align*}
        \end{theorem}
        \begin{proof}
            This is essentially due to Dwork and can be found in \cite{Dwork-rationality}. For our specific situation, see \cite[\S 4]{Wan-meromorphic_continuation} for a proof and
            \cite[\S 7]{Wan-meromorphic_continuation} for a discussion about adapting the formula to positive
            characteristic. The relevant equation is 7.5 in \cite{Wan-meromorphic_continuation}. 
            In particular, the $L$-function that Wan studies is exactly
            given by the second line of equation \eqref{eq: L-function definition} 
            and the matrix in the Fredholm determinant is the matrix of
            $V^\nu$ by Remark \ref{remark: matrix of V in terms of coefficients of alpha}.
        \end{proof}

        \begin{proof}[Proof of Theorem $\ref{t: usable trace formula}$]
        From Theorem \ref{t: trace formula}, Lemma \ref{l: action of of V on Z^m}, and Theorem \ref{t: Dwork's linearlization trick} we have
        \begin{align*}
            \frac{1}{\fielddegree}\NP(L(\chi,s))^{\times \fielddegree}&=\NP(V|B_{\F}).
        \end{align*}
        By Lemma \ref{l: action of of V on Z^m} and Lemma \ref{l: growth property of Z^m for normal basis}, we may use Lemma \ref{l: open discs with changing } to obtain:
        \begin{align*}
            \frac{1}{\fielddegree}\NP(L(\chi,s))^{\times \fielddegree}&=\NP(V|\calB_{\F}^{d/p}).\qedhere
        \end{align*}

        \end{proof}

\subsection{Newton-Hodge Interaction}
\label{ss: Newton Hodge interaction}

    \begin{definition}\label{def:HPD}
			We define $\HP(d)$ to be the polygon in $\R^2$ connecting the vertices
			\begin{equation*}
				\{(	m, \eta_m	)	: m \geq 0 \} \quad\text{where}\quad \eta_m\colonequals \frac{(p-1)m(m+1)}{2d}.
			\end{equation*}
            That is, $\HP(d)$ is the polygon consisting of segments of slopes $\frac{p-1}{d},2\frac{p-1}{d}, \dots$, each occurring with multiplicity one.
            We let $\HP(d)^{\times \fielddegree}$ denote the polygon where each slope in $\HP(d)$ is repeated $\fielddegree$ times.
		\end{definition}
\begin{theorem}
    \label{theorem: kosters-zhu plus epsilon} 
    For each $\filindex \equiv 0$ or $-1 \pmod{d}$, 
    the $T$-adic Newton polygon of $L(\chi,s)$ has the point $(\filindex,\fielddegree \eta_t)$ as a vertex. Furthermore, if $d|(p-1)$ then there is an equality of polygons
    \begin{equation*}
        \NP (L(\chi,s)) = \fielddegree\HP(d).
    \end{equation*}
\end{theorem}
\begin{proof}
    This is essentially proven in \cite{KostersZhu}, but requires a couple of small observations found in \cite{kramer-milleruptonII}. 
    In \cite[Theorem 1.20]{kramer-milleruptonII} it is shown that for towers with minimal break ratios (there called \emph{strictly stable monodromy}), the $T^\nu$-adic Newton polygon of $L(\chi,s)$ has the point $(\filindex,\eta_t)$ as a vertex for all $t\equiv 0$ or $-1 \bmod d$. Since $v_T(T^\fielddegree) = \fielddegree$ we see that $\NP L(\chi,s)$ has the point $(t, \fielddegree \eta_t)$ as a vertex. 
    In the case where $d|(p-1)$, the third part of \cite[Theorem 1.20]{kramer-milleruptonII} implies
    the equality of polygons.
\end{proof}

\begin{corollary} \label{corollary: touching polygons}
    For each $\filindex \equiv 0$ or $-1 \pmod{d}$, the polygons
    \begin{equation*}
        \NP(\bfV|\calB_\F^{d/p})~\text{and }\HP(d)^{\times \fielddegree}
    \end{equation*}
    coincide at the vertex $(\fielddegree\filindex ,\fielddegree \eta_t)$. Furthermore, if $d|(p-1)$
    the two polygons are the same.
\end{corollary}

\begin{proof}
        From Theorem \ref{t: usable trace formula} we know $\NP(V|\calB_\F^{d/p}) = \frac{1}{\fielddegree} \NP(L(\chi,s))^{\times \fielddegree}$. Thus, from Theorem \ref{theorem: kosters-zhu plus epsilon}
        we conclude $\NP(V|\calB_\F^{d/p})$ contains the vertex $(\fielddegree\filindex ,\fielddegree \eta_t)$. By the definition of $\HP(d)$, we see that
        $\HP(d)^{\times \fielddegree}$ also contains this vertex. When $d|(p-1)$ we know from Theorem \ref{theorem: kosters-zhu plus epsilon} that $\NP(L(\chi,s)) = \fielddegree\HP(d)$. As we know
        $\NP(V|\calB_\F^{d/p}) = \frac{1}{\fielddegree} \NP(L(\chi,s))^{\times \fielddegree}$ this shows
        $\NP(\bfV|\calB_\F^{d/p})=\HP(d)^{\times \fielddegree}$.
\end{proof}

\subsection{Consequences of the coincidence of NP and \texorpdfstring{$\HP(d)$}{HP(d)}} \label{ss: consequences of NP and HP coincidence}
In this subsection we make use of the coincidences between the Newton polygon and $\HP(d)$ to control the behavior of the iterates $V^r$ on $Z$. Fix $\filindex \geq 0$. We have a decomposition of $Z$ as a direct sum of $A$-modules
\begin{equation*}
    Z = \Fil{\leq \filindex} Z \oplus \Fil{> \filindex} Z.
\end{equation*}
and, for each $r\ge 0$, a corresponding representation of $V^r$ as a $2\times 2$ matrix
of $\sigma^{-1}$-semilinear $A$-module homomorphisms
    \[V^r = \begin{bmatrix}
		 \bfA_{r,\filindex} & \bfB_{r,\filindex} \\
		 \bfC_{r,\filindex} & \bfD_{r,\filindex}
		\end{bmatrix}= \begin{bmatrix}
		 \bfA_{r} & \bfB_{r} \\
		 \bfC_{r} & \bfD_{r}
		\end{bmatrix}.
  \]
So, for example, $\bfA_{r,t}$ is an endomorphism of $\Fil{\le \filindex}Z$,  $\bfB_{r,t}$
is a homomorphism from $\Fil{>\filindex}Z$ to $\Fil{\le\filindex}Z$, etc.
By convention, we will omit the $t$ in the subscript except
    when there is no possibility for ambiguity, and we will
    again write $\bfA_r$, $\bfB_r$ etc.~for the induced maps 
    on $\Fil{\bullet} \calZ$ and $\Fil{\bullet} \calZ_{\calE}$.

\begin{remark}
        Note that $A_{r,t}$ is a {\em linear} endomorphism of the finite, free $\calA_{\F}$-module
        $\Fil{\le t}\calZ $; as such, its determinant $\Det_{\calA_{\F}}(A_{r,t})$ is a well-defined element 
        of $\calA_{\F}$.
\end{remark}

    \begin{convention}\label{convention:vertex}
        In the remainder of \S\ref{ss: consequences of NP and HP coincidence} we assume that $(\nu t, \nu \eta_t)$
        is a vertex of $\NP(\bfV|\calB_\F^{d/p})$.
    \end{convention}

    \begin{proposition}\label{prop: columns and determinant of Ak}
We have $v_T(\Det_{\calE_{\F}}(A_{1, t})) = \fielddegree \eta_t$.
    \end{proposition}

    \begin{proof}
        Let $(a_{b,b'})$ be the matrix of $V$ relative to the basis $\calB_{\F}^{d/p}$
        of $\calZ_{\calE}$, considered as a Banach space over $\calE_{\F}$, so
        for each $b=\zeta_ju_i^{d/p}$ in $\calB_{\F}^{d/p}$ we have
        \begin{align*}
            V(b) &= \sum_{b' \in \calB_{\F}^{d/p}} a_{b',b} b'\quad\text{for unique}\ a_{b',b}\in \calA_{\F}.
        \end{align*}
Corollary~\ref{c: V-estimate for basis eliment} shows $v_T(a_{b',b}) \geq \frac{(p-1)i}{d}$. 
        Thus, if $S=\{\zeta_{j_1}u_{i_1}^{d/p}, \dots, \zeta_{j_n}u_{i_n}^{d/p}\}$ is any subset of $\calB_\F^{d/p}$ with cardinality $n=\nu t$ and $\sigma$ any permutation of $S$, we have
        \begin{align}\label{eq: determinant of Ak equation}
            v_T\Big( \prod_{b \in S} a_{\sigma(b),b} \Big ) &\geq \sum_{k=1}^{n} \frac{(p-1)}{d}i_k \ge \sum_{k=1}^n \frac{(p-1)}{d}\left\lceil \frac{k}{\nu} \right\rceil  = \nu \eta_t.
        \end{align}
        Moreover, the second inequality in \eqref{eq: determinant of Ak equation} is {\em strict} when $S\neq \Fil{\leq t} \calB_\F^{d/p}$, since $i_k\ge \lceil k/\nu \rceil$ with equality for all $k\le n$ if and only if $S=\Fil{\le t} \calB^{d/p}_{\F}$.    
        Considering the sum $c_{n}$ in the Leibniz-type formula \eqref{eq: leibniz type formula},
        our assumption on $\NP(\bfV|\calB_\F^{d/p})$ is that $v_T(c_n)=\nu\eta_t$.
        Therefore we see that
        \begin{align*}
             \nu \eta_t =v_T(c_n) &=  v_T\left( (-1)^n \sum_{|S| = n} \sum_\sigma \mathrm{sgn}(\sigma) \prod_{m \in S} a_{\sigma(m),m} \right) \\
            &= v_T\left( (-1)^n \sum_\sigma \mathrm{sgn}(\sigma) \prod_{m \in \Fil{\le n}\calB_{\F}^{d/p}} a_{\sigma(m),m} \right) \\
            &= v_T(\Det_{\calE_{\F}}(A_{1,t})). \qedhere
        \end{align*}
    \end{proof}
    \begin{corollary}\label{cor:invertible}
        The operator $A_1$ is invertible on $\Fil{\leq \filindex}\calZ_{\calE}$.
    \end{corollary}
    \begin{proof}
        This is clear, as $\det_{\calE_\F}(A_1)$ is a unit of $\calE_{\F}$.
    \end{proof}

    \begin{lemma}\label{l:fundamental fact SNF}
        Let $\calV$ be finite dimensional $\calE_\F$-vector space, and $u:\calV\rightarrow \calV$
        an injective (equivalently invertible) $\calE_\F$-linear map.  If $\calD\subseteq \calV$
        is any $\calA_\F$-lattice that is preserved by $u$, then
        $$
            d\cdot v_T(\Det_{\calE_\F}(u)) = \dim_{\F}( \calD / u\calD) = \dim_{\F}(u^{-1}\calD / \calD).
        $$
    \end{lemma}

    \begin{proof}
        Since $u: \calD\rightarrow \calD$ is an isomorphism after inverting $T$, 
        its cokernel is a finitely generated, torsion $\calA_\F$-module, so finite-dimensional
        as an $\F$-vector space.  Noting that $u$ induces an isomorphism of $\calA_\F$-modules
        $u: u^{-1}\calD / \calD \simeq \calD / u\calD$, the result follows from
        the existence of Smith Normal Form for $u$ over the discrete valuation ring $\calA_{\F}$,
        whose maximal ideal is  $T^{1/d}\calA_\F$.
    \end{proof}
    
 \begin{lemma}\label{lemma: inverse of Cartier part}
    As $\calA$-submodules of $\calZ_{\calE}$,
		we have $\bfA_{1,t}^{-1} (\Fil{\leq \filindex}\calZ^{d/p}) = \Fil{\leq \filindex}\calZ^d$.
	\end{lemma}
    \begin{proof}
        By Lemma \ref{l: action of of V on Z^m} we know that $\bfA_{1}(\Fil{\leq \filindex}\calZ^d) \subseteq \Fil{\leq \filindex}\calZ^{d/p}$, which induces a surjection
        of $\calA$-modules
           \begin{equation}\label{eq:Amodsurj} 
           \frac{\Fil{\leq \filindex}\calZ^d}{\bfA_{1}(\Fil{\leq \filindex}\calZ^d)} \twoheadrightarrow \frac{\Fil{\leq \filindex}\calZ^d}{\Fil{\leq \filindex}\calZ^{d/p}}
            \end{equation}
        with source (and hence also target) that is finite dimensional over 
        $\F$ by Convention \ref{convention:vertex} and Corollary \ref{cor:invertible}.
        From Proposition \ref{prop: columns and determinant of Ak} and Lemma \ref{l:fundamental fact SNF},
        we see that the $\F$-dimension of the source is $d\nu\eta_t$. By comparing $u_i^d$ and $u_{i}^{d/p}$ we see that the target is isomorphic as an $\calA$-module to $\oplus_{i\le t} \calA/T^{i(p-1)/d}\calA$
        so also has $\F$-dimension $d\nu\eta_t$, so \eqref{eq:Amodsurj} is an isomorphism, and the containment an equality.
    \end{proof}

    We now turn our attention to the operators $\bfA_{r}$. For $r\ge 1$, we have the following relations:
    \begin{align*}
        \bfA_{r} &= \bfA_{1}\bfA_{r-1} + \bfB_{1} \bfC_{r-1} \\
        \bfC_{r} &= \bfC_{1}\bfA_{r-1} + \bfD_{1} \bfC_{r-1}.
    \end{align*}
     Our goal is to study the inverse images of $\Fil{\leq t}\calZ^{d/p}$ under the $\bfA_{r}$. The next lemma does so inductively by approximating $\bfA_{r}^{-1}$ with the composite $\bfA_{r-1}^{-1}\bfA_{1}^{-1}$.

     \begin{lemma}\label{lemma: inverse of iterated Cartier part}
         For all $r \geq 0$, $\bfA_{r}$ induces an invertible operator on $ \Fil{\leq \filindex}\calZ_{\calE}$ and   
         \begin{subequations}
         \begin{equation}\label{eq:ArCra}
             \bfA_{r}^{-1} (\Fil{\leq \filindex}\calZ^{d/p}) = \bfA_{r-1}^{-1} \bfA_{1}^{-1} (\Fil{\leq \filindex}\calZ^{d/p}) \subseteq T^{- \filindex\frac{(r-1)(p-1)}{d}}\Fil{\leq \filindex}\calZ^d
         \end{equation}
         as submodules of $\calZ_{\calE}$.
         Moreover, inside $\calZ_{\calE}$, we have
         \begin{equation}\label{eq:ArCrb}
             \bfC_{r}\bfA_{r}^{-1}(\Fil{\leq \filindex}\calZ^{d/p}) = \bfC_{r}\bfA_{r-1}^{-1} \bfA_{1}^{-1} (\Fil{\leq \filindex}\calZ^{d/p}) \subseteq \Fil{> \filindex}\calZ^{d/p}.
         \end{equation}
         \end{subequations}
     \end{lemma}
     \begin{proof}
         We proceed by induction on $r$. The base case $r=1$ of \eqref{eq:ArCra} 
         follows from Lemma \ref{lemma: inverse of Cartier part}, whence \eqref{eq:ArCrb} for $r=1$
         is a consequence of Lemma \ref{l: action of of V on Z^m}. 
         Assume inductively that \eqref{eq:ArCra} and \eqref{eq:ArCrb}
         hold with $r-1$ in place of $r$. Fix $x \in \Fil{\leq \filindex}\calZ^{d/p}$, and set
         $z \colonequals \bfA_{1}^{-1}(x)$ and $y  \colonequals \bfA_{r-1}^{-1}(z).$
         By the base case of \eqref{eq:ArCra}, $z$ is contained in $\Fil{\leq \filindex}\calZ^d \subseteq T^{-\filindex \frac{p-1}{d}}\Fil{\leq t}\calZ^{d/p}$.   The induction hypothesis \eqref{eq:ArCra} for $r-1$ then tells us that 
        \begin{equation*}
           y \in T^s \Fil{\leq \filindex}\calZ^d\quad\text{for}\quad s \colonequals -\filindex\frac{p-1}{d} - \filindex\frac{(r-2)(p-1)}{d} = - \filindex\frac{(r-1)(p-1)}{d},
        \end{equation*}
        which yields the asserted containment in \eqref{eq:ArCra}.
        Now since $z\in \Fil{\le t} \calZ^d$, Lemma \ref{l: action of of V on Z^m} gives
        \begin{equation}\label{eq:induction containments 1}
            \bfC_{1}(z) \in \Fil{> \filindex}\calZ^{d/p}.
        \end{equation}
        On the other hand, we have
        \begin{equation}\label{eq:induction containments 2}
            \bfC_{r-1}(y) = \bfC_{r-1}\bfA_{r-1}^{-1}(z) \in  \bfC_{r-1}\bfA_{r-1}^{-1} (T^{-\filindex \frac{p-1}{d}}\Fil{\leq t}\calZ^{d/p}) \subseteq  T^{-\filindex\frac{p-1}{d}} \Fil{> \filindex}\calZ^{d/p}.
        \end{equation}
        wherein the final containment follows from the inductive hypothesis. %
        By Corollary \ref{c:V-estimate} we also have 
        \begin{equation}\label{eq:induction containments 3}
            \bfD_{1} \Fil{>\filindex}\calZ^{d/p} \subseteq T^{(\filindex+1)\frac{p-1}{d}}\Fil{>\filindex}\calZ^{d/p}.
        \end{equation}
        Combining \eqref{eq:induction containments 1}--\eqref{eq:induction containments 3}, we 
        see that 
        $$ \bfC_{r}(y)   =   \bfC_{1}\bfA_{r-1}(y) + \bfD_{1} \bfC_{r-1}(y) = \bfC_{1}(z) + \bfD_{1} \bfC_{r-1}(y)$$
        lies in
        \begin{align*}
                \Fil{> \filindex} \calZ^{d/p} + \bfD_{1} \left( T^{-\filindex\frac{p-1}{d}} \Fil{> \filindex} \calZ^{d/p} \right)
                & \subseteq  \Fil{> \filindex} \calZ^{d/p} + T^{(\filindex+1)\frac{p-1}{d}} \cdot T^{-\filindex\frac{p-1}{d}} \Fil{> \filindex} \calZ^{d/p} \\
                &=  \Fil{> \filindex}\calZ^{d/p},
        \end{align*}
        which is the second containment claimed in \eqref{eq:ArCrb}.

        It remains to prove the first equality in \eqref{eq:ArCra},
        which also gives the first equality in  \eqref{eq:ArCrb}.
        By Corollary \ref{c:V-estimate} we have 
        \begin{equation}\label{eq:Bcontainment}
            \bfB_{1} \Fil{>\filindex}\calZ^{d/p} \subseteq T^{(\filindex+1)\frac{p-1}{d}}\Fil{\leq\filindex}\calZ^{d/p}.
        \end{equation}
        On the other hand, the very definitions of $z= \bfA_{1}^{-1}(x)$ and 
        $y=\bfA_{r-1}^{-1}(z)$ give
        \begin{align*}
                x - \bfA_{r}\bfA_{r-1}^{-1}\bfA_{1}^{-1}(x) &= x - \bfA_r(y)  
                 = x - \bfA_{1}\bfA_{r-1}(y) - \bfB_{1} \bfC_{r-1}(y)  
                =  -\bfB_{1} \bfC_{r-1}(y),
        \end{align*}
        so by  \eqref{eq:induction containments 2} and \eqref{eq:Bcontainment} we have
        \begin{equation*}
                x - \bfA_{r}\bfA_{r-1}^{-1}\bfA_{1}^{-1}(x) \in  T^{\frac{p-1}{d}} \Fil{\leq \filindex}\calZ^{d/p}
        \end{equation*}
        for all $x\in \Fil{\leq \filindex}Z^{d/p}$.  Thus, the operator $\bfA_{r}\bfA_{r-1}^{-1}\bfA_{1}^{-1}$
        induces an endomorphism of $\Fil{\leq \filindex}\calZ^{d/p}$ that is congruent to the identity modulo $T^{\frac{p-1}{d}}$;
        it is therefore invertible on $\Fil{\leq \filindex}\calZ^{d/p}$, so certainly invertible on $\Fil{\leq \filindex}\calZ_{\calE}$.
        Moreover, as $\bfA_{r}\bfA_{r-1}^{-1}\bfA_{1}^{-1}$ carries $\Fil{\leq \filindex}\calZ^{d/p}$ isomorphically onto itself, 
        we have an equality of submodules of $\Fil{\leq \filindex}\calZ_{\calE}$
        \begin{equation*}
            \bfA_{r-1}^{-1} \bfA_{1}^{-1} (\Fil{\leq \filindex}\calZ^{d/p}) = \bfA_{r}^{-1} (\Fil{\leq \filindex}\calZ^{d/p}),
        \end{equation*}
        as claimed.
     \end{proof}

    \begin{corollary}\label{cor: inverse of iterated Cartier over OmegaRinfty}
       For all $r \geq 0$ we have
       \begin{align*}
           \bfA_{r}^{-1} (\Fil{\leq t}\calZ) \subseteq T^{-\filindex \frac{r(p-1)+1}{d}} \Fil{\leq t}\calZ^{d} \subseteq T^{-\beta(\filindex)} \Fil{\leq t}\calZ^{\frac{d}{p+1}} \quad\text{where}\quad
           \beta(\filindex) = \filindex\frac{r(p-1) + p+1}{d}.
       \end{align*}
       Moreover, for any $s\in \frac{1}{d}\Z$,
       \begin{equation*}
           \bfC_{r}\bfA_{r}^{-1} (T^s\Fil{\leq t}\calZ) \subseteq T^{s+\frac{p}{d}}\Fil{> \filindex}\calZ.
       \end{equation*}
    \end{corollary}

    \begin{proof}
       The first statement follows from Lemma \ref{lemma: inverse of iterated Cartier part} and the evident containments
    \begin{equation}\label{eq: Zcontainments}
    \Fil{\leq \filindex}\calZ \subseteq T^{-\frac{\filindex p}{d}}\Fil{\le\filindex}\calZ^{d/p}
    \quad\text{and}\quad
    \Fil{\leq \filindex}\calZ^d \subseteq T^{-\frac{\filindex p}{d}}\Fil{\le\filindex}\calZ^{\frac{d}{p+1}}.
    \end{equation}
    Similarly, since $\bfC_{r}$, $\bfA_{r}^{-1}$ are $\calA$-semilinear, the
    second assertion is a consequence of the first containment in 
    \eqref{eq: Zcontainments},
    Lemma \ref{lemma: inverse of iterated Cartier part}, and the evident
    containment
    \begin{equation*}
        T^{-\frac{(t+1)p}{d}} \Fil{>\filindex} \calZ^{d/p} \subseteq \Fil{>\filindex} \calZ. \qedhere
    \end{equation*}   
    \end{proof}

    \begin{corollary}
        \label{c: filtered part for alteration}
        Let $x \in \calZ_{\calE}$. There exists a unique element $y_r \in \Fil{\leq \filindex}\calZ_{\calE}$ such that $V^r(x+y_r) \in \Fil{>\filindex} \calZ_{\calE}$. Furthermore, if $x \in Z_E$, then $y_r \in \Fil{\leq \filindex}Z_E$.
    \end{corollary}
    \begin{proof}
        Write $V^r(x) = z + z'$, where $z \in \Fil{\leq \filindex}\calZ_{\calE}$
        and $z' \in \Fil{>\filindex} \calZ_{\calE}$. We know from Lemma 
        \ref{lemma: inverse of iterated Cartier part} that $A_r$ is invertible on 
        $\Fil{\leq \filindex}\calZ_{\calE}$. Let $y_r = A_r^{-1}(-z)$. Then,
        $V^r(x + y_r) = z' + C_r(y_r) \in \Fil{>\filindex} \calZ_{\calE}$.
        If $x \in Z_E$, using the fact that $V$ (respectively $A_r$) is an $E_{\F}$-linear map on $Z_E$ (respectively $\Fil{\leq \filindex} Z_E$),
        we see that $z$ and $y_r$ are contained in $\Fil{\leq \filindex} Z_E$.
    \end{proof}
    \begin{lemma}
        \label{lemma: filtered part for alterations lies in L^{d/p+1}}
        Let $s \in \frac{1}{d}\Z$
        and let $x \in T^{s + \frac{\filindex p}{d}}\calZ^d$. Let $y_r$ be the unique element of $\Fil{\leq \filindex}\calZ_\calE$  such that 
        $V^r(x+y_r) \in \Fil{> \filindex}\calZ_\calE$ from Corollary \ref{c: filtered part for alteration}. Then
        \begin{align*}
            y_r &\in T^s\Fil{\leq \filindex}\calZ^{\frac{d}{p+1}}\\
            V^r(x+y_r) &\in T^{s + t\frac{p}{d} + \frac{(t+1)(r-1)(p-1)}{d}} \Fil{> \filindex}\calZ^{d/p}.
        \end{align*}
    \end{lemma}
    
    \begin{proof}
        We proceed by induction on $r$. By Lemma \ref{l: action of of V on Z^m},
        we may write
        \begin{equation*}
            V(x) = z + z'\quad\text{where}\quad z \in T^{s+\frac{\filindex p}{d}}\Fil{\leq \filindex} \calZ^{d/p}\ \text{and}\  z' \in T^{s+\frac{\filindex p}{d}}\Fil{> \filindex} \calZ^{d/p}.
        \end{equation*}
        Set $y = \bfA_{1}^{-1}(-z)$, and note that 
        $$y \in T^{s + \frac{\filindex p}{d}}\Fil{\leq \filindex}\calZ^d \subseteq T^s\Fil{\leq \filindex}\calZ^{\frac{d}{p+1}}$$
        thanks to %
        Lemma \ref{lemma: inverse of Cartier part}.
        By construction, we have
        $$
        V(x + y) =V(x) + V(y) = z+z' + \bfA_1(y)+\bfC_1(y) = z'+\bfC_1 \bfA_1^{-1}(-z)
        $$
        which lies in $T^{s + \frac{\filindex p}{d}}\Fil{> \filindex}\calZ^{d/p}$
        thanks to \eqref{eq:ArCrb} with $r=1$. Thus, $y_1=y$ and the base case holds.

        Now assume the result holds for 
        $r-1$, i.e.~that there exists $y \in T^s \Fil{\le t} \calZ^{\frac{d}{p+1}}$ such that
        \begin{equation*}
            V^{r-1}(x + y) \in T^{s+t\frac{p}{d} + \frac{(t+1)(r-2)(p-1)}{d}}\Fil{> \filindex}\calZ^{d/p}.
        \end{equation*}
        Applying $V$ and appealing to Corollary \ref{c:V-estimate}, we have
        $\bfV^r(x+y) \in T^{s+t\frac{p}{d} + \frac{(t+1)(r-1)(p-1)}{d}} \calZ^{d/p}$.
        As before, we write 
        \[ V^r(x+y) = z + z' \quad \text{where} \quad 
        z\in T^{s+t\frac{p}{d} + \frac{(t+1)(r-1)(p-1)}{d}} \Fil{\le t}\calZ^{d/p} \ \text{and} \ z' \in T^{s+t\frac{p}{d} + \frac{(t+1)(r-1)(p-1)}{d}} \Fil{> t}\calZ^{d/p}.
        \]
        Setting $\tilde{y}\colonequals A_{r}^{-1}(-z)$, thanks to \eqref{eq:ArCra}, we have
        \[
            \tilde{y} \in T^{s+ t\frac{p}{d} + \frac{(r-1)(p-1)}{d}}\Fil{\leq \filindex}\calZ^d\subseteq T^s\Fil{\leq\filindex}\calZ^{\frac{d}{p+1}} 
        \] 
        and, by construction,
        \[
            V^r(x+y+ \tilde{y}) = V^r(x+y)+V^r(\tilde{y})= z+z'+A_r(\tilde{y}) + C_r(\tilde{y}) = z'+C_rA_r^{-1}(-z),
        \]
        which is contained in
        \[
             T^{s+t\frac{p}{d} + \frac{(t+1)(r-1)(p-1)}{d}}\Fil{> \filindex}\calZ^{d/p}
        \]
        due to \eqref{eq:ArCrb} and the definition of $z'$. By the uniqueness of $y_r$ we have $y_r=y+\tilde{y}$ and the result holds. \qedhere

    \end{proof}

 \begin{lemma}
     \label{lemma: endpoint of polygon for iterates}
     For $r\geq 1$ we have $v_T(\Det_{\calE_\F}(\bfA_{r})) = rv_T(\Det_{\calE_\F}(\bfA_{1}))$.
 \end{lemma}
 \begin{proof} 
    We proceed by induction on $r$, with base case $r=1$ tautological. Assume that $r>1$ and the result holds for $\bfA_{r-1}$. Set $W = \Fil{\leq \filindex}\calZ^{d/p}$. By Lemma \ref{lemma: inverse of Cartier part} we have $\bfA_{1}^{-1}(W) = \Fil{\leq \filindex}\calZ^d$. Since this is stable under $\bfA_{r-1}$ (Lemma \ref{l: action of of V on Z^m}), we see that
    \begin{equation*}
        \bfA_{1}^{-1}(W) \subseteq \bfA_{r-1}^{-1}(\bfA_{1}^{-1}(W)).
    \end{equation*}
    From Lemma \ref{lemma: inverse of iterated Cartier part} we also know that $\bfA_{r-1}^{-1}(\bfA_{1}^{-1}(W)) = \bfA_{r}^{-1}(W)$. By our inductive hypothesis and Lemma \ref{l:fundamental fact SNF}, we compute
    \begin{align*}
        d\cdot v_T(\Det_{\calE_\F}(\bfA_{r})) &= \dim_\F(\bfA_{r}^{-1}(W) / W) \\
        &= \dim_\F(\bfA_{r-1}^{-1}(\bfA_{1}^{-1}(W))/\bfA_{1}^{-1}(W) ) + \dim_\F(\bfA_{1}^{-1}(W)/W)\\
        &=  d\cdot v_T(\Det_{\calE_\F}(\bfA_{r-1})) + 
        d\cdot v_T(\Det_{\calE_\F}(\bfA_{1}))   \\
        &= r d\cdot v_T(\Det_{\calE_\F}(\bfA_{1})).\qedhere
    \end{align*}  
 \end{proof}

\section{Calculation of higher \texorpdfstring{$a$}{a}-numbers} 
 \label{s:calculation of higher a numbers}
As before let $\{X_n\}$ be a $\Z_p$-tower over $\PP^1$ with minimal break ratios and ramification invariant $d$. 
 The goal of this section is to give an explicit formula relating higher $a$-numbers of $X_n$ and certain lattice point counts.  
 Throughout, we fix both $d$ and $p$.
 
 \begin{definition} \label{defn:t and point count}
 Let $n,t$ be nonnegative integers.  We define
 \begin{itemize}
    \item $\dnom:=(r+1)p-(r-1) = (p-1)r+(p+1).$
     \item $\Delta_n(t) \colonequals \Delta_n \cap \left \{(i,j) \in \Z_{\geq 1}^2: i > t \right \}$, where $\Delta_n$ is as in Definition~\ref{defn:muj}.
 
     \item $t(n) \colonequals d \left \lfloor (p^n-1)/\dnom \right \rfloor$. 

        \item $s(n)\colonequals p^n-1 \bmod \delta = p^n - 1 - t(n)\delta/d$.

     \item $\lambda \colonequals \displaystyle \max_{i \geq 1} \left |\mu_i - i \frac{p+1}{d} \right|$, where $\mu_i$ is defined as in Definition~\ref{defn:muj}.

    \item $\varepsilon(n)\colonequals s(n)+1+\lambda-\delta/d = p^n - (t(n)+1)\delta/d+\lambda$.

     \item $t'(n) \colonequals  \left \lfloor d p^n/\dnom \right \rfloor$.

 \end{itemize}
Note that $\lambda\in \frac{1}{d}\Z$ and that $\lambda<1$ by Corollary~\ref{cor:muestimate}.  
 \end{definition}

 Our main result is the following. 
 
\begin{theorem} \label{theorem:higheranumber}
 There exists an explicit constant $C(p,d,r)$ depending only on $p,d$, and $r$ such that 
\begin{equation}
0 \leq  \left( \frac{r(p-1) t(n)( t(n)+1)}{2d} + \# \Delta_n(t(n)) \right )  -a^{(r)}_n \leq C(p,d,r)
 \end{equation}
 for all $n\ge 1$. 
 When $d|(p-1)$ $($respectively $d\le p+1$$)$ we have
 \begin{equation}\label{eq: MT second case}
 a^{(r)}_n = \frac{r(p-1) t'(n)(t'(n)+1)}{2d} + \# \Delta_n(t'(n)) .
 \end{equation}
 for all $n$ $($respectively all $n$ satisfying $s(n) < \frac{\delta}{d}-1$$)$.
\end{theorem} 

\begin{remark}
The constant $C(p,d,r)$ is made explicit in Proposition~\ref{prop:piece2}.
By definition, $s(n)$ is the remainder upon division of $p^n-1$ by $\delta$,
so lies in the interval $[0, \delta-1]$.
When $d\le p+1$, so $\delta/d - 1 > 0$ for all $r$, the
condition $s(n) < \frac{\delta}{d}-1$ says that 
this remainder actually lies in the smaller interval $[0,\delta/d-1)$.
If $r\not\equiv 1\bmod p$, we have $(\delta,p)=1$  
and $s(n)=0$ for all $n$ with $n\equiv 0 \bmod m$, where $m$ is the order
of $p$ modulo $\delta$. It follows that when $d\le p+1$, for each $r$
there is a modulus $m$ so that the exact formula \eqref{eq: MT second case} holds
for all $n\equiv 0\bmod m$; in particular, this formula holds for a positive proportion
of positive integers $n$.
One can derive analogous statements when $r\equiv 1\bmod p$ as well:
for example, when $r=p+1$ and $d< p+1$ (so $d\le p-1$), \eqref{eq: MT second case}
holds for all {\em odd} $n$.
\end{remark}

\subsection{Cutoff parameters}

 \renewcommand{\Fil}[1]{\textup{Fil}^{ >  #1 }}
\renewcommand{\oFil}[1]{\textup{Fil}^{ \leq  #1 }}
\newcommand{\dFil}[2]{\textup{Fil}^{ ( #1,  #2]}}

For  $n,t \in \Z_{\geq 1}$,  we use the filtrations
$\Fil{t}(M)$ (respectively $\oFil{t}(M)$) on $M$ and $\Fil{t}(M_n)$ (respectively $\oFil{t}(M_n)$) on the sub-quotient $M_n$ introduced in Section~\ref{ss:cartierestimates}.

\begin{lemma} \label{lemma:breaking up the kernel}
For any $n,t \in \Z_{\geq 0}$, the projection map $\prt{t} : Z \to \Fil{t}(Z)$ induces an exact sequence
        \begin{equation*} 
	\begin{tikzcd}
		0 \arrow[r] & \oFil{t}(M_n) \arrow[r] & M_n \arrow[r,"\prt{t}"]& \Fil{t}(M_n) \arrow[r] & 0.
	\end{tikzcd}
	\end{equation*}
Furthermore, 
\begin{equation*}
\label{cor: breaking up the kernel eq1} 
           a^{(r)}_n = \dim_k(\ker(\bfV^r|_{M_n})) = \dim_k(\ker(\bfV^r|_{\oFil{t}(M_n)})) + \dim_k(\prt{t}(\ker(\bfV^r|_{M_n}))). 
\end{equation*}
\end{lemma}

\begin{proof}
The first claim follows from Proposition~\ref{prop:omega as T module}, Proposition~\ref{prop: control type results}, and the definition of $\Fil{\star}$.  The second
claim follows from linear algebra.
\end{proof}

We will analyze the contributions from $\ker(\bfV^r|_{\oFil{t}(M_n)}))$ and $\prt{t}(\ker(\bfV^r|_{M_n}))$ separately.  The parameter $t$ will ultimately be chosen to be $t(n)$ or $t'(n)$: in the picture of Figure~\ref{fig:picture}, it controls the location of the vertical line dividing the two cases of the argument.  As we will now explain, the parameter $t'(n)$ would yield the strongest result, but is not always a valid choice, while $t(n)$ is always valid.

\begin{definition}\label{def: cuttoff}
    Let $\eta_m=(p-1)m(m+1)/2d$ be as in Definition \ref{def:HPD}.
    A {\em cutoff parameter} is any $t\in \Z_{\ge 0}$ with $\left(\fielddegree t,\fielddegree\eta_t\right)$ a vertex of the Newton polygon $\NP(\bfV|\calB_\F^{d/p})$.  Given $n$, we consider the following conditions on a cutoff parameter $t$:
    \begin{enumerate}[label*={(\Roman*)}]
        \item\label{condition I} 
            $\displaystyle p^n - tr\frac{p-1}{d}  \geq t\frac{p+1}{d} +\lambda.$
        \item\label{condition II} 
            $\displaystyle p^n - tr\frac{p-1}{d}  \geq \mu_t.$
    \end{enumerate}
\end{definition}

\begin{remark}\label{rem:cutoffsets}
    Since $t\frac{p+1}{d}+\lambda \ge \mu_t$ by Definition \ref{defn:t and point count}, 
    any cutoff parameter satisfying \ref{condition I} also satisfies \ref{condition II}.
    By Corollary \ref{corollary: touching polygons}, when $d|(p-1)$,
    the set of cutoff parameters is the set
    of all nonnegative integers. For general $d$, the set of cutoff points will depend on the individual tower. However, by Corollary \ref{corollary: touching polygons} we know the set of cutoff points contains all nonnegative integers congruent to $0$ or $-1$ modulo $d$.
\end{remark}

\begin{definition} \label{defn:parameter D}
For a cutoff parameter $t$ satisfying \ref{condition I}, we define %
\[
D_t := \max\{1,p^n - t\frac{ \dnom}{d} - \frac{r(p-1)+1}{d} + 1-\lambda\} 
\]
\end{definition}
Note that $D_t$ a priori depends on $p,d,r$ {\em and} $n$;
for the choice $t=t(n)$, we show that $D_t$ is bounded independently of $n$.

\begin{lemma} \label{lem:parameters}
For fixed $n$:
\begin{enumerate}[(i)]
\item \label{parameters:general}  $t=t(n)$ is a cutoff parameter satisfying \ref{condition I} and 
\[
        D_t  = \max\{1,s(n) - \frac{r(p-1)+1}{d}+2-\lambda\} \le \left(r(p-1)+1\right)\left(1-\frac{1}{d}\right)+2\left\lfloor\frac{p}{2}\right\rfloor+1-\lambda
\]

\item \label{parameters:d|p-1} Let $t = t'(n)$. Then $(t+1) \dnom \geq d p^n$.  If $d|(p-1)$, then $t$ is
a cutoff parameter satisfying \ref{condition II}.

\item \label{parameters:r=1} %
If $r=1$ and $p>2$, then $t = d (p^{n-1}-1)/2$ is a cutoff parameter satisfying \ref{condition I} and 
$D_t = p\left(1-\frac{1}{d}\right)+1-\lambda$.  %

\end{enumerate}
\end{lemma}

\begin{proof}
By definition \ref{defn:t and point count}, $t(n)$ is a nonnegative integral multiple of $d$, so is a cutoff parameter
by Remark \ref{rem:cutoffsets}.  Likewise, $t'(n)$ is a nonnegative integer, 
so is also a cutoff parameter when $d|(p-1)$. 
To complete the proof of \ref{parameters:general}, write $p^n-1= q\dnom+s(n)$ with $q,s\in \Z$ and $0\le s(n) \le \dnom-1$.
We first claim that $s(n) \le \dnom -2$ if $p$ is odd.  For this, 
it suffices to prove that $\dnom$ does not divide $p^n$, 
which is clear as $\dnom = (p-1)r + (p+1)$ is always even when $p$ is odd. Thus, we have an upper bound on $s(n)+1$ that can be expressed concisely as:
\begin{equation}\label{eq: bound on s(n)}
    s(n) +1\le  \delta + (2 \left\lfloor \frac{p}{2} \right \rfloor -p).
\end{equation}
By definition, $t=t(n)=dq$ so that $p^n -t\dnom/d = p^n-q\dnom = s(n)+1$, which combined with \eqref{eq: bound on s(n)} gives 
$$D_t = \max\{1,s(n)+1 - \frac{r(p-1)+1}{d}+1-\lambda \}\le  \dnom  +(2\left\lfloor\frac{p}{2}\right\rfloor-p) -\frac{r(p-1)+1}{d}+1-\lambda.$$
Recalling the definition $\dnom=(p-1)r+(p+1)$ and simplifying completes the proof.

To prove \ref{parameters:d|p-1}, let $t=t'(n)$.
The first claim that $(t+1)\dnom \geq d p^n$ is elementary.
For the second, assume $d|(p-1)$ 
and define $\psi := \mu_{t} - \left \lfloor t \frac{p+1}{d}  \right \rfloor$,
so that $\psi\in \{0,1\}$ by Proposition \ref{prop:muestimates}. Setting
$x := dp^n / \dnom$, we have
$\{x\} := x - \lfloor x \rfloor = x - t$.
Using Proposition~\ref{prop:muestimates} and the assumption $d|(p-1)$, we see that
\[
r \frac{p-1}{d}t + \mu_t = \left \lfloor t \frac{r(p-1) + (p+1)}{d}  \right \rfloor + \psi = \left \lfloor t\dnom/d  \right \rfloor + \psi.
\]
Now note that
\[
\left \lfloor  t \dnom/d \right \rfloor = \left \lfloor  (x - \{x\}) \dnom/d \right \rfloor = \left \lfloor p^n - \dnom \{x\}/d \right \rfloor = p^n + \left \lfloor- \dnom \{x\}/d \right \rfloor.
\] 
If $\{x\} \neq 0$, then $\lfloor -\{x\}\dnom/d \rfloor <0$ and since $\psi \leq 1$ we conclude that $r \frac{p-1}{d} t + \mu_t \leq p^n$,
which shows that $t$ satisfies \ref{condition II} in this case. 
Alternatively, if $\{x\} =0$, then $x =t\in \Z$. Writing $y$ for the denominator of $\dnom/d$ when written in lowest terms,
it follows from the definition of $x$ that we must have $y | x$.  As $d |(p-1)$ and $\dnom/d = (p-1)r/d+(p+1)/d$,
we see that $y$ is also the denominator of $(p+1)/d$ when written in lowest terms.  It follows that 
$\frac{p+1}{d} x = \frac{p+1}{d} t$ is an integer, so $\psi=0$ by the definition 
    of $\mu_t$ (Definition \ref{defn:muj}). Thus, we again have $r \frac{p-1}{d} t + \mu_t = p^n$, 
    and $t$ satisfies \ref{condition II}.

For \ref{parameters:r=1}, note that $\delta=2p$ when $r=1$, so if $p$ is
odd we have $t=t(n) = d(p^{n-1}-1)/2$. Thus, by \ref{parameters:general} we see $t=d(p^{n-1}-1)/2$ is a cutoff parameter satisfying \ref{condition I}. We have
$t \cdot (\delta/d) = p^n-p.$
Thus, $D_t=\max\{1, p-p/d + 1-\lambda\}$, and since $1-\lambda \ge 1/d$, this maximum is $p(1-1/d)+1-\lambda$, as claimed.
\end{proof}

\begin{remark}
If $p>2$, the proof of Lemma \ref{lem:parameters} \ref{parameters:general}
shows that $t(n)=d\lfloor p^n/\dnom \rfloor$. In particular, for $p$ odd,
$t(n)$ (respectively $t'(n)$) is the largest element of $d\Z$ (resp.~$\Z$)
 less than or equal to $d p^n/\dnom$.
\end{remark}

\subsection{Breaking Up the Computation of \texorpdfstring{$a^{(r)}_n$}{arn}} 
The proof of Theorem \ref{theorem:higheranumber} will depend on two propositions: Proposition \ref{prop:piece1} corresponds to the region
left of the vertical line in Figure \ref{fig:picture} while Proposition 
\ref{prop:piece2} corresponds to the region to the right of the vertical line.
The location of the vertical line is controlled by the cutoff parameter $t$,
which will ultimately be selected to optimize our calculations

We will prove the following two propositions in Sections~\ref{ss:proof1}-\ref{ss:proof2b}.

\begin{proposition} \label{prop:piece1}
Let $n$ be a positive integer and $t$ a cutoff parameter.  Assume either that $t$ satisfies \ref{condition I}, or that $d|(p-1)$ and $t$ satisfies \ref{condition II}. 
Then
        \begin{equation}\label{eq:piece1}
            \dim_k\left (\ker \left(\bfV^r|_{\oFil{t}(M_n)}\right)\right) = r\eta_t=  \frac{r(p-1) t( t+1)}{2d} .
        \end{equation}      
\end{proposition}

 \begin{proposition} \label{prop:piece2} 
 Let $n$ be a positive integer and $t$ a cutoff parameter.
 \begin{enumerate}
 \item \label{prop:piece2 part1} Assume $t$ satisfies \ref{condition I} %
and let $D:=D_t$ be as in Definition 
\ref{defn:parameter D}.  Define
\[
    \varepsilon:=D-1+2\lambda-\frac{p}{d}\quad\text{and}\quad
    C(p,d,r) \colonequals \begin{cases}
        \left(1+\left\lfloor\frac{d\varepsilon}{p}\right\rfloor\right)\left(\varepsilon - \frac{p}{2d}\left\lfloor\frac{d\varepsilon}{p}\right\rfloor\right) & \text{if}\ \varepsilon > 0 \\
        \phantom{\left(1+\left\lfloor\frac{d\varepsilon}{p}\right\rfloor\right)} 0 & \text{if}\ \varepsilon \le 0
    \end{cases}
\]
Then
  \begin{equation*}
  0 \leq \# \Delta_n(t) - \dim_k \prt{t}(\ker(\bfV^r|_{M_n})) \leq 
    C(p,d,r). %
  \end{equation*}
If moreover $t=t(n)$, then $\varepsilon=\max\{2\lambda-\frac{p}{d},\varepsilon(n)\}$ 
with $\varepsilon(n)$ as in Definition \ref{defn:t and point count}.

 \item\label{prop:second half stronger} Set $t=t'(n)$ and assume that either $d|(p-1)$,
 or that $d\le p+1$ and $t(n)=t'(n) = t$. Then
    $$\# \Delta_n(t) = \dim_k \prt{t}(\ker(\bfV^r|_{M_n})).$$ 
 \end{enumerate}
 \end{proposition}

\begin{proof}[Proof of Theorem~\ref{theorem:higheranumber}]
For the general case, take $t = t(n)$,
which is a cutoff parameter satisfying 
\ref{condition I} thanks to Lemma \ref{lem:parameters} \ref{parameters:general}. As such, we may apply 
Proposition \ref{prop:piece1} and Proposition \ref{prop:piece2} \ref{prop:piece2 part1}, which combine with 
Lemma~\ref{lemma:breaking up the kernel} to give
\[
0 \leq \frac{r(p-1) t(n)(t(n)+1)}{2d} + \# \Delta_n(t(n)) -  a^{(r)}_n  \leq C(p,d,r)
\]
with $C(p,d,r)$ depending only on $p,d$, and $\varepsilon=\max\{2\lambda-\frac{p}{d},\varepsilon(n)\}$ as in Proposition  \ref{prop:piece2} \eqref{prop:piece2 part1}. It follows from Definition
\ref{defn:t and point count} 
and \eqref{eq: bound on s(n)}
that 
\[
    \varepsilon(n) \le \delta\left(1-\frac{1}{d}\right)+\lambda+2\left\lfloor \frac{p}{2}\right\rfloor-p,
\]
which is visibly independent of $n$, so that $C(p,d,r)$
is likewise independent of $n$.  Explicitly, 
\[
        C(p,d,r) \le \frac{d}{2p}\left(\varepsilon(n) + \frac{1}{d}\left\lfloor\frac{p}{2}\right\rfloor\right)\left(\varepsilon(n) + \frac{1}{d}\left\lceil\frac{p}{2}\right\rceil\right).
\]

To prove \eqref{eq: MT second case}, we first observe that
the hypothesis $s(n) < \delta/d -1$ is readily seen to be equivalent to $t(n)=t'(n)$.
Thus, the hypotheses under which 
\eqref{eq: MT second case} is claimed are equivalent to
the hypotheses of Proposition \ref{prop:piece2} \eqref{prop:second half stronger}, which together with Lemma \ref{lem:parameters} and Proposition \ref{prop:piece1} yield, upon invoking
Lemma~\ref{lemma:breaking up the kernel}, 
the equality \eqref{eq: MT second case}.
\end{proof}  

\begin{example}
When $d|(p+1)$, we have $\lambda=0$ by Corollary \ref{cor:muestimate}, so taking $t=t(n)$
in Proposition \ref{prop:piece2} \eqref{prop:piece2 part1} gives $\varepsilon=\max\{-\frac{p}{d}, s(n)+1-\delta/d\}$,
which is non-positive when $s(n)\le \delta/d - 1$, whence $C(p,d,r)=0$ in such cases.  
This recovers (and slightly improves) the result of Proposition \ref{prop:piece2} \eqref{prop:second half stronger}
in such cases.  Again assuming $d|(p+1)$, when $p$ is odd and $r=1$ we have $t(n)=d(p^{n-1}-1)/2$ and $\varepsilon=\varepsilon(n)=p(1-2/d)$
by Lemma~\ref{lem:parameters} \ref{parameters:r=1}, and it follows that $C(p,d,r)=\frac{p}{2d}(d-1)(d-2)$. 
\end{example}
  
\subsection{Proof of Proposition~\ref{prop:piece1}} \label{ss:proof1}

Fix a cutoff parameter $t$, and assume either that $t$ satisfies \ref{condition I} or 
that $d|(p-1)$ and $t$ satisfies \ref{condition II}.
We continue the notation $A_{r,t} = A_r, B_r, \ldots$ (with the fixed $t$ understood) from Section~\ref{ss: consequences of NP and HP coincidence}.

\begin{notation}
Let $\calM_n \colonequals M_n \otimes_{A} \calA$ and $\calM \colonequals M \otimes_{A} \calA$. Note that $\Fil{t}(\calM_\star) = \Fil{t}(M_\star) \otimes_{A} \calA$ and $\oFil{t}(\calM_\star)= \oFil{t}(M_\star) \otimes_{A} \calA$, where $\star$ is $n$ or nothing.
\end{notation}

\begin{lemma} \label{lemma:ar^{-1} inclusion}
With the standing hypothesis, 
\begin{equation} \label{eq:Ainverse}
\bfA_{r}^{-1}(T^{p^n}\oFil{t}(\calZ)) \subset \oFil{t}(\calM).
\end{equation}
\end{lemma}

\begin{proof}
Let $\beta(t)=t \dnom/d$ as in Corollary~\ref{cor: inverse of iterated Cartier over OmegaRinfty}.  If $t$ satisfies \ref{condition I}, then $p^n - \beta(t) \geq \lambda$. Using Corollary~\ref{cor: inverse of iterated Cartier over OmegaRinfty} and the definition of $\lambda$, we see that
\begin{equation*} 
\bfA_{r}^{-1}(T^{p^n}\oFil{t}(\calZ)) \subset T^{p^n -  \beta(t)}\oFil{t}(\calZ^{d/(p+1)}) \subset \oFil{t}(\calM).
\end{equation*}

When $d|(p-1)$ and $t$ satisfies \ref{condition II}, the argument is very similar except we consider the coefficients of $u_{t}$, and the coefficients of $u_i$ for $i<t$, occurring in an element of $\bfA_{r}^{-1}(T^{p^n}\oFil{t}(\calZ))$ separately.  
The assumption that $d|(p-1)$ implies that $t-1$ is {\em also} a cutoff parameter when $t\ge 1$ (see Remark \ref{rem:cutoffsets}),
so we may apply Corollary~\ref{cor: inverse of iterated Cartier over OmegaRinfty} with $t-1$;
when $t=0$, the conclusions of this Corollary are vacuously true for $t-1=-1$ as well. 
In any case, applying \ref{condition II}, the fact that $\mu_t$ is within $1$ of $\frac{p+1}{d} t$, and the assumption $d|(p-1)$ yields
\[
p^n - \beta(t-1) = p^n - t  \frac{r(p-1) + p+1}{d} + \frac{r(p-1) + p+1}{d} > \mu_t - t \frac{p+1}{d} +  2 >1.
\]
Then the previous strategy shows that
\[
A_{r,t-1}^{-1}(T^{p^n}\oFil{t-1}(\calZ)) \subset T^{p^n -  \beta(t-1)}\oFil{t-1}(\calZ^{d/(p+1)}) \subset \oFil{t-1}(\calM).
\]
For any element of $\bfA_{r}^{-1}(T^{p^n}\oFil{t}(\calZ))$, Corollary~\ref{cor: inverse of iterated Cartier over OmegaRinfty} (applied with $t$) shows the coefficient of $u_t$ is a multiple of $T^{p^n-\beta(t) + t\frac{p+1}{d}}$.  The exponent is greater than $\mu_t$ by our assumption
that $t$ satisfies \ref{condition II}, which completes the proof.  
\end{proof}

\begin{lemma} \label{lemma:alternatekernel}
With the standing hypotheses,
 \[
            \ker \left (\bfV^r|_{\oFil{t}(\calM_n)} \right)  = \frac{\bfA_{r}^{-1}(T^{p^n}\oFil{t}(\calZ))}{T^{p^n}\oFil{t}(\calZ)}.
\]
\end{lemma}

\begin{proof}
Using Proposition~\ref{prop: control type results} and the definition of the filtrations, extending scalars gives that 
\[
\oFil{t}(\calM_n) = \oFil{t}(\calM) / ( T^{p^n}\oFil{t}(\calZ) \cap \calM).
\]
 We claim that $T^{p^n}\oFil{t}(\calZ)$ is contained in $\oFil{t}(\calM)$.  For $1\leq i \leq t$ we know that
  \[
            \mu_i  \leq \mu_{t} \leq    p^n - tr \frac{p-1}{d} \le p^n . 
\]
  Since $\{T^{\mu_i}u_i\}_{1\leq i \leq t}$ is a basis of $\oFil{t}(\calM)$ over $\calA$ (Proposition~\ref{prop:omega as T module}), the claim follows.  
Thus,
\begin{equation} \label{eq:alternate filtration}
\oFil{t}(\calM_n)= \frac{\oFil{t}(\calM) }{ T^{p^n}\oFil{t}(\calZ)} .
\end{equation}

Note that any $x \in \oFil{t}(\calM)$ such that $\bfV^r(x) \in T^{p^n}\calZ$ satisfies $\bfA_{r}(x) \in T^{p^n}\oFil{t}(\calZ)$.  Hence there is an injection
 \begin{equation} \label{eq:isomorphism}
            \ker(\bfV^r|_{\oFil{t}(\calM_n) })  \hookrightarrow \frac{\bfA_{r}^{-1}(T^{p^n}\oFil{t}(\calZ))}{T^{p^n}\oFil{t}(\calZ)}.
  \end{equation}
  Conversely, given $x\in T^{p^n} \oFil{t}(\calZ)$, Lemma~\ref{lemma:ar^{-1} inclusion} gives $y = \bfA_{r}^{-1}(x) \in \oFil{t}(\calM)$.  Then 
        \[
            V^r(y) = \bfA_{r}(y) + \bfC_{r}(y) = x + \bfC_{r}(y).
        \]
        The second summand lies in $T^{p^n} \Fil{t}(\calZ)$ using  Corollary~\ref{cor: inverse of iterated Cartier over OmegaRinfty}.  Using \eqref{eq:alternate filtration} it follows that $V^r(y)$ is $0$ in $\calM_n$.  This shows that the morphism in \eqref{eq:isomorphism} is an isomorphism as desired.  
\end{proof}

\begin{proof}[Proof of Proposition~\ref{prop:piece1}]
Using the fact that $A\to \calA$ is flat, it follows from \eqref{eq: dimension compare with d-th root of T}
that one has
$$d \cdot \dim_k\left (\ker \left(\bfV^r|_{\oFil{t}(M_n)}\right)\right) = \dim_k \left (\ker \left(\bfV^r|_{\oFil{t}(\calM_n)}\right)\right).$$ 
On the other hand, Lemma \ref{l:fundamental fact SNF} gives
\[
d\cdot v_{T}(\Det_{\calE_\F}(A_r)) = \dim_\F \frac{\bfA_{r}^{-1}(T^{p^n}\oFil{t}(\calZ))}{T^{p^n}\oFil{t}(\calZ)}
=\nu \dim_k \frac{\bfA_{r}^{-1}(T^{p^n}\oFil{t}(\calZ))}{T^{p^n}\oFil{t}(\calZ)}.
\]
Combining
Lemma~\ref{lemma:alternatekernel} with Proposition \ref{prop: columns and determinant of Ak} and
Lemma \ref{lemma: endpoint of polygon for iterates} then completes the proof. 
\end{proof}

\subsection{Proof of Proposition~\ref{prop:piece2} \eqref{prop:piece2 part1}} \label{ss:proof2}
Fix a cutoff parameter $t$ satisfying \ref{condition I}, and let $D=D_t$
be as in Definition \ref{defn:parameter D}. 
Define
\begin{equation}
W(n,t) \colonequals \text{Im} \Big(T^{\lambda+D-1+\frac{tp}{d}} \calZ^d \cap \Fil{t}(\calM)\to \Fil{t}(\calM_n) \Big) , 
\end{equation}
where the map is induced by the canonical map $\calM \to \calM_n$. The space
$W(n,t)$ plays the role of the top right region in Figure \ref{fig:picture 2}, i.e., the region above the line $y=13$ and to the right of the line $x=6$. In Figure \ref{fig:picture 2} we saw that \emph{every} lattice point in the top right region is in the kernel of $V^3$. Accordingly, we will show that $W(n,t)$ is a subspace of $\prt{t}(\ker(\bfV^r|_{\calM_n}))$ and is also almost all of $\Fil{t}(\calM_n)$.

By the definition of $\lambda$ we know
        \begin{equation} \label{lemma: error for the T_k part eq1}
            T^{\lambda}\Fil{t}(\calZ^{\frac{d}{p+1}}) \subset \Fil{t}(\calM) \subset T^{-\lambda}\Fil{t}(\calZ^{\frac{d}{p+1}}).
        \end{equation}

    \begin{lemma}\label{lemma: estimating the T_k part}
With the standing hypotheses,  
        \begin{align*}
            W(n,t) \subset \prt{t}(\ker(\bfV^r|_{\calM_n})).
        \end{align*}
    \end{lemma}
    
    \begin{proof}
        Let $x \in T^{\lambda+D-1+t \frac{p}{d}}\calZ^d \cap \Fil{t}(\calM)$. By Lemma \ref{lemma: filtered part for alterations lies in L^{d/p+1}} there exists $y \in T^{\lambda+D-1}\oFil{t}\calZ^{\frac{d}{p+1}}$ with
        \[\bfV^r(x+y) \in T^{\lambda +D-1+t \frac{p}{d} + \frac{(t+1)(r-1)(p-1)}{d}} \Fil{t}(\calZ^{d/p}) =  T^{\lambda +D-1+ t\frac{r (p-1)}{d} + \frac{t}{d} +\frac{(r-1)(p-1)}{d}} \Fil{t}(\calZ^{d/p}). \]
        By \eqref{lemma: error for the T_k part eq1} and the fact that $D\ge 1$ by Definition \ref{defn:parameter D},
        we have $T^{\lambda+D-1}\oFil{t}\calZ^{\frac{d}{p+1}} \subset \oFil{t}(\calM)$,
        so $y \in \calM$ and therefore $x+y\in \calM$. 
        As $\Fil{t}(\calZ^{d/p}) \subset T^{(t+1) p/d}\Fil{t}(\calZ)$, 
        it again follows from the definition of $D=D_t$ that
 \[
 \bfV^r(x+y) \in T^{\lambda+D-1 + t\frac{r(p-1)}{d} + \frac{t}{d} + (t+1)\frac{p}{d} + \frac{(r-1)(p-1)}{d} }\Fil{t}(\calZ) \subset T^{p^n} \Fil{t}(\calZ) .
 \]
In particular, the image of $x+y$ in $\calM_n$ is contained in 
$\ker(\bfV^r|_{\calM_n})$. As $\prt{t}(x+y)=\prt{t}(x)=x$, the result follows. 
    \end{proof}

    From Lemma \ref{lemma: estimating the T_k part} we deduce
    \begin{equation}\label{eq: bound the cokernel of V on right side}
        \dim_k \left( \frac{\oFil{t}(\calM_n)}{\prt{t}(\ker(\bfV^r|_{\calM_n}))} \right) \le \dim_k \left( \frac{\oFil{t}(\calM_n)}{W(n,t)} \right).
    \end{equation}
    We would like to explicitly bound the right side of this inequality. This is difficult to do directly, as the $\mu_i$'s behave erratically (as evidenced in Proposition \ref{prop:muestimates}). Our strategy is to map $\frac{\oFil{t}(\calM_n)}{W(n,t)}$ injectively into a larger space whose dimension we can easily calculate. This motivates the following definition.

\begin{definition}
Define $k$-vector spaces (depending on the parameters $\lambda, t, D$) by
\[
        X'(n) \colonequals \frac{\Fil{t}(\calM)}{T^{t\frac{p}{d} + \lambda+D-1}\calZ^d \cap \Fil{t}(\calM)} \quad \text{and} \quad 
            X(n) \colonequals \frac{T^{-\lambda}\Fil{t}(\calZ^{\frac{d}{p+1}})}{T^{t \frac{p}{d} + \lambda+D-1}\calZ^d \cap T^{-\lambda}\Fil{t}(\calZ^{\frac{d}{p+1}})}.
\]
\end{definition}

\begin{lemma} \label{lemma:x'x}
With the standing hypotheses, 
\[  \dim_k \left( \frac{\oFil{t}(\calM_n)}{W(n,t)} \right)\leq \dim_k X'(n) \leq \dim_k X(n).
\]
\end{lemma}

\begin{proof}
The inclusion $\Fil{t}(\calM) \subset T^{-\lambda}\Fil{t}(\calZ^{\frac{d}{p+1}})$ induces a map $\imath: X'(n) \to X(n)$.  It is injective as
\[
T^{t \frac{p}{d} + \lambda + D-1}\calZ^d \cap \Fil{t}(\calM) = \left(T^{t\frac{p}{d} + \lambda + D-1 }\calZ^d \cap T^{-\lambda}\Fil{t}(\calZ^{\frac{d}{p+1}})\right) \cap \Fil{t}(\calM).\]
This gives the second inequality. For the first inequality, note that there is a surjective map 
$$ X'(n) \twoheadrightarrow \frac{\oFil{t}(\calM_n)}{W(n,t)}, $$
which is induced by the map $\oFil{t}(\calM) \to \oFil{t}(\calM_n)$. 
\end{proof}

\begin{lemma} \label{lemma:rightlattice}
We have that $\dim_k \Fil{t}(M_n) = \# \Delta_n(t)$ and hence $\dim_k \Fil{t}(\calM_n) = d  \# \Delta_n(t)$.
\end{lemma}

\begin{proof}
A $k$-basis for $\Fil{t}(M_n)$ is given by
\[
\{ T^{j} u_i : i > t, \mu_i \leq j \leq p^n-1\}.
\] 
Remembering Definition~\ref{defn:muj}, we observe a natural bijection between this basis and $\Delta_n(t)$.
\end{proof}

\begin{proof}[Proof of Proposition~\ref{prop:piece2} \eqref{prop:piece2 part1}]

Let $t$ be a cutoff parameter satisfying \ref{condition I}, set $D=D_t$ and 
for $i> t $ define 
\[
    c_i \colonequals \max \left\{ \frac{tp}{d} + \lambda +D-1 + \frac{i}{d}, \frac{i(p+1)}{d} - \lambda \right\}.
  \]
We know that $\{T^{\frac{i(p+1)}{d} - \lambda} u_i \}_{i > t}$
        is an $\calA$-basis of $T^{-\lambda}\Fil{t}(\calZ^{\frac{d}{p+1}})$ and it is easy to see that $\{ T^{c_i} u_i\}_{i >t}$ is an $\calA$-basis of $T^{t\frac{p}{d} + \lambda+D-1}\calZ^d \cap T^{-\lambda}\Fil{t}(\calZ^{\frac{d}{p+1}})$.
It follows that
        \begin{equation}\label{eq: upper bound sum}
            \dim_k(X(n)) = d \sum_{i>t} \left( c_i -\frac{i(p+1)}{d}+ \lambda\right). 
        \end{equation}
Setting $\varepsilon= D-1+2\lambda - p/d$ as in the statement of Proposition \ref{prop:piece2} \eqref{prop:piece2 part1},
and writing $i = t + 1+ j$ with $j\ge 0$, a quick argument shows
        \begin{align*}
            c_i &= \begin{cases}
                \frac{t p}{d} + \lambda + D-1 + \frac{i}{d}, & j \leq \left\lfloor \frac{d\varepsilon}{p}\right\rfloor \\[10pt]
                \frac{i(p+1)}{d} - \lambda, & j > \left\lfloor \frac{d\varepsilon}{p}\right\rfloor
            \end{cases}.
        \end{align*}
The right side of \eqref{eq: upper bound sum} is therefore $0$ when $\varepsilon < 0$, and when $\varepsilon \ge 0$ is equal to
        \begin{align}\label{eq: dimXRHS}
      d\sum_{j=0}^{\left\lfloor \frac{d\varepsilon}{p}\right\rfloor} \left(\frac{tp}{d} + \lambda + D-1 + \frac{t + 1+j}{d} -\frac{(t +1+ j)(p+1)}{d} + \lambda \right) & = d C(p,d,r),
        \end{align}
        which is readily seen to be zero when $\varepsilon=0$ as well.
Combining \eqref{eq: bound the cokernel of V on right side} with Lemmas \ref{lemma:x'x} and \ref{lemma:rightlattice} gives the inequalities
\begin{equation}\label{eq: semifinal bound}
    0 \le d\Delta_n(t) - \dim_k \prt{t}(\ker(\bfV^r|_{\calM_n})) \le \dim_k X'(n) \le \dim_k X(n) 
\end{equation}
As in the proofs of Lemma \ref{lemma:breaking up the kernel} and Proposition \ref{prop:piece1},
we have 
\begin{equation}
    d \dim_k \prt{t}(\ker(\bfV^r|_{M_n}))=\dim_k \prt{t}(\ker(\bfV^r|_{\calM_n})).
    \label{eq:projection divide by d}
\end{equation}
Combining this with \eqref{eq: semifinal bound}, \eqref{eq: upper bound sum}--\eqref{eq: dimXRHS} 
and dividing through by $d$ yields the claimed inequalities in Proposition \ref{prop:piece2} \eqref{prop:piece2 part1}.
For the final statement of \eqref{prop:piece2 part1}, note that when $t=t(n)$,
we have
\[  
    \varepsilon  = D_{t}-1+2\lambda-p/d =\max\{2\lambda - p/d, \varepsilon(n)\}  
    \le \delta\left(1-\frac{1}{d}\right)+\lambda - \left(p-2\left\lfloor \frac{p}{2}\right\rfloor\right)
\]
by Definition \ref{defn:t and point count} and Lemma \ref{lem:parameters} \ref{parameters:general},
and the right hand side depends only on $p,d$, and $r$.
It follows from the definition of $C(p,d,r)$ that this constant is likewise 
bounded independently of $n$.
\end{proof}

\subsection{Proof of Proposition~\ref{prop:piece2} \eqref{prop:second half stronger}} \label{ss:proof2b}

We now prove the stronger result
of Proposition \ref{prop:piece2} (\ref{prop:second half stronger}).

\begin{lemma} \label{lemma:projectionlemma}
Let $t=t'(n)$ and
assume either that that $d|(p-1)$, or that $d\le p+1$ and $t(n)=t$.  Then 
\[\Fil{t}(\calM_n)=\prt{t}(\ker(V^r|_{\calM_n})).\]
\end{lemma}

\begin{proof}
First note that the hypotheses guarantee that $t$ is a cutoff parameter,
thanks to Lemma \ref{lem:parameters}.
It is immediate that $\prt{t}(\ker(V^r|_{\calM_n})) \subset \Fil{t}(\calM_n)$.  
For $i > t$, we will show that there exists 
    $y \in \oFil{t}(\calM)$ such that $V^r(T^{\mu_i}u_i + y) \in T^{p^n} \calZ$. As the $T^{\mu_i}u_i$ with $i>t $ and $\mu_i < p^n$ generate $\Fil{t}(\calM_n)$ as an $\calA$-module, this will prove the lemma. 
    
    We have
    \[T^{\mu_i}u_i=T^{(\mu_i - \frac{p+1}{d}i) + \frac{p+1}{d}i}u_i \in T^{(\mu_i - \frac{p+1}{d}i) + \frac{p}{d}i}\calZ^d=T^{(\mu_i - \frac{p+1}{d}i + (i-t)\frac{p}{d}) + \frac{p}{d}t}\calZ^d. \]
    Then by Lemma \ref{lemma: filtered part for alterations lies in L^{d/p+1}} there exists $y \in T^{\mu_i - \frac{p+1}{d}i + (i-t)\frac{p}{d}}\oFil{ t} \calZ^{\frac{d}{p+1}}$ such that
    \begin{align}
    \begin{split}\label{eq: right side d|p-1}
        V^r(T^{\mu_i}u_i+y) &\in T^{\alpha_i}\Fil{t}\calZ^{d/p}, \text{ where}\\
        \alpha_i &= \left (\mu_i - \frac{p+1}{d}i + (i-t)\frac{p}{d}\right) + t\frac{p}{d} + \frac{(t+1)(r-1)(p-1)}{d}.
        \end{split}
    \end{align}
    Moreover, since $T^{\mu_i}u_i \in Z$ we also know from Corollary \ref{c: filtered part for alteration}
    that $y \in Z$.  
     By Corollary \ref{cor:muestimate}, we have $\mu_i-\frac{p+1}{d}i \ge -1+1/d$. Since $i>t$, for all $j\ge 0$ we have
     \[
            \left\lceil \mu_i - \frac{p+1}{d}i+(i-t)\frac{p}{d} +\frac{p+1}{d}j \right\rceil \ge \left\lceil -1+\frac{1}{d}+\frac{p}{d}+\frac{p+1}{d}j\right\rceil
            \ge \left\lceil\frac{p+1}{d}j \right\rceil\ge \mu_j
     \]
     with the penultimate inequality resulting from the fact that $d\le p+1$ in any case,
     and the final inequality following from Proposition \ref{prop:muestimates}.
     We conclude that $y \in  \oFil{t}M$.

    We will now prove $V^r(T^{\mu_i}u_i + y)$ is contained in $T^{p^n} Z \subset T^{p^n}\calZ$. 
    From \eqref{eq: right side d|p-1} we obtain
    \begin{align*}
        V^r(T^{\mu_i}u_i+y) &\in T^{\alpha_i + \frac{p}{d}(t+1)}\calZ \cap Z = T^{\left\lceil \alpha_i + \frac{p}{d}(t+1)\right \rceil }Z.
    \end{align*}
    A straightforward computation then gives
    \begin{align*}
        \alpha_i + \frac{p}{d}(t+1) &=\mu_i - \frac{p+1}{d}i + (i-t-1)\frac{p}{d} + (t + 1) \frac{\dnom}{d} \\
        &> (t + 1) \frac{\dnom}{d} -1\geq p^n-1
    \end{align*}
    wherein the final inequality is Lemma~\ref{lem:parameters} \ref{parameters:d|p-1}.
    This yields $\left\lceil \alpha_i + \frac{p}{d}(t+1)\right \rceil \geq p^n$, giving the desired inclusion.
\end{proof}

  \begin{proof}[Proof of Proposition~\ref{prop:piece2} \eqref{prop:second half stronger}] 
  Combine \eqref{eq:projection divide by d} and Lemma~\ref{lemma:rightlattice} with Lemma \ref{lemma:projectionlemma}.
  \end{proof}
  
\section{Formulae and estimates of higher \texorpdfstring{$a$}{a}-numbers}
\label{s: formulae and estimates}

As usual, 
 let $\{X_n\}$ be a \fullystable{} $\Z_p$-tower totally ramified over $\infty \in \PP^1$ and unramified elsewhere, with $s_n=d p^{n-1}$ the $n$th break in the upper numbering for the ramification filtration above infinity. 

\subsection{Estimating Lattice Point Counts}
In this subsection, we prove Corollary~\ref{cor:asymptoticmain}: in the over-simplified picture of Figure~\ref{fig:picture}, the idea is to approximate the number of lattice points above the diagonal line as the area of the triangle.  Crude estimates suffice for the asymptotic statement.

We proved in Theorem~\ref{theorem:higheranumber} that
\begin{equation} \label{eq:boundeddifference}
 \frac{r(p-1) t(n)( t(n)+1)}{2d} + \# \Delta_n(t(n))  -a_n^{(r)} 
\end{equation}
is non-negative and bounded above independently of $n$, where $t(n) \colonequals d \left \lfloor \frac{p^n-1}{(r+1)p - (r-1)} \right \rfloor$.  
Elementary estimates show that
\begin{equation} \label{eq:leftestimate}
 \frac{r(p-1) t(n)( t(n)+1)}{2d}  = \frac{r(p-1) d}{2(r (p-1) + (p+1))^2 } p^{2n} + O(p^n)
\end{equation}
where the implied constant depends only on $p$, $d$, and $r$. 

\begin{lemma} \label{lem:rightestimate}
We have $\displaystyle \# \Delta_n(t(n)) = \left( \frac{r (p-1)}{ r(p-1) + (p+1)} \right)^2 \frac{d p^{2n}}{2(p+1)} + O(p^n).$
\end{lemma}

\begin{proof}
Using Proposition~\ref{prop:muestimates} and Definition~\ref{defn:muj}, we see that the area of the triangle bounded by the lines $y = p^n$, $x=t(n)$, and $y = \frac{p+1}{d} x $ is an approximation to $\# \Delta(t(n))$, with error (depending on the perimeter) which is $O(p^n)$.  Hence $\# \Delta(t(n))$ is approximately
\[
\frac{1}{2} \left( p^n - \frac{(p+1) p^n}{(r+1)p - (r-1)} \right) \left( \frac{p^n d}{p+1} - \frac{d p^n}{(r+1)p - (r-1)} \right) = \left( \frac{r (p-1)}{ r(p-1) + (p+1)} \right)^2 \frac{d p^{2n}}{2(p+1)}.
\]
with error $O(p^n)$.  
\end{proof}

\begin{proof}[Proof of Corollary~\ref{cor:asymptoticmain}]
Combine Lemma~\ref{lem:rightestimate} with \eqref{eq:leftestimate} and \eqref{eq:boundeddifference}.
\end{proof}

\subsection{A formula for the \texorpdfstring{$a$}{a}-number}  We now work out an explicit formula for the $a$-number when $d | p-1$, establishing Corollary~\ref{MC2:formula}.  The argument, while a bit involved, uses only elementary techniques so we leave many of the details to the reader.  The basic idea is to count the number of elements in $\Delta(t'(n))$ with a given second coordinate, and then sum.  In terms of the picture of Figure~\ref{fig:picture}, this corresponds to counting the number of elements in each row that are to the right of the vertical line.  Recall that $\xi_\star$ %
is defined in Definition~\ref{not:basep}.

\begin{definition}
For a positive integer $n$, define $\displaystyle f_n(x) \colonequals \sum_{b=0}^{p^n-1} \lfloor x + \xi_b \rfloor$.
\end{definition}

We will establish an explicit formula for $f_n(x)$ by induction on $n$.

\begin{lemma} \label{lemma:f1(x)}
We have $\displaystyle f_1(x) = \frac{(p-1)(d-1)}{2} + \lfloor px \rfloor$.
\end{lemma}

\begin{proof}
Let $e = (p-1)/d$ and group the sum into segments of length $e$ to see that
\begin{align*}
f_1(x) &= \lfloor x \rfloor + \sum_{i=0}^{d-1} \sum_{j=1}^e \left \lfloor x + \frac{i(p-1)}{p} + \frac{j d}{p} \right \rfloor = \lfloor x \rfloor + \sum_{i=0}^{d-1} \sum_{j=1}^e \left ( i + \left \lfloor x + \frac{jd-i}{p} \right \rfloor \right )\\
& = e \frac{d(d-1)}{2} + \sum_{j=0}^{p-1} \left \lfloor x + \frac{j}{p} \right \rfloor = \frac{(p-1)(d-1)}{2} + \lfloor px \rfloor
\end{align*}
where the last step uses Hermite's identity.
\end{proof}

\begin{lemma} \label{lemma:fn(x)}
For $n>1$, we have
\[
f_n(x) = \sum_{i=0}^{p-1} f_{n-1}\left( x + i d \frac{p^{n-1} + p^{-n}}{p+1} \right).
\]
\end{lemma}

\begin{proof}
Grouping the sum based on the first base-$p$ digit of $b$, we compute
\[
f_n(x) = \sum_{b=0}^{p^n-1} \lfloor x + \xi_b \rfloor = \sum_{i=0}^{p-1} \sum_{j=0}^{p^{n-1}-1} \left \lfloor x + i d_n p^{-n}+ \xi_j \right \rfloor = \sum_{i=0}^{p-1} f_{n-1}( x + i d_n p^{-n}).
\]
Then use the formula for $d_n$ in Proposition~\ref{prop:genusbreaks}.
\end{proof}

\begin{proposition} \label{prop:fn(x)explicit}
We have $\displaystyle f_n(x) = \frac{d p (p^n-1)(p^{n-1}-1)}{2 (p+1)} + \frac{(d-1)(p^n-1)}{2} + \lfloor p^n x \rfloor$.
\end{proposition}

\begin{proof}
We prove this by induction, with the base case $n=1$ being Lemma~\ref{lemma:f1(x)}.  For the inductive step, suppose the result holds true for $n$.  Using Lemma~\ref{lemma:fn(x)} and the inductive hypothesis gives that
\begin{align*}
f_{n+1}(x) &= \sum_{i=0}^{p-1} f_n\left(x + i d \frac{p^{n}+p^{-n-1}}{p+1}\right) \\
&= p \frac{dp(p^n-1)(p^{n-1}-1)}{2 (p+1)} + p \frac{(d-1)(p^n-1)}{2} + \sum_{i=0}^{p-1} \left \lfloor p^n x + i d \frac{p^{2n}+p^{-1}}{p+1} \right \rfloor.
\end{align*}
Now notice that $\displaystyle\frac{p^{2n}+p^{-1}}{p+1} = \frac{p^{2n+1} -p +p +1}{p(p+1)} =  p^{-1} + \frac{p^{2n} -1}{p+1}$.  Thus
\[
 \sum_{i=0}^{p-1} \left \lfloor p^n x + i d \frac{p^{2n}+p^{-1}}{p+1} \right \rfloor = \frac{(p-1)p d}{2} \frac{p^{2n}-1}{p+1} + \sum_{i=0}^{p-1} \left \lfloor p^n x + \frac{id}{p} \right \rfloor.
\]
Recognizing the last sum as $f_1(p^nx)$ and using Lemma~\ref{lemma:f1(x)}, we conclude 
\[
f_{n+1}(x) = \frac{d p (p^{n+1}-1)(p^{n}-1)}{2 (p+1)} + \frac{(d-1)(p^{n+1}-1)}{2}+ \lfloor p^{n+1} x \rfloor. \qedhere
\]
\end{proof}

\begin{proof}[Proof of Corollary~\ref{MC2:formula}]
Recall that $p$ is odd and $d|p-1$.  Theorem~\ref{theorem:higheranumber} gives that
\[
a(X_n) =  \frac{r(p-1) t'(n)(t'(n)+1)}{2d} + \# \Delta_n(t'(n)).
\]
When $r=1$, we see that $t'(n) = d p^{n-1}/2$ when $d$ is even and $t'(n) = (d p^{n-1}-1)/2$ when $d$ is odd.  For conciseness, we will only treat the case that $d$ is even: the other case is similar.

Note that $\xi_{(p+1)/2 \cdot p^{n-1}} = d (p^{n-1} + p^{-n})/2$, so $\xi_{a} \geq d p^{n-1}/2 +1$ only when $a > (p+1)p^{n-1} /2$ since $\xi_a$ is an increasing function of $a$.  Thus recalling Definition~\ref{defn:muj}, when $b \geq (p+1) p^{n-1}/2$ the number of points in $\Delta(t'(n))$ whose second coordinate is $b$ is given by $\lfloor \xi_b - d p^{n-1}/2\rfloor$.  It follows from Lemma~\ref{lem:xiprops}\ref{xifor} that
 if $b + ip^{n-1} < p^n$ then 
\[
\xi_{b + i p^{n-1}} = \xi_b +  \frac{d i }{p+1} (p^{n-1} + p^{-n}).
\]
Thus we see that for $i <p$
\[
\sum_{b= i p^{n-1}}^{(i+1)p^{n-1} -1} \lfloor \xi_{b} - d p^{n-1}/2 \rfloor = \sum_{j=0}^{p^{n-1}-1} \left \lfloor \xi_{j}  + \frac{di}{p+1} (p^{n-1} + p^{-n})  - \frac{dp^{n-1}}{2} \right \rfloor = f_{n-1}\left(\frac{di}{p+1} (p^{n-1} + p^{-n}) - \frac{dp^{n-1}}{2} \right).
\]
We conclude that 
\begin{equation}
\# \Delta(t'(n)) = \sum_{i=(p+1)/2}^{p-1} f_{n-1}\left(\frac{di}{p+1} (p^{n-1} + p^{-n}) - \frac{dp^{n-1}}{2} \right).
\end{equation}
But this can be evaluated explicitly using Proposition~\ref{prop:fn(x)explicit}: a lengthy but elementary computation to simplify gives the result when $d$ is even.  The odd case is similar.
\end{proof}

\begin{example}
When $p=5$ and $d=4$, Corollary~\ref{MC2:formula} says that the $a$-number of the $n$th curve in a break-minimal $\Z_p$-tower is $\frac{2}{3} (5^{2n-1} + 1)$ which for $n=1,2,3,4$ equals $4, 84, 2084, 52084$ and matches the computations of \cite[Example 4.8]{boohercais21}.
\end{example}

\appendix

\section{Index of notation}

\subsection{Basic Notation}
\begin{itemize}[label={}]
    \item $\F\colonequals\FF$ the prime field of characteristic $p$
    \item $k$ a finite extension of $\F$
    \item $\nu\colonequals[k:\F]$ the degree of $k/\F$
    \item $K \colonequals k(x)$ the rational function field of $\PP_k^1$
    \item $A\colonequals k\llbracket T \rrbracket$ and $A_\F \colonequals \F\llbracket T \rrbracket$
    \item $E \colonequals \Frac(A)$ and $E_\F \colonequals \Frac(A_F)$
    \item $\mathcal{A}\colonequals k\llbracket T^{1/d} \rrbracket$ and $\mathcal{A}_\F \colonequals \F\llbracket T^{1/d} \rrbracket$ (except in Section \ref{ss:recollectionsbanach} where $\calA$ is an arbitrary complete valuation ring).
    \item $\mathcal{E} \colonequals \Frac(\mathcal{A})$ and $\mathcal{E}_\F \colonequals \Frac(\mathcal{E}_F)$
\end{itemize}
As general practice, modules over $A$ and $A_\F$ will be denoted using Roman letters, while modules over $\mathcal{A}$ and $\mathcal{A}_\F$ will be denoted using calligraphic letters.

\subsection{Notation Pertaining to \texorpdfstring{$\Z_p$}{Zp}-Towers}
\begin{itemize}[label={}]
    \item $L/K$ a geometric $\Z_p$-extension totally ramified at $S = \{\infty\}$ and unramified elsewhere.
    \item $\gamma$ a topological generator of $\Gal(L/K)$ providing an isomorphism $\Gal(L/K) \cong \Z_p$.
    \item $\chi:\Gal(L/K) \to A_\F^\times$ the character corresponding to the choice of $\gamma$, see \eqref{eq:galoischar}.
    \item $d$ the ramification invariant of a $\Z_p$-tower with minimal break ratios, see Definition \ref{d:minimal-break-ratios}.
\end{itemize}
A subscript of $n$ will often denote objects pertaining to the $n$th level in the $\Z_p$-tower $L/K$:
\begin{itemize}[label={}]
    \item $K_n$ the unique sub-extension of $L/K$ of degree $p^n$ over $K$.
    \item $X_n$ the smooth, projective, geometrically connected curve over $k$ with function field $K_n$.
    \item $\pi$ the covering $\pi:X_{n+1} \to X_n$ corresponding to $K_{n+1}/K_n$. Note we have suppressed the dependence on $n$.
    \item $\pi_*$ the trace map $\pi_*\calO_{X_{n+1}} \to \calO_{X_n}$ induced by $\pi$, see \S \ref{ss:iwasawa-modules}.
    \item $P_n$ the unique point of $X_n$ lying above $\infty \in \PP_k^1$
    \item $d_n$ the unique break in the ramification filtration of the cover $X_n \to X_{n-1}$ at $P_n$
    \item $s_n$ the $n$th break in the upper-numbering ramification filtration of $L/K$ over $\infty$.
    \item $A_n \colonequals A/(T^{p^n})$
    \item $R_n \colonequals \calO_{X_n}(X_n - P_n)$ the affine coordinate ring of $X_n - P_n$
    \item $M_n \colonequals H^0(X_n,\Omega_{X_n/k}^1)$ the space of regular differentials on $X_n$
    \item $\mathcal{M}_n \colonequals M_n \otimes_A \mathcal{A}$
    \item $Z_n \colonequals H^0(X_n-P_n,\Omega_{X_n/k}^1)$ the space of differentials on $X_n$ regular away from $P_n$
    \item $\mathcal{Z}_n \colonequals Z_n \otimes_A \mathcal{A}$
\end{itemize}
For objects at the ``infinite level'' we use the same letters without the subscript:
\begin{itemize}[label={}]
    \item $A\langle x \rangle \colonequals \varprojlim_n A_n[x]$ the Tate algebra
    \item $R \colonequals \varprojlim_n R_n$
    \item $M \colonequals \varprojlim_n M_n$
    \item $\mathcal{M} \colonequals M \otimes_A \mathcal{A}$
    \item $Z \colonequals \varprojlim_n Z_n$
    \item $\mathcal{Z} \colonequals Z \otimes_A \mathcal{A}$
    \item $\calZ_{\calE}\colonequals \calZ \otimes_{\calA} \calE$
    \item $L \colonequals x A\langle x \rangle$ the ideal of functions vanishing at $x = 0$.
    \item $\mathcal{L} \colonequals x \mathcal{A}\langle x \rangle$
\end{itemize}

\subsection{Growth conditions} We use a superscript $m$ to denote functions and differentials ``convergent with values bounded by $1$ on an open disk of radius $-1/m$:''
\begin{itemize}[label={}]
    \item $\displaystyle A\langle x\rangle^m \colonequals  \left\{ \sum_{i\ge 0} c_i x^i \in A\langle x\rangle \ :\ v_T(c_i) \ge \frac{i}{m}\right\}$
    \item $\displaystyle L^m \colonequals  \left\{ \sum_{i\ge 1} c_i x^i \in L\ :\ v_T(c_i) \ge \frac{i}{m}\right\}$
    \item $\displaystyle \mathcal{L}^m \colonequals  \left\{ \sum_{i\ge 1} c_i x^i \in \mathcal{L}\ :\ v_T(c_i) \ge \frac{i}{m}\right\}$
\end{itemize}
In Definition \ref{def:Bbasis} we define elements $u_1,u_2,\dots \in Z$ that form an orthonormal basis of $\calZ_{\calE}$.
\begin{itemize}[label={}]
    \item $\displaystyle Z^m \colonequals  \left\{ \sum_{i\ge 1} c_i u_i \in Z\ :\ v_T(c_i) \ge \frac{i}{m}\right\}$
    \item $\displaystyle \mathcal{Z}^m \colonequals   \left\{ \sum_{i\ge 1} c_i u_i \in \mathcal{Z}\ :\ v_T(c_i) \ge \frac{i}{m}\right\}$
\end{itemize}

\bibliographystyle{amsalpha}
\bibliography{trace}

\end{document}